\documentclass[reqno, 11pt, a4paper]{amsart}
\usepackage{amsmath, amssymb, enumerate, esint, tikz}
\usepackage{bbm}

\usepackage{xcolor}

\usepackage{bm}

\usepackage[top=3cm, bottom=3cm, left=2.8cm, right=2.8cm]{geometry}

\usepackage{multirow}
\usepackage{longtable}
\usepackage{booktabs}
\usepackage{tablefootnote}

\newcounter{newcounter}

\usepackage[mathscr]{eucal}

\allowdisplaybreaks
\usepackage{comment}

\newcommand\Tstrut{\rule{0pt}{2.7ex}}         

\usepackage[colorlinks=true,linkcolor=blue,citecolor=red]{hyperref}

\usepackage[nameinlink,capitalize]{cleveref}
\crefname{equation}{}{}
\crefname{subsection}{subsection}{Subsections}
\crefname{theo}{Theorem}{Theorems}
\crefname{coro}{Corollary}{Corollaries}
\crefname{prop}{Proposition}{Propositions}
\crefname{assum}{Assumption}{Assumptions}

\makeatletter
\newcommand{\refcheckize}[1]{%
	\expandafter\let\csname @@\string#1\endcsname#1%
	\expandafter\DeclareRobustCommand\csname relax\string#1\endcsname[1]{%
		\csname @@\string#1\endcsname{##1}\wrtusdrf{##1}}%
	\expandafter\let\expandafter#1\csname relax\string#1\endcsname
}
\makeatother

\newcommand{\overbar}[1]{\mkern 1.5mu\overline{\mkern-1.5mu#1\mkern-1.5mu}\mkern 1.5mu}

\newtheorem{theo}{Theorem}[section]
\newtheorem{coro}[theo]{Corollary}
\newtheorem{lemm}[theo]{Lemma}
\newtheorem{prop}[theo]{Proposition}

\theoremstyle{definition}
\newtheorem{defi}[theo]{Definition}
\newtheorem{exam}[theo]{Example}
\newtheorem{rema}[theo]{Remark}
\newtheorem{assum}[theo]{Assumption}
\newtheorem{conven}[theo]{Convention}

\numberwithin{equation}{section}

\newcommand{\R}{\mathbb R}
\newcommand{\bN}{\mathbb N}
\newcommand{\bQ}{\mathbb Q}

\newcommand{\bD}{\mathbb D}
\newcommand{\E}{\mathbb E}
\newcommand{\p}{\mathbb P}
\newcommand{\bF}{\mathbb F}
\newcommand{\1}{\mathbbm{1}}
\newcommand{\m}{\mathbbm{m}}
\newcommand{\bI}{\mathbb I}

\newcommand{\cB}{\mathcal B}

\newcommand{\cN}{\mathcal N}
\newcommand{\cE}{\mathcal E}
\newcommand{\G}{\mathcal G}
\newcommand{\F}{\mathcal F}
\newcommand{\cS}{\mathcal S}
\newcommand{\cT}{\mathcal T}
\newcommand{\cSM}{\mathcal{SM}}
\newcommand{\RH}{\mathcal{RH}}
\newcommand{\cR}{\mathcal{R}}

\newcommand{\scrD}{\mathscr D}
\newcommand{\sS}{\mathscr S}
\newcommand{\sUS}{\mathscr{US}}

\newcommand{\BMO}{\mathrm{BMO}}
\newcommand{\bmo}{\mathrm{bmo}}

\newcommand{\rc}{\mathrm{c}}

\newcommand{\corr}{\mathrm{corr}}

\newcommand{\sign}{\mathrm{sgn}}

\newcommand{\re}{\mathrm{Re}}

\newcommand{\Lip}{\mathrm{Lip}}

\newcommand{\riem}{\mathrm{Rm}}
\newcommand{\im}{\mathrm{i}}
\newcommand{\od}{\mathrm{d}}
\newcommand{\adm}{\mathcal A(S)}

\newcommand{\CL}{\mathrm{CL}}

\newcommand{\wt}{\widetilde}

\newcommand{\ol}{\overbar}

\newcommand{\lee}{\leqslant}
\newcommand{\gee}{\geqslant}

\newcommand{\ep}{\varepsilon}
\newcommand{\e}{\mathrm{e}}
\newcommand{\Leb}{\lambda}
\newcommand{\pd}{\partial}

\newcommand{\QV}{\<\vartheta, \tau\>}
\newcommand{\ts}{\textstyle}
\def\({\left(}
\def\){\right)}
\def\[{\hspace{-.05cm}\left[}
\def\]{\right]}

\def\<{\langle}
\def\>{\rangle}

\newcommand{\ce}[2]{\E^{#1}\hspace{-.05cm}\left [ #2 \right ]}
\newcommand{\bce}[2]{\E^{#1} \hspace{-.05cm}\big[ #2 \big]}

\newcommand{\bbce}[2]{\E^{#1} \hspace{-.05cm}\bigg[ #2 \bigg]}

\newcommand{\comb}[2]{#1 \sqcup #2 }

\begin{document}
	
	\title[Approximation of stochastic integrals with jumps via weighted BMO]{Approximation of stochastic integrals with jumps via weighted BMO approach}

	\author{Nguyen Tran Thuan}
	\address{$^{1}$Department of Mathematics, Saarland University,  Postfach 15~11~50, 66041 Saarbr\"ucken, Germany}
	\email{nguyen@math.uni-sb.de}
	
	\address{$^{2}$Department of Mathematics, Vinh University, 182 Le Duan, Vinh, Nghe An, Vietnam}
	
	\thanks{The author was supported by the Project 298641 ``\textit{Stochastic Analysis and Nonlinear Partial Differential Equations, Interactions and Applications}'' of the Academy of Finland, and by Vietnam’s National Foundation for Science and Technology	Development (NAFOSTED) under grant number 101.03-2020.18.}
	
	\date{\today}
	
	
	
	\maketitle

	\begin{abstract} This article investigates discrete-time approximations of stochastic integrals driven by semimartingales with jumps via weighted bounded mean oscillation (BMO) approach. This approach enables $L_p$-estimates, $p \in (2, \infty)$, for the approximation error depending on the weight, and it allows a change of the underlying measure which leaves the error estimates unchanged. To take advantage of this approach, we propose a new approximation scheme obtained from a correction for the Riemann approximation based on tracking jumps of the underlying semimartingale. We also discuss a way to optimize the approximation rate by adapting the discretization times to the setting. When the small jump activity of the semimartingale behaves like an $\alpha$-stable process with $\alpha \in (1, 2)$, our scheme achieves under a regular regime the same convergence rate for the error as in Rosenbaum and Tankov [\textit{Ann. Appl. Probab.} \textbf{24} (2014) 1002--1048]. Moreover, our approach extends to the case $\alpha \in (0, 1]$ and to the $L_p$-setting which are not treated there. As an application, we apply the methods in the special case where the  semimartingale is an  exponential L\'evy process to mean-variance hedging of European type options.
	\end{abstract}
	
	
	\section{Introduction}
	
	\subsection{The problem and main results}
	\label{subsec:problem}
	This article is concerned with discrete-time approximation problems for  stochastic integrals and studies the error process 
	$ E= (E_t)_{t\in [0,T]}$ defined by
	\begin{align}\label{eq:stochastic-integral-general}
		E_t := \int_0^t  \vartheta_{u-} \od S_u- A_t,
	\end{align}
	where the time horizon $T \in (0,\infty)$ is fixed, $ \vartheta$  is an admissible integrand, $S$ is a semimartingale on a complete filtered probability space $(\Omega, \F, \p, (\F_t)_{t\in [0, T]})$  and $A = (A_t)_{t \in [0, T]}$ is an approximation scheme for the stochastic integral.

	We will consider two approximation methods, where the second builds on the first one. For the first method, the  \textbf{basic approximation method},  we
	assume that   $A = A^{\riem}$ is   the Riemann approximation  process of the above integral, 
	$$ A_t^{\riem}(\vartheta, \tau)  :=\sum_{i=1}^n \vartheta_{t_{i-1}-}(S_{t_i \wedge t}-S_{t_{i-1}\wedge t})$$
	for the \textit{deterministic} time-net $\tau_n=\{ 0=t_0 < t_1<\cdots <t_n =T\}.$
	We will study the corresponding error $E^{\riem}$   in $L_2,$ but {\it locally in time} in the sense that for any stopping time $\rho$ with values in $[0,T]$ we measure the error which accumulates within $[\rho,T].$  The term  {\it locally in time}  also includes that at the random time $\rho$ we restrict our problem
	to all sets $B\in \mathcal{F}_\rho$ of positive measure, which leads to the notion of {\it Bounded Mean Oscillation} (there are two abbreviations for it used in this article, $\bmo$ and $\BMO$, which express two different spaces).
	More precisely, we will work with {\it weighted}  $\bmo$-norms introduced in \cite{Ge05, GN20},  because we consider
	\begin{align} 
		\label{bmo-introduction}
		\ce{\F_\rho}{| E^{\riem}_T - E^{\riem}_{\rho} |^2} \lee c_{\eqref{bmo-introduction}}^2  \Phi^2_\rho \quad \mbox{a.s.},  \forall \rho. 
	\end{align}
	Here, $\E^{\F_\rho}$ stands for the conditional expectation with respect to $\F_{\rho},$ and the {\it weight process} $\Phi = (\Phi_t)_{t \in [0,T]}$ will be specified later.  
	Denote by $\|E^{\riem}\|_{{\bmo}_2^{\Phi}(\p)}$ the infimum of the  $c_{\eqref{bmo-introduction}} >0$ such that \eqref{bmo-introduction} is satisfied. We assert in \cref{theo:approximation-QV-BMO} that, under certain conditions, one has
	$$ \|E^{\riem}\|_{{\bmo}_2^{\Phi}(\p)} \lee c \sqrt{\|\tau_n\|_\theta}, $$
	where $\theta\in (0,1]$ is related (but not only) to the growth property of the integrand  $\vartheta$ by
	\begin{align}\label{eq:intro-growth}
		\ts \sup_{t \in [0, T)}(T -t)^{\frac{1 - \theta}{2}} |\vartheta_t| < \infty \quad \mbox{a.s.,}
	\end{align}
 and $\|\tau_n\|_\theta$ denotes a nonlinear mesh size of $\tau_n$ related to $\theta$. In \Cref{subsec:approx-accuracy} we discuss that $\tau_n$ can be chosen such that $\|\tau_n\|_\theta \lee c/n$, implying the approximation rate 
	$$ \|E^{\riem}\|_{{\bmo}_2^{\Phi}(\p)} \lee c/ \sqrt{n}.$$ Roughly speaking, the  faster  the  integrand  grows as $t\uparrow T$, the more the time-net should be concentrated near $T$ to compensate the growth.
	
	If the semimartingale $S$ has jumps, replacing  $E_{\rho}$ by  $E_{\rho-}$ in \eqref{bmo-introduction} leads to different norms, the $\BMO_2^{\Phi}(\p)$-norms. We will see in \eqref{eq:tail-estimate-BMO} and \cref{lemm:feature-BMO} that the $\BMO_2^\Phi(\p)$-norm gives us a way to achieve good distributional tail estimates for the error $E$ such as polynomial or exponential tail decay depending on the weight. Moreover, this approach allows us to switch the underlying measure $\p$ to an equivalent measure $\bQ$, which is frequently encountered in mathematical finance,  provided the change of measure satisfies a reverse H\"older inequality, so that the $\BMO_2^\Phi(\bQ)$-norm is equivalent to the $\BMO_2^\Phi(\p)$-norm.

	However, \cref{counterexample} below shows that if $S$ has jumps, then the Riemann approximation error $E^{\riem}$ does  in general not converge  to zero  if measured in the   $\BMO_2^{\Phi}(\p)$-norm.  The reason for
	this fact is that the $\BMO_2^{\Phi}(\p)$-norm is too sensitive to detect possibly large jumps of $S$, which is in contrast to the case of no jump in \cite{Ge05}.
	To overcome this difficulty, we adapt and develop further the idea using a
	\textit{small-large jump decomposition} of $S$ presented in Dereich and Heidenreich \cite{DH11} to our problem. This lets us design a new approximation scheme based on an adjustment of
	the Riemann sum which approximates the stochastic integral. This will be our  second method, the \textbf{jump correction method}, see \cref{defi:approximation-correction}. 
	
	Generally speaking, the
	approximation with jump correction $A^{\corr}(\vartheta, \tau|\ep, \kappa)$ is of the form
	\begin{align*}
		A^{\corr}_t(\vartheta, \tau|\ep, \kappa)  = A^{\riem}_t(\vartheta, \tau) + \mbox{Correction}_t(\ep, \kappa)
	\end{align*}
	where the parameters $\ep>0$ and $\kappa\gee 0$ in the correction term relate to the threshold for which we decide which jumps of $S$ are (relatively)
	large or small, and this threshold might continuously shrink when the time $t$ approaches $T$, see \cref{defi:random-times}. The time-net used in this approximation method is a combination of the given deterministic time-net $\tau$ in
	the Riemann sum and random times of carefully chosen large jumps of $S$.  A consequence of \cref{prop:cardinality-combined-nets} shows that the expected value of the cardinality of this combined time-net is, up to a multiplicative constant, comparable to the cardinality of $\tau$. This new approximation scheme can be interpreted in the context of mathematical finance as follows: Before trading, we arrange
	to use the Riemann approximation $A^\riem(\vartheta, \tau)$ associated with a trading strategy $\vartheta$ along with
	a preselected deterministic trading dates represented by $\tau$. During trading with that initial
	plan, as soon as the large jumps of $S$ occur, we trade additionally with the amount given in the correction term.
	
	Denote by $E^{\corr}$ the error caused from the approximation with the jump correction scheme.  To formulate the result, we assume  that $S$ is given as the (strong) solution of 
	$$  \od S_t  = \sigma (S_{t-}) \od Z_t, $$
	with $ \sigma $ specified later, where $Z$ is a square integrable semimartingale defined in \Cref{subsec:setting-stochastic-integal}. Then 
	\cref{theo:BMO-convergent-rate} implies that, for suitably chosen  time-nets and corrections, and for a suitable modification $\ol \Phi$ of $\Phi$, it holds that
	\begin{align}\label{eq:intro-1}
		\| E^{\corr} \|_{\BMO_2^{\ol \Phi}(\p)} \lee  c/ \sqrt{n}	
	\end{align}
	under the condition that  the random measure $\pi_Z$ of the predictable semimartingale characteristics  of $Z$ satisfies that $\pi_Z(\od t, \od z) = \nu_t(\od z) \od t$ and that
	\begin{align}\label{eq:intro-2}
		\ts \sup_{r \in (0, 1)} \big\|(\omega, t) \mapsto \int_{r < |z| \lee 1}  z \nu_t(\omega, \od z) \big\|_{L_\infty(\Omega \times [0, T], \p \otimes \Leb)} < \infty,
	\end{align}
	where $\Leb$ is the Lebesgue measure, and one has
	\begin{align}\label{eq:intro-3}
		\| E^{\corr} \|_{\BMO_2^{\ol\Phi}(\p)} \lee c \begin{cases}
			1/\sqrt[2\alpha]{n} & \mbox { if } \alpha \in (1, 2]\\
			(1+\log n)/\sqrt{n}& \mbox{ if } \alpha = 1\\			
			1/\sqrt{n} & \mbox{ if } \alpha \in (0, 1)
		\end{cases}
	\end{align} 
	provided that 
	\begin{align}\label{eq:intro-4}
		\ts \sup_{r \in (0, 1)} \big\|(\omega,t) \mapsto r^\alpha \int_{r < |z| \lee 1} \nu_t(\omega, \od z)\big\|_{L_\infty(\Omega \times [0, T], \p \otimes \Leb)} <\infty.
	\end{align}
	Condition \eqref{eq:intro-2} aims to indicate a local symmetry of $\nu$ around the origin rather than the small jump intensity of $Z$ which is described by \eqref{eq:intro-4}.

	Since the integrator $Z$ and structure conditions imposed on the approximated stochastic integral to achieve \eqref{eq:intro-1} and \eqref{eq:intro-3} are quite general, those obtained convergence rates are in general not the best possible. We will show in \Cref{subsec:model-improve-rate} that one can drastically improve those convergence rates in the particular case when $S$ is the Dol\'eans--Dade exponential of a pure jump process $Z$ where $Z$ has independent increments. Namely, \cref{theo:BMO-convergent-rate} asserts that, for the error $E^{\corr}$ as above and under certain structure conditions for the approximated stochastic integral, one has
	\begin{align}\label{eq:intro-5}
		\| E^{\corr} \|_{\BMO_2^{\ol\Phi}(\p)} \lee c \begin{cases}
			1/n^{\frac{1}{\alpha}(1- \frac{1}{2}(1- \theta)(\alpha -1))} & \mbox { if } \eqref{eq:intro-4} \mbox{ holds for } \alpha \in (1, 2]\\
			(1+\log n)/n& \mbox{ if } \eqref{eq:intro-4} \mbox{ holds for } \alpha = 1\\			
			1/n & \mbox{ if } \eqref{eq:intro-4} \mbox{ holds for } \alpha = 1 \mbox{ and } \eqref{eq:intro-2} \mbox{ holds}\\
			1/n & \mbox{ if } \eqref{eq:intro-4} \mbox{ holds for } \alpha \in (0, 1),
		\end{cases}
	\end{align} 
	where $\theta \in (0, 1]$ relates to the growth of $\vartheta$ mentioned in \eqref{eq:intro-growth}.

	Furthermore, \cref{theo:BMO-convergent-rate} also reveals that, if the weight $\Phi$ is sufficiently regular, then the estimates \eqref{eq:intro-1}, \eqref{eq:intro-3} and \eqref{eq:intro-5} hold true for the $L_p$-norm, $p \in (2, \infty)$, in place of the $\BMO_2^{\ol\Phi}(\p)$-norm. In addition, the measure $\p$ can be substituted by a suitable equivalent probability measure $\bQ$ while keeping those estimates unchanged. 
	
	The parameter $n$ in \eqref{eq:intro-1}, \eqref{eq:intro-3} and \eqref{eq:intro-5} refers to certain moments of the cardinality of the combined time-net used in the approximation. This cardinality represents in the context of mathematical finance the number of transactions performed in trading, see \cref{rema:factor-n}.

	As an application, we  choose $S$ to be an  exponential L\'evy  process and measure the discretization error for stochastic integrals where the integrands are  mean-variance hedging (MVH) strategies  of European payoffs.
	To do this, we provide in \cref{theo:MVH-strategy} using Malliavin calculus an explicit representation of the MVH  strategy for a European payoff for which we do not require any
	regularity for payoff functions nor specific structures from the underlying L\'evy process. This result is, to the best of our knowledge, 
	new in this generality and it might have an independent interest.

	Let us end this subsection by listing some examples taken from \cref{coro:convergence-rate-levy} showing convergence rates for $E^{\corr}$ under the ${\BMO_2^{\ol\Phi}(\p)}$-norm in the exponential L\'evy setting. Namely, we let $S = \e^X$ where $X$ is a L\'evy process without the Brownian part whose small jump intensity behaves like an $\alpha$-stable process with $\alpha \in (0, 2)$, and let $\vartheta$ be the MVH strategy of a European payoff $g(S_T)$. Then, for the European call/put option (or any Lipschitz $g$), the convergence rate is of order:  $1/n$ if $\alpha \in (0, 1)$, $(1+\log n)/n$ if $\alpha =1$, and $1/\sqrt[\alpha]{n}$ if $\alpha \in (1, 2)$. For the binary option (or any bounded $g$), the order of convergence rate is: $1/n$ if $\alpha \in (0, 1)$, $(1+\log n)/n$ if $\alpha =1$, and   $1/n^{\frac{1}{\alpha}[1- \frac{1}{\alpha}(\alpha -1)^2] - \delta}$ (for any $0<\delta < \frac{1}{2}(1- \frac{1}{\alpha})(\frac{2}{\alpha}-1)$) if $\alpha \in (1, 2)$. Moreover, if $\E \e^{pX_T} <\infty$ for some $p \in (2, \infty)$, then measuring $E^{\corr}$ in $L_p$ yields the same rates case-wise as above.  Lastly, our results are valid for some powered call/put options which are obtained from an interpolation, in a sense, between the binary and the call/put option.

	\subsection{Literature overview} 
	Besides its own mathematical interest and its application to numerical methods, the  approximation of a stochastic integral has a direct motivation in mathematical finance. Let us briefly discuss this for the Black--Scholes model. Assume that the (discounted) price of a risky asset is modelled by a stochastic process $S$ which solves the stochastic differential equation (SDE) $\od S_t = \sigma(S_t) \od W_t$, where $W$ is the standard Brownian motion and the function $\sigma$ satisfies some suitable conditions. For a  European type payoff $g(S_T)$ satisfying an integrability condition, it is known that $g(S_T) =\E g(S_T) + \int_0^T \pd_y G(t, S_t)\od S_t$,
	where $G(t, y): = \E(g(S_T)|S_t =y)$ is the option price function and $(\pd_y G(t, S_t))_{t\in [0, T)}$ is the so-called delta-hedging strategy. The stochastic integral in the representation of $g(S_T)$ above can be interpreted as the theoretical hedging portfolio which is rebalanced continuously. However, it is not feasible in practice because one can only readjust the portfolio finitely many times. This leads to a replacement of the stochastic integral by a discretized version, and this substitution causes the discretization error. 
	
	The error represented by the difference between a stochastic integral and its discretization has been extensively analyzed in various contexts. It is usually studied in $L_2$ for which one can exploit the orthogonality to reduce the probabilistic setting to a ``more deterministic'' setting where the corresponding quadratic variation is employed instead of the original error. In the Wiener space, we refer, e.g., to \cite{GG04, Ge02, GT01}, where the error along with its convergence rates was examined. The weak convergence of the error was treated in  \cite{GT09, GT01}. When the driving process is a continuous semimartingale, the convergence in the $L_2$-sense was studied in \cite{Fu14}, and in the almost sure sense it was considered in \cite{GL14}.

	In this article, we allow the semimartingale to jump since many important processes used in financial modelling are not continuous (see \cite{CT03, Sc03}), and the presence of jumps has a significant effect on the hedging errors. Moreover, models with jumps typically correspond to incomplete markets. This means that beside the error resulting from the impossibility of continuously rebalancing a portfolio, there is another hedging error due to the incompleteness of the market. The latter problem was studied in many works (see an overview in \cite{Sc01} and the references therein). The present article mainly focuses on the first type of hedging error. The discretization error was studied within L\'evy models in the weak convergence sense in \cite{TV09}, in the $L_2$-sense in \cite{BT11, GGL13}, and for a general jump model under the $L_2$-setting in \cite{RT14}.

	In general, the classical $L_2$-approach for the error yields a second-order polynomial decay for its distributional tail by Markov's inequality. If higher-order decays are needed, then the $L_p$-approach  ($2 < p < \infty$) is considered as a natural choice, and this direction has been investigated for diffusions on the Wiener space in \cite{GT15}. A remarkably different route given in \cite{Ge05} is that one can study the error in weighted $\BMO$ spaces. The main benefit of the weighted $\BMO$-approach is a John--Nirenberg type inequality (\cite[Corollary 1(ii)]{Ge05}): \textit{If the error process $E$ belongs to  $\BMO_p^\Phi(\p)$ for some $p\in (0, \infty)$, where $\Phi$ is some weight function specified in \cref{definition:weighted_bmo}, then there are constants $c, d>0$ such that for any stopping time $\rho \colon \Omega \to [0, T]$ and any $\alpha, \beta >0$,
		\begin{align}\label{eq:tail-estimate-BMO}
			\textstyle \p\(\sup_{u \in [\rho, T]} |E_u - E_{\rho-}| > c \alpha \beta \big|\F_\rho\) \lee \e^{1- \alpha} + d \p\(\sup_{u\in [\rho, T]} \Phi_u > \beta\big|\F_\rho\).
	\end{align}}
	Obviously, if $\Phi$ has a good distributional tail  estimate, for example, if it has a polynomial or exponential tail decay, then by adjusting $\alpha$ and $\beta$ one can derive a tail estimate for $E$ accordingly. Especially, one can then derive $L_p$-estimates, $p\in (2, \infty)$, for the error.

	\subsection{Comparison to other works} Regarding models with jumps, let us first mention the works done by Brod\'en and Tankov \cite{BT11} and by Geiss, Geiss and Laukkarinen \cite{GGL13} which treat the Riemann approximation of stochastic integrals driven by the stochastic exponential of a L\'evy process using deterministic discretization times. Although the approaches in \cite{BT11} and \cite{GGL13} are different, both arrive at a result saying that, if the approximated stochastic integral is sufficiently regular, then the \textit{asymptotically optimal} convergence rate of the error measured in $L_2$ is of order $1/\sqrt{n}$ when $n \to \infty$ (see \cite[Corollary 3.1]{BT11} and \cite[Theorem 5]{GGL13}), where $n$ is the cardinality of the used time-net. For this direction, we also achieve in \cref{theo:BMO-convergent-rate}\eqref{item:coro:bmo-convergent-rate} the convergence rate of order $1/\sqrt{n}$ for the error under the $\bmo_2^{\Phi}$-norm, which is stronger than the $L_2$-norm.
	
	Later, Rosenbaum and Tankov \cite{RT14} show  that the convergence rate can be faster than $1/\sqrt{n}$ by using Riemann approximations associated with random discretization times. It is asserted in \cite[Remark 5]{RT14}  that, when the small jump activity of the semimartingale integrand behaves like an $\alpha$-stable process with $\alpha \in (1, 2)$, then the convergence rate measured in $L_2$ is of order $1/\sqrt[\alpha]{n}$, which is also \textit{asymptotically optimal} in their setting. In our framework, under a \textit{regular regime} when $\theta =1$, we derive from \eqref{eq:intro-5} the rate $1/\sqrt[\alpha]{n}$ under the weighted BMO-norm (which is stronger than the $L_2$-norm) when \eqref{eq:intro-4}  holds for some $\alpha \in (1, 2]$. This rate is consistent with that in \cite[Remark 5]{RT14} when $n$ represents the expected number of transactions. Moreover, our results are valid for $\alpha \in (0,1]$ which is not covered in \cite{RT14}.

	We stress that our jump correction scheme is different from that in \cite{RT14}. The authors in \cite{RT14} use Riemann approximation schemes along with random times and  employed time-nets are the hitting times of a space grid which are obtained by continuously tracking
	jumps of the integrand (which represents the trading strategy). Differently from that, time-nets in our method are a combination between preselected deterministic time-nets with random times obtained by continuously tracking jumps of the integrator (which represents the price process). In general, the computational cost of our method is less expensive than that in \cite{RT14}. This can be argued in a situation when many options with different strategies are hedged at the same time with respect to a risky asset. 
	
	One also remarks that, in this article, we only require upper bound conditions for integrands, e.g., \eqref{eq:intro-growth}, and for small jump intensity of the integrator $Z$, e.g., \eqref{eq:intro-2} or \eqref{eq:intro-4}, to obtain upper bounds for approximation errors, and so far no lower bounds are investigated.
	
	Other contributions of this work are, thanks to features of the (weighted) BMO approach as aforementioned, to provide a situation that one can deduce $L_p$-estimates, $p \in (2, \infty)$, for the approximation error which are, to the best of our knowledge, still missing in the literature for models with jumps. Moreover, as a benefit to applications in mathematical finance, our results allow a change of the underlying measure which leaves the error estimates unchanged if the change of measure satisfies a reverse H\"older inequality, see \cref{lemm:feature-BMO}.

	\subsection{Structure of the article} Some standard notions and notations are contained in \cref{sec:preliminaries}. The main results are provided in \cref{sec:main-results} and theirs proofs are given in \cref{sec:proofs-main-results}. In
	\cref{sec:application-Levy-model}, we give some applications of those main results in exponential L\'evy models. \cref{app-sec:proof-MVH-strategy} presents briefly Malliavin calculus for L\'evy processes which is the main tool to obtain an explicit MVH strategy for a European type option in \cref{theo:MVH-strategy}. The regularity of weight processes used in this article is shown in \cref{sec:weight-regularity}. In \cref{sub-sec:estimae-semigroup}, we establish some gradient type estimates for a L\'evy semigroup on H\"older spaces, which are used to verify the main results in the L\'evy setting.

	\section{Preliminaries}
	\label{sec:preliminaries}
	
	\subsection{Notations and conventions} \label{subsec:notation}
	\subsubsection*{General notations}
	Denote $\bN := \{0, 1, 2, \ldots\}$, $\R_+:=(0, \infty)$ and $\R_0 := \R \backslash \{0\}$. For $a, b \in \R$, we set $a\vee b := \max\{a, b\}$ and $a\wedge b : = \min\{a, b\}$. For $A, B \gee 0$ and $c\gee 1$, the notation $A \sim_c B$ stands for $A/c \lee B \lee cA$. The notation $\log$ indicates the logarithm to the base 2 and $\log^+ x : = \log (x\vee 1)$. Subindexing a symbol by a label means the place where that symbol appears (e.g., $c_{\eqref{eq:reverse-Holder-condition}}$ refers to the relation \eqref{eq:reverse-Holder-condition}).

	The Lebesgue measure on the Borel $\sigma$-algebra $\cB(\R)$ is denoted by $\Leb$, and we also write $\od x$ instead of $\Leb(\od x)$ for simplicity. For $p \in [1, \infty]$ and $A \in \cB(\R)$, the notation $L_p(A)$ means the space of all $p$-order integrable Borel functions on $A$ with respect to $\Leb$, where the essential supremum is taken when $p=\infty$. 
	
	Let $\xi$ be a random variable defined on a probability space $(\Omega, \F, \p)$. The push-forward measure of $\p$ with respect to $\xi$ is denoted by $\p_\xi$. If $\xi$ is integrable (non-negative), then the (generalized) conditional expectation of $\xi$ given a sub-$\sigma$-algebra $\G\subseteq \F$ is denoted by  $\ce{\G}{\xi}$. We also agree on the notation $L_p(\p) : = L_p(\Omega, \F, \p)$. 

	\subsubsection*{Notations for stochastic processes} Let $T \in (0, \infty)$ be fixed and $(\Omega, \F, \p)$  a complete probability space equipped with a right continuous filtration $\bF = (\F_t)_{t\in [0, T]}$. Assume that $\F_0$ is generated by $\p$-null sets only. Because of the  conditions imposed on $\bF$, we may assume that every martingale adapted to this filtration is \textit{c\`adl\`ag} (right-continuous with left limits). For $\bI = [0, T]$ or $\bI = [0, T)$, we use the following notations:
	\begin{enumerate}[--]
		\itemsep0.3em
		
		\item For two processes $X =(X_t)_{t\in \bI}$, $Y = (Y_t)_{t\in \bI}$, by writing $X = Y$ we mean that $X_t = Y_t$ for all $t\in \bI$ a.s., and similarly when the relation ``$=$'' is replaced by some standard relations such as ``$\gee$'', ``$\lee$'', etc. 
		\item For a c\`adl\`ag  process $X = (X_t)_{t\in \bI}$, we define the process $X_{-} = (X_{t-})_{t\in \bI}$ by setting $X_{0-}: = X_0$ and $X_{t-} : = \lim_{0<s \uparrow t}X_s$ for $t\in \bI\backslash\{0\}$. In addition, set $\Delta X := X - X_-$.
		\item $\CL(\bI)$ denotes the family of all  c\`adl\`ag and $\bF$-adapted processes $X= (X_t)_{t\in \bI}$.
		\item $\CL_0(\bI)$ (resp. $\CL^+(\bI)$)  consists of all $X\in \CL(\bI)$ with $X_0=0$ a.s.  (resp. $X\gee 0$);
		\item Let $M= (M_t)_{t\in \bI}$ and $N= (N_t)_{t\in \bI}$ be $L_2(\p)$-martingales adapted to $\bF$. The \textit{predictable quadratic covariation} of $M$ and $N$ is denoted by $\<M, N\>$. If $M=N$, then we simply write $\<M\>$ instead of $\<M, M\>$.
		\item For $p \in [1, \infty]$ and $X \in \CL([0, T])$, we denote $\|X\|_{S_p(\p)}: = \|\sup_{t \in [0, T]}|X_t|\|_{L_p(\p)}$.
	\end{enumerate}

	\subsection{Weighted bounded mean oscillation and regular weight} \label{defi:weighted-BMO/bmo} 
	We recall the notions of weighted bounded mean oscillation and the space $\cSM_p(\p)$ of regular weight processes (the abbreviation $\cSM$ indicates the property resembling a supermartingale). Let $\cS([0, T])$ be the family of all stopping times $\rho\colon \Omega \to [0, T]$, and set $\inf\emptyset : = 
	\infty$.
	\begin{defi}[\cite{Ge05, GN20}]\label{definition:weighted_bmo} 
		For $p\in(0,\infty)$, $Y\in \CL_0([0, T])$ and $\Phi \in \CL^+([0, T])$, define
		\begin{align*}
			\|Y\|_{\BMO_p^{\Phi}(\p)} &:= \inf\left\{ c \gee 0 : \ce{\F_\rho}{|Y_T-Y_{\rho-}|^p} \leqslant c^p \Phi_{\rho}^p \quad \mbox{a.s., } \forall \rho \in \cS([0, T])\right\},\\
			\|Y\|_{\bmo_p^{\Phi}(\p)} &:= \inf\left\{ c \gee 0 : \ce{\F_\rho}{|Y_T-Y_{\rho}|^p} \leqslant c^p \Phi_{\rho}^p \quad \mbox{a.s., } \forall \rho \in \cS([0, T])\right\},\\
			\| \Phi\|_{\cSM_p(\p)} & := \inf\left\{c \gee 0 : \ce{\F_\rho}{\ts \sup_{\rho \lee t \lee T} \Phi_t^p} \leqslant c^p \Phi_\rho^p \quad \mbox{a.s., } \forall \rho \in \cS([0, T]) \right\}.
		\end{align*}
		For $\Gamma \in \{\BMO^\Phi_p(\p), \bmo^\Phi_p(\p)\}$, if $\|Y\|_{\Gamma} <\infty$ (resp. $\|\Phi\|_{\cSM_p(\p)} <\infty$), then we write $Y \in\Gamma$ (resp. $\Phi \in \cSM_p(\p)$). In the non-weighted case, i.e., $\Phi \equiv 1$, we drop $\Phi$ and simply use the notation $\BMO_p(\p)$ or $\bmo_p(\p)$.
	\end{defi}
	
	\begin{rema}\label{rema:deterministic-bmo}
		By \cite[Propositions A.4 and A.1]{GN20}, the definitions of $\|\cdot\|_{\bmo^\Phi_p(\p)}$ and $\|\cdot\|_{\cSM_p(\p)}$ can be simplified by using deterministic times $a \in [0, T]$ instead of stopping times $\rho$, i.e.,
		\begin{align*}
			\|Y\|_{\bmo_p^{\Phi}(\p)} &= \inf\left\{ c \gee 0 : \ce{\F_a}{|Y_T-Y_{a}|^p} \leqslant c^p \Phi_{a}^p \quad \mbox{a.s., } \forall a\in [0, T] \right\},\\
			\| \Phi\|_{\cSM_p(\p)} &= \inf\left\{c \gee 0 : \ce{\F_a}{\ts \sup_{a \lee t \lee T} \Phi_t^p} \leqslant c^p \Phi_a^p \quad \mbox{a.s., } \forall a\in [0, T] \right\}.
		\end{align*}
	\end{rema}

	The theory of classical non-weighted $\BMO$/$\bmo$-martingales can be found in \cite[Ch.VII]{DM82} or \cite[Ch.IV]{Pr05}, and they were used later in different contexts (see, e.g., \cite{CKS98, DMSSS97}). The notion of weighted $\BMO$ space above was introduced and discussed in \cite{Ge05} where it was developed for general c\`adl\`ag processes which are not necessarily martingales. 
	
	It is clear that if $Y \in \CL_0([0, T])$ is continuous, then $\|Y\|_{\bmo^\Phi_p(\p)}  = \|Y\|_{\BMO^\Phi_p(\p)}$. If $Y$ has jumps, then the relation between weighted $\BMO$ and weighted $\bmo$  is 
	as follows.
	
	\begin{lemm}[\cite{GN20}, Propositions A.5 and A.3]\label{lemm:relation-BMO-bmo} If $\Phi \in \cSM_p(\p)$ for some $p \in (0, \infty)$, then there is a constant $c= c(p, \|\Phi\|_{\cSM_p(\p)}) >0$ such that for all $Y \in \CL_0([0, T])$,
		\begin{align*}
			\|Y\|_{\BMO^\Phi_p(\p)} \sim_{c} \|Y\|_{\bmo^\Phi_p(\p)} + |\Delta Y|_{\Phi},
		\end{align*}
		where $|\Delta Y|_{\Phi} : = \inf\{ c\gee 0 : |\Delta Y_t| \lee c \Phi_t \;\mbox{ for all } t\in [0, T] \mbox{ a.s.}\}.$
	\end{lemm}

	\begin{defi}[\cite{Ge05}]\label{defi:reverse-Holder} 
		Let $\bQ$ be an equivalent probability measure to $\p$ so that $U : =\od \bQ/\od \p >0$. Then $\bQ \in \RH_s(\p)$ for some $s\in (1, \infty)$ if $U \in L_s(\p)$ and if there is a constant $c_{\eqref{eq:reverse-Holder-condition}}>0$ such that $U$ satisfies the following \textit{reverse H\"older inequality}
		\begin{align}\label{eq:reverse-Holder-condition}
			\ce{\F_\rho}{U^s} \lee c_{\eqref{eq:reverse-Holder-condition}}^s (\ce{\F_\rho}{U})^s \quad\mbox{a.s., } \forall \rho \in \cS([0, T]),
		\end{align}
		where the conditional expectation $\E^{\F_\rho}$ is computed under $\p$.
	\end{defi}

	We recall in \cref{lemm:feature-BMO} some features of weighted $\BMO$ which play a key role in our applications. Notice that \cref{lemm:feature-BMO} is, in general, \textit{not} valid for weighted $\bmo$.
	
	\begin{prop}[\cite{Ge05, GN20}]\label{lemm:feature-BMO} Let $p, q \in (0, \infty)$ and $\Phi \in \CL^+([0, T])$.
		\begin{enumerate}[\rm (1)]
			\itemsep0.3em
			
			\item\label{item:Lp-estimate-BMO-feature} There exists a constant $c_1 = c_1(p, q)>0$ such that $\|\cdot\|_{S_p(\p)} \lee c_1 \|\Phi\|_{S_p(\p)} \|\cdot\|_{\BMO^\Phi_q(\p)}$. 
			
			\item \label{item:SM-BMO-relation} If $\Phi \in \cSM_{p}(\p)$, then for any $r \in (0, p]$ there is a constant $c_2=c_2(r, p, \|\Phi\|_{\cSM_p(\p)})>0$ such that  $\|\cdot\|_{\BMO_p^\Phi(\p)} \sim_{c_2} \|\cdot\|_{\BMO_r^\Phi(\p)}$.
			
			\item \label{item:BMO-feature-RH} If $\bQ \in \RH_s(\p)$ for some $s \in (1, \infty)$ and $\Phi \in \cSM_p(\bQ)$, then there is a constant $c_3 = c(s, p) >0$ such that $		\|\cdot\|_{\BMO^\Phi_p(\bQ)} \lee c_3 \|\cdot\|_{\BMO^\Phi_p(\p)}$.
		\end{enumerate}
	\end{prop}
	\begin{proof}
		Items \eqref{item:Lp-estimate-BMO-feature} and \eqref{item:SM-BMO-relation} are due to \cite[Proposition A.6]{GN20}. For Item \eqref{item:BMO-feature-RH}, we apply \cite[combine Corollary 1(i) with Theorem 3]{Ge05} to the weight $\Phi + \ep>0$ and then let $\ep \downarrow 0$.
	\end{proof}


	\subsection{The class of approximated stochastic integrals} \label{subsec:setting-stochastic-integal}
	Throughout this article, the assumptions for the stochastic integral in \eqref{eq:stochastic-integral-general} are the following. 
\begin{enumerate}[\quad]
	\item[\textbf{[S]}] The process $S \in \CL([0, T])$ is the strong solution of the SDE\footnote{See, e.g., \cite[Ch.V, Sec.3]{Pr05}, for the existence and uniqueness of $S$.}
	\begin{align}\label{eq:SDE-price-process}
		\od S_t = \sigma(S_{t-})\od Z_t, \;\;S_0 \in \cR_S,
	\end{align}	
	where $\sigma\colon \cR_S \to (0, \infty)$ is a Lipschitz function on an open set $\cR_S \subseteq \R$ with $S_t(\omega), S_{t-}(\omega) \in \cR_S$ for all $(\omega, t) \in \Omega \times [0, T]$. We denote
	\begin{align*}
		|\sigma|_{\Lip} : = \sup_{x, y \in \cR_S,\, x\neq y} \left|\frac{\sigma(y)-\sigma(x)}{y-x}\right| <\infty.
	\end{align*}

\item[\textbf{[Z]}] The process $Z \in \CL([0, T])$ is a square integrable semimartingale on $(\Omega, \F, \p, (\F_t)_{t\in [0, T]})$ with the representation
\begin{align}\label{eq:decomposition-Z}
	Z_t = Z_0  + Z^{\rc}_t  + \int_0^t \!\!\int_{\R_0} z (N_Z - \pi_Z)(\od u, \od z) + \int_0^t b^Z_u \od u,\quad t \in [0, T],
\end{align}
where $Z_0 \in \R$, $b^Z$ is a progressively measurable process, $Z^\rc$ is a pathwise continuous square integrable martingale with $Z^\rc_0=0$, $N_Z$ is the jump random measure\footnote{$N_Z((s, t] \times B): = \#\{u \in (s, t] : \Delta Z_u \in B\}$ and $N_Z(\{0\} \times B): = 0$ for $0\lee s <t \lee T$, $B \in \cB(\R_0)$.} of $Z$ and $\pi_Z$ is the predictable compensator\footnote{$\pi_Z$ is such that: (i) for any $\omega \in \Omega$, $\pi_Z(\omega, \cdot)$ is a measure on $\cB([0, T]\times \R)$ with $\pi_Z(\omega, \{0\} \times \R) = 0$; (ii) for any $\mathcal P \otimes \cB(\R)$-measurable and non-negative $f$, the process $\int_0^{\boldsymbol{\cdot}}\int_{\R} f(u, z) \pi_Z(\od u, \od z)$ is $\mathcal P$-measurable satisfying $\E \int_0^T \!\int_{ \R}f(u, z) N_Z(\od u, \od z) = \E \int_0^T \!\int_{ \R} f(u, z) \pi_Z(\od u, \od z)$, where $\mathcal P$ is the predictable $\sigma$-algebra on $\Omega \times [0, T]$  (see  \cite[Ch.II, Sec.1]{JS03}  for more details).} of $N_Z$. Assumptions for $Z$ are the following:
\begin{enumerate}[\rm (Z1)]
	\item For all $\omega \in \Omega$,
	\begin{align} \label{eq:condition-compensator-measure}
		\pi_Z(\omega, \od t, \od z) = \nu_t(\omega, \od z) \od t,
	\end{align} 
	where the transition kernel $\nu_t(\omega, \cdot)$ is a L\'evy measure, i.e., a Borel measure on $\cB(\R)$ satisfying $\nu_t(\omega, \{0\}) := 0$ and $\int_{\R}(z^2 \wedge 1) \nu_t(\omega, \od z) <\infty$. 
	
	\item There is a progressively measurable process $a^Z$ such that $\od \< Z^\rc\>_t = |a_t^Z|^2 \od t$ and 
	\begin{align}\label{eq:uniformly-bounded-C}
		a^Z_{\eqref{eq:uniformly-bounded-C}} : = \|a^Z\|_{L_\infty(\Omega \times [0, T], \p \otimes \Leb)} <\infty.
	\end{align}
	
	\item The processes $b^Z$ and $j^Z$, where $j^Z_t : = \big(\int_{\R} z^2 \nu_t(\od z)\big)^{1/2}$, satisfy that
	\begin{align}\label{eq:uniformly-bounded-KV}
		& b^Z_{\eqref{eq:uniformly-bounded-KV}} : = \|\|b^Z\|_{L_2([0, T], \Leb)}\|_{L_\infty(\p)}  <\infty, \quad j^Z_{\eqref{eq:uniformly-bounded-KV}} : = \|j^Z\|_{L_\infty(\Omega \times [0, T], \p \otimes \Leb)} <\infty.
	\end{align}
\end{enumerate}

\item[\textbf{[I]}] The process $\vartheta$ belongs to the family $\adm$ of \textit{admissible integrands},  where
\begin{align*}
	\adm : = \left\{\vartheta \in \CL([0, T)) : \E \int_0^T \vartheta_{t-}^2 \sigma(S_{t-})^2 \od t <\infty \quad \mbox{and} \quad \Delta \vartheta_{t} = 0 \mbox{ a.s., } \forall t\in [0, T) \right\}.
\end{align*}
\end{enumerate}

	\begin{rema} \label{rema:assumption-S-Z}
		\begin{enumerate}[\rm (1)]
			\item By a standard stopping argument and Gronwall's lemma, \eqref{eq:SDE-price-process} implies that $S$ is an $L_2(\p)$-semimartingale and 
			\begin{align} \label{eq:square-integrability-S}
				\E \int_0^T \sigma(S_{u})^2 \od u = \E \int_0^T \sigma(S_{u-})^2 \od u  <\infty.
			\end{align}
			
			\item  For each $t\in [0, T]$, it follows from \eqref{eq:condition-compensator-measure} that $N_Z(\{t\} \times \R_0
			) = 0$ a.s., which verifies $\Delta Z_t = 0$ a.s.,  and hence, $\Delta S_t =0$ a.s. In other words, $Z$ and $S$ have no fixed-time discontinuity. Since admissible integrands $\vartheta$ in applications are often functionals of the integrator $S$, it is natural to assume technically that $\Delta \vartheta_t =0$ a.s.
		\end{enumerate}

		
	\end{rema}

	\section{Approximation via weighted bounded mean oscillation approach} \label{sec:main-results}

	
	To examine the discrete-time approximation problem in weighted $\bmo$ or weighted $\BMO$, further structure of the integrand is required. We begin with the following assumption which is an adaptation of \cite[Assumption 5.1]{GN20}.

	\begin{assum}\label{pre-assumption-stochastic-integral} Assume for a $\vartheta \in  \adm$ that there exists a random measure $$\Upsilon \colon \Omega \times \cB((0, T)) \to [0, \infty]$$
		such that $\Upsilon(\omega, (0, t]) <\infty$ for all $(\omega, t) \in \Omega \times (0, T)$
		and there is a constant $c_{\eqref{eq:assumption-stochastic-integral}}>0$ such that for any $0 \lee a < b <T$,
		\begin{align}\label{eq:assumption-stochastic-integral}
			\bbce{\F_a}{\int_{(a, b]} |\vartheta_t - \vartheta_a|^2 \sigma(S_t)^2 \od t} \lee c_{\eqref{eq:assumption-stochastic-integral}}^2 \bbce{\F_a}{\int_{(a, b]} (b-t) \Upsilon (\cdot, \od t)} \quad \mbox{a.s.}
		\end{align}
	\end{assum}
	
	The left-hand side of \eqref{eq:assumption-stochastic-integral} appears as the one-step conditional $L_2$-approximation error. This error is assumed to be controlled from above by a conditional integral with respect to an appropriate random measure $\Upsilon$, where $\Upsilon$ might have some singularity at the terminal time $T$. This structure condition allows us to derive the multi-step approximation from the one-step approximation. Apparently, the measure $\Upsilon$ looks  artificial, however, it origins from the diffusion setting on the Wiener space. We briefly explain this in \cref{exam:Upsilon-diffusion-setting} below.
	\begin{exam}[Diffusion setting]\label{exam:Upsilon-diffusion-setting} We recall the setting from \cite{GG04} (see also \cite{GN20}):
		Let $\hat \sigma\colon \R \to \R$ be bounded, $\inf_{x\in \R}\hat \sigma(x) >0$ and infinitely differentiable with bounded derivatives.  Let $\sigma(x): = x \hat \sigma(\ln x)$, which is Lipschitz on $(0, \infty)$, and consider the SDE 
		$$\od S_t = \sigma(S_t) \od W_t,\quad S_0 = \e^{x_0} >0,$$
		 where $W$ is a standard Brownian motion. Then one has $S_t = \e^{X_t}$,  where  $\od X_t = \hat \sigma(X_t) \od W_t  - \frac{1}{2}\hat \sigma(X_t)^2 \od t$, $X_0=x_0 \in \R$. For any  Borel function $g\colon (0, \infty) \to \R$ with polynomial growth, It\^o's formula asserts $g(S_T) = \E g(S_T) + \int_0^T \pd_y G(t, S_t) \od S_t$
		where $G(t, y) : = \E(g(S_T)|S_t = y)$.  Then \cite[Corollary 3.3]{Ge02} (see also \cite[Section 7]{GN20}) verifies that \cref{pre-assumption-stochastic-integral} is satisfied for
		$$\vartheta_t := \pd_y G(t, S_t), \quad \Upsilon(\omega, \od t)  := |(\sigma^2 \pd^2_y G)(t, S_t(\omega))|^2 \od t.$$
		In this case, since $\Upsilon$ is related to the second derivative, it describes some kind of \textit{curvature} of the stochastic integral.
	\end{exam}

	We now provide in \cref{exam:Upsilon-martingale-setting} another formula for $\Upsilon$ which is used later in the exponential L\'evy setting in \cref{sec:application-Levy-model}.

	\begin{exam}\label{exam:Upsilon-martingale-setting}
		Assume for $\vartheta \in  \adm$ that $\vartheta \sigma(S)$  is a semimartingale with the representation
		\begin{align*}
			\vartheta_t \sigma(S_t) = M_t + V_t, \quad t \in [0, T),
		\end{align*}
	where $M = (M_t)_{t \in [0, T)}$ is an $L_2(\p)$-martingale, $V_t = \int_0^t v_u \od u$ for a progressively measurable $v$ with $\int_0^t v_u^2(\omega) \od u < \infty$ for all $(\omega, t) \in \Omega \times [0, T)$. Then \cref{pre-assumption-stochastic-integral} is satisfied for
		\begin{align*}
			\Upsilon(\omega, \od t) : = \od \<M\>_t(\omega) +  |\sigma|_{\Lip}^2 M^2_t(\omega) \od t + v_t^2(\omega) \od t + |\sigma|_{\Lip}^2 \int_0^t v_u^2(\omega) \od u\, \od t.
		\end{align*}
		Indeed, for any $0\lee a <b<T$, first using the triangle inequality, then applying conditional It\^o's isometry for the martingale term and H\"older's inequality for the Lebesgue integrals, together with Fubini's theorem and  \cref{lem:sigma(Y)-estimate}, we have, a.s., 
		\begin{align*}
			&\frac{1}{3}\bbce{\F_a}{\int_{(a, b]} |\vartheta_t - \vartheta_a|^2 \sigma(S_t)^2 \od t}\\
			 & \lee \bbce{\F_a}{\int_{(a, b]}|M_t - M_a|^2 \od t} + \bbce{\F_a}{\int_a^b |V_t - V_a|^2 \od t} + \bbce{\F_a}{\int_a^b \vartheta_a^2 |\sigma(S_t) - \sigma(S_a)|^2 \od t}\\
			& \lee \bbce{\F_a}{\int_{(a, b]} \int_{(a, t]} \od \<M\>_u \od t} + (b-a) \bbce{\F_a}{\int_a^b \!\int_a^t v_u^2 \od u \, \od t} +  c_{\eqref{eq:2:conditional-Y-1-step}}^2(b-a)^2 \vartheta_a^2 \sigma(S_a)^2\\
			& \lee \bbce{\F_a}{\int_{(a, b]} (b-u) (\od \<M\>_u + (b-a) v_u^2 \od u)} + 2c_{\eqref{eq:2:conditional-Y-1-step}}^2 (b-a)^2 (M_a^2 + V_a^2)\\
			& \lee \bbce{\F_a}{\int_{(a, b]} (b-u) (\od \<M\>_u + T v_u^2 \od u)} + 4c_{\eqref{eq:2:conditional-Y-1-step}}^2 \bbce{\F_a}{\int_a^b (b-u) \bigg(M_u^2  + T\int_0^u v_r^2 \od r\bigg)\od u}.
		\end{align*}
	Since $|\sigma|_{\Lip}$ is a factor of $c_{\eqref{eq:2:conditional-Y-1-step}}$ as shown in the proof of \cref{lem:sigma(Y)-estimate}, the assertion follows.
	\end{exam}
	
	The key assumption which enables to derive the approximation results is as follows.

	\begin{assum}\label{assumption-stochastic-integral}
		Let $\theta \in (0, 1]$. Assume that \cref{pre-assumption-stochastic-integral} is satisfied and  there is an a.s. non-decreasing process $\Theta \in \CL^+([0, T])$ such that the following two conditions hold:
		
		\begin{enumerate}[\quad \rm(a)]
			
			\item \label{item:growth-varphi}   \textit{$($Growth condition$)$} There is a constant $c_{\eqref{eq:growth-strategy}}>0$ such that 
			\begin{align} \label{eq:growth-strategy}
				|\vartheta_a| \lee c_{\eqref{eq:growth-strategy}} (T-a)^{\frac{\theta -1}{2}} \Theta_a \quad \mbox{a.s., } \forall a \in [0, T).
			\end{align}

			\item  \label{item:subordination-varphi} $($\textit{Curvature condition}$)$ For $\Phi : = \Theta \sigma(S)$, there is a constant $c_{\eqref{eq:thm:curvature-condition}}>0$ such that
			\begin{align}\label{eq:thm:curvature-condition}
				\bbce{\F_a}{\int_{(a, T)} (T-t)^{1- \theta} \Upsilon(\cdot, \od t)} \lee c_{\eqref{eq:thm:curvature-condition}}^2 \Phi_a^2 \quad \mbox{a.s., } \forall a \in [0, T).
			\end{align}
		\end{enumerate}
	Here, the constants $c_{\eqref{eq:growth-strategy}}$ and $c_{\eqref{eq:thm:curvature-condition}}$ may depend on $\theta$.
	\end{assum}
	
	The parameter $\theta$ in \eqref{eq:growth-strategy} describes the growth (pathwise and relatively to $\Theta$) of $\vartheta$ and the integrand $(T-t)^{1- \theta}$ in \eqref{eq:thm:curvature-condition} is employed to compensate the singularity of $\Upsilon$ when the time variable approaches $T$. Hence, the bigger $\theta$ is, the less singular at $T$ of both $\vartheta$ and $\Upsilon$ get, which leads the approximated stochastic integral to be more regular. In particular, for the Black--Scholes model with the delta-hedging strategy $\vartheta$, i.e., $\sigma(x) = x$ in \cref{exam:Upsilon-diffusion-setting}, the parameter $\theta$ can be interpreted as the \textit{fractional smoothness} of the payoff in the sense of \cite{GG04, GT15} where $\theta =1$ corresponds to the smoothness of order 1. Thus, in this article we will refer to the situation when \cref{assumption-stochastic-integral} holds for  $\theta =1$ as \textit{regular regime}. It is also clear that if \cref{assumption-stochastic-integral} is satisfied for a $\theta \in (0, 1]$, then it also holds, with the same $\Theta$, for any $\theta' \in (0, \theta)$ with a suitable change for $c_{\eqref{eq:growth-strategy}}$ and $c_{\eqref{eq:thm:curvature-condition}}$.
	
	Various specifications of \cref{assumption-stochastic-integral} in the Brownian setting or in the L\'evy setting are provided in \cite{GN20}. In \cref{sec:application-Levy-model} below, we consider \cref{assumption-stochastic-integral} in the exponential L\'evy setting which extends \cite{GN20}.

	\subsection{The basic method: Riemann approximation} \label{subsec:Riemann-approx}

	\begin{defi}\label{defi:time-net-Riemann-approx}
		\begin{enumerate}[(1)]
			\item Let $\cT_{\det}$ be the family of all \textit{deterministic} time-nets $\tau = (t_i)_{i=0}^n$ on $[0, T]$ with $0 = t_0 < t_1 < \cdots < t_n = T$, $n \gee 1$. The mesh size of $\tau = (t_i)_{i=0}^n \in \cT_{\det}$ is measured with respect to a parameter $\theta \in (0, 1]$ by
			\begin{align*}
				\|\tau\|_\theta : = \max_{i=1, \ldots, n}\frac{t_i - t_{i-1}}{(T-t_{i-1})^{1- \theta}}.
			\end{align*}
			
			\item For $\vartheta \in \adm$, $\tau = (t_i)_{i=0}^n \in \cT_{\det}$ and $t\in [0, T]$, we define
			\begin{align*}
				A^{\riem}_t(\vartheta, \tau): = \sum_{i=1}^n \vartheta_{t_{i-1}-}(S_{t_i \wedge t} - S_{t_{i-1}\wedge t}),\quad E^{\riem}_t(\vartheta, \tau): = \int_0^t \vartheta_{u-} \od S_u - A^{\riem}_t(\vartheta, \tau).
			\end{align*} 
		\end{enumerate}
	\end{defi}
	
	Below is the main result in this subsection.

	\begin{theo}\label{theo:approximation-QV-BMO} Let \cref{assumption-stochastic-integral} hold for some $\theta \in (0, 1]$. Then there exists a constant $c_{\eqref{eq:local-Riemann-approx}}>0$ such that for any $\tau \in \cT_{\det}$,
		\begin{align}\label{eq:local-Riemann-approx}
			\|E^{\riem}(\vartheta, \tau)\|_{\bmo^\Phi_2(\p)} \lee c_{\eqref{eq:local-Riemann-approx}}  \sqrt{\|\tau\|_\theta}.
		\end{align}
	\end{theo}

	Since the weighted $\bmo$- and weighted $\BMO$-norms of  $E^{\riem}(\vartheta, \tau)$ coincide when $S$ is continuous, we derive directly from \cref{theo:approximation-QV-BMO} and \cref{lemm:feature-BMO} the following result.

	\begin{coro}\label{coro:approximation-continuous-case-BMO} Let \cref{assumption-stochastic-integral} hold for some $\theta \in (0, 1]$. If $S$ is continuous, then the following assertions hold, where the constants $c_1, c_2, c_2', c_3 >0$ do not depend on $\tau$:
		\begin{enumerate}[\rm (1)]
			\item \label{item:1-coro:approximation-continuous-case-BMO} One has $\|E^{\riem}(\vartheta, \tau)\|_{\BMO^{\Phi}_2(\p)} \lee c_1 \sqrt{\|\tau\|_\theta}$ for any $\tau \in \cT_{\det}$.\\ Furthermore, if $\Phi \in \cSM_p(\p)$ for some $p \in [2, \infty)$, then for any $\tau \in \cT_{\det}$,
			\begin{align*}
				\|E^{\riem}(\vartheta, \tau)\|_{S_p(\p)} \lee c_2 \|E^{\riem}(\vartheta, \tau)\|_{\BMO^\Phi_p(\p)} \lee c'_2 \sqrt{\|\tau\|_\theta}.
			\end{align*}
			
			\item \label{item:2-coro:approximation-continuous-case-BMO} If $\bQ \in \RH_s(\p)$ for some $s \in (1, \infty)$ and $\Phi \in \cSM_2(\bQ)$, then for any $\tau \in \cT_{\det}$,
			\begin{align*}
				\|E^{\riem}(\vartheta, \tau)\|_{\BMO^{\Phi}_2(\bQ)} \lee c_3 \sqrt{\|\tau\|_\theta}.
			\end{align*}
		\end{enumerate}
	\end{coro}
	
	In particular, when $S$ is a geometric Brownian motion and $\vartheta$ is the delta-hedging strategy of a Lipschitz functional of $S_T$, then \cref{coro:approximation-continuous-case-BMO} gives the upper bound in \cite[Theorem 7]{Ge05}.

	\subsection{The jump correction method: General results} \label{sebsec:approx-correction}
	
	In \cref{coro:approximation-continuous-case-BMO}, the continuity of $S$ is crucial to derive the conclusions for $E^{\riem}$. If $S$ has jumps, then those results may fail as shown in the following example.
	
	\begin{exam} \label{counterexample} In the notations of \Cref{subsec:setting-stochastic-integal}, we let $Z = \tilde J$, where $\tilde J_t : = J_t - rt$ is a compensated Poisson process with intensity $r>0$. Choose $\sigma \equiv 1$ so that $S = Z$. Let  $f\colon (0, T] \times \bN \to \R$ be a Borel function with $\|f\|_\infty: =\sup_{(t, k) \in (0, T] \times \bN}|f(t, k)| <\infty$ and $\ep : = \inf_{t \in (0, T]}|f(t, 0)| > 0$. Assume that 
		\begin{align*}
			\delta: = \ep - r T \|f\|_\infty >0.
		\end{align*}
		
		Let $\rho_1: = \inf\{t>0 : \Delta J_t =1\}\wedge T$ and $\rho_2 : = \inf\{t> \rho_1 : \Delta J_t = 1\}\wedge T$.  Let $\vartheta_0 \in \R$ and define $\vartheta_t = \vartheta_0 + \int_{(0, t \wedge \rho_2]} f(s, J_{s-}) \od \tilde  J_s$, $t\in (0, T]$. It is not difficult to check that $\vartheta \in \adm$ is a martingale with $\|\vartheta_T\|_{L_\infty(\p)} <\infty$. Then \cref{pre-assumption-stochastic-integral} is satisfied with the selection $\Upsilon(\cdot, \od t) : = \od\<\vartheta\>_t$ as showed in \cref{exam:Upsilon-martingale-setting} (with $V \equiv 0$). In addition, it is straightforward to check that \cref{assumption-stochastic-integral} holds true for $\Theta \equiv \Phi \equiv 1$ and for $\theta =1$.

		Take $\tau = (t_i)_{i=0}^n \in \cT_{\det}$ arbitrarily. On the set $\{0< \rho_1 < \rho_2 < t_1\}$ we have
		\begin{align*}
			\big|\Delta E^{\riem}_{\rho_2}(\vartheta, \tau)\big| & = \sum_{i=1}^n |\vartheta_{\rho_2-} - \vartheta_{t_{i-1}-}| \1_{(t_{i-1}, t_i]}(\rho_2) |\Delta J_{\rho_2}|  \\
			& = |\vartheta_{\rho_2-} - \vartheta_0| = \bigg|f(\rho_1, J_{\rho_1-}) - r \int_{(0,\rho_2)} f(s, J_{s-})\od s\bigg|\\
			& \gee  |f(\rho_1, 0)| - r T  \|f\|_\infty \gee \delta.
		\end{align*}
		Since $\p(0< \rho_1 < \rho_2 < t_1)>0$, it implies that $\inf_{\tau \in \cT_{\det}} \|\Delta E^{\riem}_{\rho_2}(\vartheta, \tau)\|_{L_\infty(\p)} \gee \delta$. Due to \cref{lemm:relation-BMO-bmo}, we obtain $\inf_{\tau \in \cT_{\det}} \|E^{\riem}(\vartheta, \tau)\|_{\BMO_p(\p)} >0$ for any $p\in (0, \infty)$.
	\end{exam}

	Therefore, in order to exploit benefits of the weighted $\BMO$ approach to derive results as in \cref{coro:approximation-continuous-case-BMO} for jump models, we introduce a new approximation scheme based on an adjustment of the classical Riemann approximation. The time-net for this scheme is obtained by  combining a given deterministic time-net, which is used in the Riemann sum of the stochastic integral, and a suitable sequence of random times which captures the (relative) large jumps of the driving process. With this scheme, we not only can utilize the features of weighted $\BMO$, but can also control the cardinality of the combined time-nets.

	Let us begin with the random times. Due to the assumptions imposed on $S$ in \Cref{subsec:setting-stochastic-integal}, one has $\sigma(S_{-})>0$ and
	\begin{align}\label{eq:jump-S-Z}
		\Delta S = \sigma(S_-)\Delta Z
	\end{align}
	from which we can see that jumps of $S$ can be determined from knowing jumps of $Z$. However, if we would use $S$ to model the stock price process, then it is more realistic to track the jumps of $S$ rather than of $Z$.
	Therefore, we define the random times $\rho(\ep, \kappa) = (\rho_i(\ep, \kappa))_{i\gee 0}$ based on tracking the jumps of $S$ as follows (recall that $\inf \emptyset : = \infty$).
	
	\begin{defi}\label{defi:random-times}
		For $\ep >0$ and $\kappa \gee 0$, let $\rho_0(\ep, \kappa) :=0$ and
		\begin{align}
			\rho_{i}(\ep, \kappa) &: = \inf\left\{T\gee t >\rho_{i-1}(\ep, \kappa) : |\Delta S_t| > \sigma(S_{t-}) \ep(T-t)^\kappa\right\}\wedge T, \; i \gee 1, \label{eq:random-times-S} \\
			\cN_{\eqref{eq:cardinality-S}}(\ep, \kappa) &: = \inf\{i \gee 1 : \rho_i(\ep, \kappa) = T\}.\label{eq:cardinality-S}
		\end{align}
	\end{defi}

	The quantity $\ep(T-t)^\kappa$ in \eqref{eq:random-times-S} is the level at time $t$ from which we decide which jumps of $S$ are (relatively) large, and moreover, for $\kappa >0$, this level continuously shrinks when $t$ approaches the terminal time $T$. Hence, $\kappa$ describes the \textit{jump size decay rate}.   The idea for using the decay function $(T-t)^\kappa$ is to compensate the growth of integrands. By specializing $\kappa=0$, the control parameter $\ep$ can be interpreted as the \textit{jump size threshold}.

	\begin{defi}[Jump correction scheme]\label{defi:approximation-correction}
		Let $\ep>0$, $\kappa \in [0, \frac{1}{2})$ and $\tau = (t_i)_{i=0}^n \in \cT_{\det}$. 
		\begin{enumerate}[(1)]
			\item \label{item:defi:combined-net} Let $\comb{\tau}{\rho(\ep, \kappa)}$ be the (random) discretization times of $[0, T]$ by combining $\tau$ with $\rho(\ep, \kappa)$ and re-ordering their time-knots. 
			
			\item For $t\in [0, T]$, we define
			\begin{align}\label{eq:approximation-correction}
				\vartheta^\tau_t &:= \sum_{i=1}^n \vartheta_{t_{i-1}-}\1_{(t_{i-1}, t_i]}(t),\notag\\
				A^{\corr}_t(\vartheta, \tau | \ep, \kappa) &:= A^{\riem}_t(\vartheta, \tau)  +   \sum_{\rho_i(\ep, \kappa) \in [0, t] \cap [0, T)} \big(\vartheta_{\rho_i(\ep, \kappa)-}- \vartheta^\tau_{\rho_i(\ep, \kappa)}\big) \Delta S_{\rho_i(\ep, \kappa)},\\
				E_t^{\corr}(\vartheta, \tau | \ep, \kappa)  &:= \int_0^t \vartheta_{u-} \od S_u - A^{\corr}_t(\vartheta, \tau | \ep, \kappa),\notag
			\end{align}
			where $A^{\riem}(\vartheta, \tau)$ is given in \cref{defi:time-net-Riemann-approx}.
		\end{enumerate}
	\end{defi}

	As verified later in \Cref{subsec:proofs-sec3},  each $\rho_i(\ep, \kappa)$ is a stopping time. Moreover, in our setting the sum on the right-hand side of \eqref{eq:approximation-correction} is a finite sum a.s. as a consequence of \cref{lemm:cardinality-random-nets} below. Hence, by adjusting this sum on a set of probability zero, we may assume that $A^{\corr}(\vartheta, \tau | \ep, \kappa)  \in \CL_0([0, T])$. Besides, we also restrict the sum over the stopping times taking values in $[0, T)$ instead of $[0, T]$ because of two technical reasons. First, $\vartheta$ does not necessarily have the left-limit at $T$, and secondly, since $\Delta S_T =0$ a.s. as mentioned in \cref{rema:assumption-S-Z},  any value of the form $a \Delta S_T$ ($a \in \R$) added to the correction term does not affect the approximation in our context.

	To formulate main results in this section, we need to modify the weight processes. For $\Phi \in \CL^+([0, T])$ and $t\in [0, T]$, we define
	\begin{align}\label{defi-Phi-bar}
		\ol\Phi_t : = \Phi_t + \ts\sup_{s \in [0, t]} |\Delta \Phi_s|.
	\end{align}
	The reason to consider $\ol \Phi$ is that in the calculation below we will end up with $\Phi_{-}$ which is not c\`adl\`ag and therefore is not a candidate for a weight process. For $\ol \Phi$, it is clear that $\ol \Phi \in \CL^+([0, T])$ with $\Phi \vee \Phi_{-} \lee \ol \Phi$, and $\Phi \equiv \ol \Phi$ if and only if $\Phi$ is continuous. Moreover, \cref{lemm:properties-BMO-SM}\eqref{item:bar-Phi-SM} shows that  $\Phi \in \cSM_p(\p)$ implies  $\ol \Phi \in \cSM_p(\p)$.

	\begin{theo}\label{theo:conergence-rate-BMO} 
		Let \cref{assumption-stochastic-integral} hold for some $\theta \in (0, 1]$ and let $\Phi \in \cSM_2(\p)$.
		\begin{enumerate}[\rm (1)]
			\item \label{item:1:small-jump-intensity} If there is some $\alpha \in (0, 2]$ such that
			\begin{align}\label{eq:item:1:small-jump-intensity}
				c_{\eqref{eq:item:1:small-jump-intensity}}: = \ts \sup_{r \in (0, 1)}\big\|(\omega, t) \mapsto r^\alpha \int_{r < |z|\lee 1}   \nu_t(\omega, \od z)\big\|_{L_\infty(\Omega \times [0, T], \p \otimes \Leb)} <\infty,
			\end{align}
			then a constant $c_{\eqref{1eq:approx-BMO}}>0$ exists such that for all $\tau \in \cT_{\det}$, $\ep >0$,
			\begin{align}\label{1eq:approx-BMO}
				\big\|E^{\corr}\big(\vartheta, \tau \,\big|\, \ep, \tfrac{1- \theta}{2}\big)\big\|_{\BMO_2^{\ol \Phi}(\p)} \lee
				c_{\eqref{1eq:approx-BMO}}\max\left\{\ep, \sqrt{\|\tau\|_\theta}, \,h(\ep) \sqrt{ \|\tau\|_\theta} \right\},
			\end{align}
			where $h(\ep) = \ep^{1- \alpha}$ if $\alpha \in (1, 2]$, $h(\ep) = \log^+(1/\ep)$ if $\alpha =1$, and $h(\ep) = 1$ if $\alpha \in (0, 1)$.

			\item \label{item:2:jump-behavior} If \begin{align}\label{eq:item:2:jump-behavior}
				c_{\eqref{eq:item:2:jump-behavior}} : = \ts \sup_{r \in (0, 1)} \big\| (\omega, t) \mapsto \int_{r < |z| \lee 1}  z \nu_t(\omega, \od z) \big\|_{L_\infty(\Omega \times [0, T], \p \otimes \Leb)} < \infty,
			\end{align}
			then a constant $c_{\eqref{1eq:approx-BMO-symmetric}}>0$ exists such that for all $\tau \in \cT_{\det}$, $\ep >0$, 
			\begin{align}\label{1eq:approx-BMO-symmetric}
				\big\|E^{\corr}\big(\vartheta, \tau \,\big|\, \ep, \tfrac{1- \theta}{2}\big)\big\|_{\BMO_2^{\ol\Phi}(\p)} \lee
				c_{\eqref{1eq:approx-BMO-symmetric}}\max\left\{\ep, \sqrt{\|\tau\|_\theta}\right\}.
			\end{align}
		\end{enumerate}  
	\end{theo}
	We postpone the proof to \cref{sec:proofs-main-results}.
	Minimizing the right-hand side of \eqref{1eq:approx-BMO} and \eqref{1eq:approx-BMO-symmetric} over $\ep>0$ leads us to the  following:
	\begin{coro}\label{coro:conergence-rate-BMO} Let \cref{assumption-stochastic-integral} hold for some $\theta \in (0, 1]$ and let $\Phi \in \cSM_2(\p)$.
		\begin{enumerate}[\rm (1)]
			\item If \eqref{eq:item:1:small-jump-intensity} is satisfied for some $\alpha \in (0, 2]$, then there exists a constant $c'>0$ independent of $\tau$ such that, for $\ep(\tau, \alpha) : = \|\tau\|_\theta^{\frac{1}{2}(\frac{1}{\alpha} \wedge 1)}$, one has
			\begin{align*}
				\big\|E^{\corr}\big(\vartheta, \tau \,\big| \ep(\tau, \alpha), \tfrac{1- \theta}{2}\big)\big\|_{\BMO_2^{\ol\Phi}(\p)} \lee c' \begin{cases}
					\sqrt[2\alpha]{\|\tau\|_\theta} & \mbox{ if } \alpha \in (1, 2]\\
					\big[1 + \frac{1}{2}\log^+ \big(\tfrac{1}{\|\tau\|_\theta}\big)\big] \sqrt{\|\tau\|_\theta} & \mbox{ if } \alpha =1\\[2pt]
					\sqrt{\|\tau\|_\theta} & \mbox{ if } \alpha \in (0, 1).
				\end{cases}
			\end{align*}
			
			\item If \eqref{eq:item:2:jump-behavior} is satisfied, then $
				\big\|E^{\corr} \big(\vartheta, \tau \,|\, \sqrt{\|\tau\|_\theta}, \tfrac{1- \theta}{2}\big)\big\|_{\BMO_2^{\ol\Phi}(\p)} \lee c_{\eqref{1eq:approx-BMO-symmetric}} \sqrt{\|\tau\|_\theta}$.
		\end{enumerate}
		
	\end{coro}

	\begin{rema}
		\begin{enumerate}[\rm(1)]
			\item The assumption $j^Z_{\eqref{eq:uniformly-bounded-KV}}<\infty$  means that
			\begin{align}\label{eq:square-integrability-nu}
				\ts \left\|(\omega, t) \mapsto \int_{\R} z^2  \nu_t(\omega, \od z)\right\|_{L_\infty(\Omega \times [0, T], \p \otimes \Leb)} <\infty,
			\end{align}
			which implies that \eqref{eq:item:1:small-jump-intensity} automatically holds for $\alpha = 2$ in our context.
			
			\item It is easy to check that condition \eqref{eq:item:2:jump-behavior} is satisfied if: (finite variation property)
			\begin{align}\label{eq:finite-variation-prop}
				\ts \big\|(\omega, t) \mapsto \int_{|z|\lee 1}  |z| \nu_t(\omega, \od z)\big\|_{L_\infty(\Omega \times [0, T], \p \otimes \Leb)} <\infty,
			\end{align}
			or  if: (local symmetry property) there is an $r_0 \in (0, 1)$ such that the measure $\nu_t(\omega, \cdot)$ is symmetric  on $(-r_0, r_0)$ for $\p \otimes \Leb$-a.e. $(\omega, t) \in \Omega \times [0, T]$.
			
			\item If  \eqref{eq:item:1:small-jump-intensity} is satisfied for some $\alpha \in (0, 1)$, then \eqref{eq:finite-variation-prop}, and hence, \eqref{eq:item:2:jump-behavior} hold true. This assertion can be deduced from applying 
			\cref{lem:small-ball-property} $(\omega, t)$-wise with $\alpha < \gamma =1$.
		\end{enumerate}
	\end{rema}

	\subsection{A pure-jump model with faster convergence rates}\label{subsec:model-improve-rate}
	
	We investigate in this part the effect of small jumps of the underlying process $Z$ on the convergence rate of the approximation error resulting from the jump correction method.
	
	 \begin{assum}\label{assumption-pure-jump} In \Cref{subsec:setting-stochastic-integal}, we assume that
	 \begin{enumerate}[\quad \rm(a)]
	 	\item \label{item:assum-purejum-1} $Z^{\rm{c}} \equiv 0$, $b^Z \equiv 0$, i.e., $Z$ is a purely discontinuous martingale,
	 	\item \label{item:assum-purejum-2} The family of L\'evy measures $(\nu_t)_{t \in [0, T]}$ does not depend on $\omega$,
	 	\item \label{item:assum-purejum-3} $\sigma(x) = x$, i.e., $S$ is the Dol\'eans--Dade exponential of $Z$,
	 	\setcounter{newcounter}{\value{enumi}}
	 \end{enumerate}
 and assume in addition for some $\vartheta \in \adm$ that
 \begin{enumerate}[\quad \rm(a)]
 	\setcounter{enumi}{\value{newcounter}}
 	\item \label{item:assum-purejum-4} $\vartheta S = (\vartheta_t S_t)_{t \in [0, T)}$ has the semimartingale representation
 	\begin{align*}
 		\vartheta_t S_t = M_t + V_t, \quad t \in [0, T),
 	\end{align*}
 where $M = (M_t)_{t \in [0, T)}$ is an $L_2(\p)$-martingale and $V_t = \int_0^t v_u \od u$ for a progressively measurable $v$ with $\int_0^t v_u^2(\omega) \od u <\infty$ for $(\omega, t) \in \Omega \times [0, T)$.

\item \label{item:assum-purejum-5} \cref{assumption-stochastic-integral} holds for some $\theta \in (0, 1]$, non-decreasing $\Theta \in \CL^+([0, T])$ and for 
\begin{align*}
	\Upsilon(\cdot, \od t)  = \od \<M\>_t + M_t^2 \od t + v_t^2 \od t + \int_0^t v_u^2 \od u\, \od t, \quad t \in [0, T).
\end{align*}
Here, recall $\Phi = \Theta S$.
 \end{enumerate}
	\end{assum}
Let us shortly discuss conditions in \cref{assumption-pure-jump}. According to \cite[Ch.II, Theorem 4.15]{JS03}, condition \eqref{item:assum-purejum-2} implies that $Z$ has independent increments. Typical examples for $Z$ are time-inhomogeneous L\'evy processes, and especially, L\'evy processes.  This condition is necessary for our technique exploiting the orthogonality of martingale increments which appear in the approximation error with the integrand given in \eqref{item:assum-purejum-4}. Condition \eqref{item:assum-purejum-3} is merely a convenient condition and it can be extended to more general context with an appropriate modification for weight processes. We use here $\sigma(x) = x$ to make the proof more transparent and reduce unnecessary technicalities, and moreover, this is a classical case in application. Condition \eqref{item:assum-purejum-4} is the semimartingale decomposition of $\vartheta S$, and this can be easily verified if one knows the semimartingale representation of $\vartheta$. Condition \eqref{item:assum-purejum-5} enables to obtain convergence rates for the approximation error, and this structure of $\Upsilon$ is discussed in \cref{exam:Upsilon-martingale-setting}. We will show in \cref{theo:application-Levy} below that \cref{assumption-pure-jump} is fully satisfied in the exponential L\'evy model when $\vartheta$ is the mean-variance hedging strategy of a European type option. 
	
	The conditions for small jump behavior of the underlying process $Z$ are adapted respectively from \eqref{eq:item:1:small-jump-intensity} and  \eqref{eq:item:2:jump-behavior} to the current setting as follows:
	\begin{align}
		& \ts \sup_{r \in (0, 1)} \big\|t \mapsto r^\alpha \int_{r < |z| \lee 1} \nu_t(\od z)\big\|_{L_\infty([0, T], \Leb)} < \infty \quad \mbox{for some }  \alpha \in (0, 2], \label{eq:small-jump-pure-jump-1}\\
		& \ts \sup_{r \in (0, 1)} \big\|t \mapsto \int_{r < |z| \lee 1} z \nu_t(\od z)\big\|_{L_\infty([0, T], \Leb)} < \infty. \label{eq:small-jump-pure-jump-2}
	\end{align}
It is easy to check that the condition $\|t \mapsto \int_{|z| \lee 1} |z| \nu_t(\od z) \|_{L_\infty([0, T], \Leb)} <\infty$ implies both \eqref{eq:small-jump-pure-jump-1} (for $\alpha =1$) and \eqref{eq:small-jump-pure-jump-2}. We recall the modification $\ol \Phi$ of $\Phi$ from \eqref{defi-Phi-bar}.
	
	\begin{theo}\label{theo:improve-rate}
		Let \cref{assumption-pure-jump} hold for some $\theta \in (0, 1]$ and $\Phi  \in \cSM_2(\p)$. If \eqref{eq:small-jump-pure-jump-1} is satisfied for some $\alpha \in (0, 2]$, then   a constant $c>0$ exists such that for all $\tau \in \cT_{\det}$, $\ep>0$,
		\begin{align*}
			&\big\|E^{\corr}\big(\vartheta, \tau \,\big|\, \ep, \tfrac{1- \theta}{2}\big)\big\|_{\BMO_2^{\ol \Phi}(\p)} \\
			& \lee c\begin{cases}
				\max\Big\{\ep,\, \ep^{1 - \frac{\alpha}{2}} \sqrt{\|\tau\|_\theta}, \, \|\tau\|_\theta + \ep^{1- \alpha} \|\tau\|_\theta^{1- \frac{1}{2}(1- \theta)(\alpha -1)}\Big\} & \mbox{if } \alpha \in (1, 2],\\[3pt]
				\max\left\{\ep,\, \sqrt{\ep} \sqrt{\|\tau\|_\theta}, \, \big[1+ \log^+(\tfrac{1}{\ep}) + \log^+\big(\tfrac{1}{\|\tau\|_\theta}\big)\big]\|\tau\|_\theta\right\} & \mbox{if } \alpha =1,\\[3pt]
				\max\Big\{\ep,\, \sqrt{\ep} \sqrt{\|\tau\|_\theta}, \, \|\tau\|_\theta \Big\} & \mbox{if } \alpha =1 \mbox{ and } \eqref{eq:small-jump-pure-jump-2} \mbox{ holds},\\[3pt]
				\max\Big\{\ep,\, \sqrt{\ep} \sqrt{\|\tau\|_\theta}, \, \|\tau\|_\theta \Big\} & \mbox{if } \alpha \in (0, 1).
			\end{cases}
		\end{align*}
		In particular, minimizing the right-hand side over $\ep >0$ yields another constant $c' >0$  such that for all $\tau \in \cT_{\det}$, 
		\begin{align*}
			&\big\|E^{\corr}\big(\vartheta, \tau \,\big|\, \ep(\tau,\alpha ), \tfrac{1- \theta}{2}\big)\big\|_{\BMO_2^{\ol \Phi}(\p)} \lee c' \begin{cases}
				\|\tau\|_\theta^{\frac{1}{\alpha}(1- \frac{1}{2}(1-\theta)(\alpha -1))} & 
				\mbox{if } \alpha \in (1, 2], \\
				\big[1+ \log^+\big(\tfrac{1}{\|\tau\|_\theta}\big)\big]\|\tau\|_\theta & \mbox{if } \alpha = 1,\\
				\|\tau\|_\theta & \mbox{if } \alpha =1 \mbox{ and } \eqref{eq:small-jump-pure-jump-2} \mbox{ holds},\\
				\|\tau\|_\theta & \mbox{if } \alpha \in (0, 1),
			\end{cases}
		\end{align*}
	where $\ep(\tau, \alpha): = \|\tau\|_\theta^{\frac{1}{\alpha}(1- \frac{1}{2}(1-\theta)(\alpha -1))}$ if $\alpha \in [1, 2]$, and $\ep(\tau, \alpha): = \|\tau\|_\theta$ if $\alpha \in (0, 1)$.
	\end{theo}
	
	The proof will be provided later in \cref{sec:proofs-main-results}.
	
	\subsection{Adapted time-nets and approximation accuracy} \label{subsec:approx-accuracy} We discuss in this part how to improve the approximation accuracy by using suitable time-nets.

	\subsubsection*{Adapted time-net}
	The conclusions in \cref{theo:approximation-QV-BMO,theo:improve-rate},  \cref{coro:approximation-continuous-case-BMO,coro:conergence-rate-BMO}  assert that the errors measured in $\bmo_2^\Phi(\p)$ or in $\BMO_2^{\ol \Phi}(\p)$ are, up to multiplicative constants, upper bounded by  $\|\tau\|_\theta^r$  with $r \in [\frac{1}{4}, 1]$. Assume $\tau_n \in \cT_{\det}$ with $\#\tau_n = n+1$, where $n \gee 1$ represents in the context of stochastic finance the number of transactions in trading. If  one uses the equidistant nets $\tau_n = (Ti/n)_{i=0}^n$, then  $\|\tau_n\|_\theta = T^\theta/n^\theta$, and thus $\theta \in (0, 1]$ describes the convergence rate in this situation.  
	
	In order to accelerate the convergence rate we need to employ other suitable time-nets. First, it is straightforward to check that $\|\tau_n\|_\theta \gee T^\theta/n$ for any $\tau_n \in \cT_{\det}$ with $\#\tau_n = n+1$. Next, minimizing $\|\tau_n\|_\theta$  over $\tau_n \in \cT_{\det}$ with $\# \tau_n = n+1$ leads us to the following adapted time-nets, which were used (at least) in \cite{GG04, GGL13, Ge05, GT15, GT09}: For $\theta \in (0, 1]$ and $n\gee 1$, the \textit{adapted time-net} $\tau^\theta_n = (t^\theta_{i, n})_{i=0}^n$ is defined by 
	\begin{align*}
		t^\theta_{i, n} := T \big(1- (1-i/n)^{1/\theta}\big), \quad i = 1,\ldots, n.
	\end{align*} 
	Obviously, the equidistant time-net corresponds to $\theta =1$. By a computation, it holds that 
	\begin{align}\label{eq:theta-norm-time-net}
		T^\theta/n \lee \|\tau^\theta_n\|_\theta \lee T^\theta/(\theta n).
	\end{align}

	\subsubsection*{Cardinality of the combined time-net} The time-net used in \cref{theo:conergence-rate-BMO,theo:improve-rate} is $\comb{\tau}{\rho(\ep, \frac{1-\theta}{2})}$. Due to the randomness, a simple way to quantify the cardinality of this combined time-net is to evaluate its expected cardinality, i.e., $\E \[\#\comb{\tau}{\rho(\ep, \frac{1-\theta}{2})}\]$ (see, e.g., \cite{Fu11} or \cite[Eq. (10) with $\beta = 0$]{RT14}). Thus, we provide in the next result an estimate for certain moments of the cardinality. Since we aim to apply \cref{lemm:feature-BMO}\eqref{item:BMO-feature-RH} later, changes of the underlying measure are also taken into account.

	\begin{prop}\label{prop:cardinality-combined-nets}
		Let $q \in [1, 2]$, $r\in [2, \infty]$ with $\frac{q}{2} + \frac{1}{r} =1$. 
		Assume a probability measure $\bQ$ absolutely continuous with respect to $\p$ with $\od \bQ/\od \p \in L_r(\p)$. If \eqref{eq:item:1:small-jump-intensity} holds for some $\alpha \in (0, 2]$, then for any $\theta \in (0, 1]$ and $(\ep_n)_{n\gee 1} \subset (0, \infty)$ with $\inf_{n\gee 1}\sqrt[\alpha]{n} \ep_n >0$, there exists a constant $c_{\eqref{eq:norm-combined-nets}}>0$ such that for any $n\gee 1$, any $\tau_n \in \cT_{\det}$ with $\#\tau_n = n+1$ one has
		\begin{align}\label{eq:norm-combined-nets}
			\left\|\#\comb{\tau_n}{ \rho\(\ep_n, \tfrac{1-\theta}{2}\)}\right\|_{L_q(\bQ)} \sim_{c_{\eqref{eq:norm-combined-nets}}} n.
		\end{align}
	\end{prop}
	
	
	Plugging the adapted time-nets $\tau^\theta_n$ into previous results, we derive the following.
	\begin{theo} \label{theo:BMO-convergent-rate}
		Let \cref{assumption-stochastic-integral} hold true for some $\theta \in (0, 1]$.
		\begin{enumerate}[\rm (1)]
			\itemsep0.3em
			
			\item \label{item:coro:bmo-convergent-rate} One has
			$\sup_{n \gee 1} \sqrt{n} \|E^{\riem}(\vartheta, \tau^\theta_n)\|_{\bmo_2^\Phi(\p)} <\infty$.
		
			\item \label{item:1:coro:BMO-convergent-rate} Assume $\Phi \in \cSM_2(\p)$. 
			\begin{enumerate}[\rm (a)]
				\item \label{item:rate-general-alpha-(0,2]} If \eqref{eq:item:1:small-jump-intensity} holds for some $\alpha\in (0, 2]$, then 
				$$\ts \sup_{n \gee 1}  R(n)	 \big\|E^{\corr}\big(\vartheta, \tau^{\theta}_n \,\big|\, \ep_n , \tfrac{1-\theta}{2}\big)\big\|_{\BMO_2^{\ol\Phi}(\p)} <\infty,$$
				where $R(n) = 1/\ep_n = \sqrt[2 \alpha]{n}$ if $\alpha \in (1, 2]$, $R(n) = \sqrt{n}/(1 + \log n) $ and $\ep_n = \sqrt{1/n}$ if $\alpha = 1$, and $R(n) = 1/\ep_n = \sqrt{n}$ if $\alpha \in (0, 1)$.
				
				\item \label{item:rate-general-symmetry} If \eqref{eq:item:2:jump-behavior} holds, then 
				$$\ts \sup_{n \gee 1} \sqrt{n}	\big\|E^{\corr}\big(\vartheta, \tau^{\theta}_n \,\big| \sqrt{1/n}, \tfrac{1-\theta}{2}\big)\big\|_{\BMO_2^{\ol\Phi}(\p)} <\infty.$$
			\end{enumerate}

			\item \label{item:2:coro:BMO-convergent-rate} Assume \cref{assumption-pure-jump} and $\Phi \in \cSM_2(\p)$. If \eqref{eq:small-jump-pure-jump-1} holds for some $\alpha \in (0, 2]$, then $$\ts \sup_{n \gee 1}  R(n)	 \big\|E^{\corr}\big(\vartheta, \tau^{\theta}_n \,\big|\, \ep_n , \tfrac{1-\theta}{2}\big)\big\|_{\BMO_2^{\ol\Phi}(\p)} <\infty,$$
			where 
			\begin{align*}
				\begin{cases}
					R(n) = 1/\ep_n =  n^{\frac{1}{\alpha}(1-\frac{1}{2}(1- \theta)(\alpha -1))} & \mbox {if } \alpha \in (1, 2],\\
					R(n) = n/(1 + \log n), \ep_n = 1/n & \mbox {if } \alpha = 1,\\
					R(n) = 1/\ep_n = n & \mbox{if } \alpha = 1 \mbox{ and } \eqref{eq:small-jump-pure-jump-2} \mbox{ holds},\\
					R(n) = 1/\ep_n = n & \mbox{if } \alpha \in (0, 1).
				\end{cases}
			\end{align*}
			
			\item \label{item:Lp-estimate-BMO} If in addition $\Phi \in \cSM_p(\p)$ for some $p >2$, then the conclusions in Items \eqref{item:1:coro:BMO-convergent-rate}--\eqref{item:2:coro:BMO-convergent-rate} hold case-wise for the $\BMO_p^{\ol\Phi}(\p)$-norm and, consequently, for the $S_p(\p)$-norm, in place of the $\BMO_2^{\ol\Phi}(\p)$-norm.
			
			\item \label{item:change-measure-BMO} If in addition  $\bQ \in \RH_s(\p)$ for some $s\in (1, \infty)$ and $\Phi \in \cSM_2(\bQ)$, then the conclusions in Items \eqref{item:1:coro:BMO-convergent-rate}--\eqref{item:2:coro:BMO-convergent-rate} hold case-wise for the $\BMO_2^{\ol\Phi}(\bQ)$-norm in place of the $\BMO_2^{\ol\Phi}(\p)$-norm.
		\end{enumerate}
	\end{theo}
	
	\begin{proof}
		Item \eqref{item:coro:bmo-convergent-rate} (resp. \eqref{item:1:coro:BMO-convergent-rate}, \eqref{item:2:coro:BMO-convergent-rate}) follows from \cref{theo:approximation-QV-BMO} (resp. \cref{theo:conergence-rate-BMO}, \cref{theo:improve-rate}) and \eqref{eq:theta-norm-time-net}. Items \eqref{item:Lp-estimate-BMO}--\eqref{item:change-measure-BMO} are due to  \cref{lemm:feature-BMO,lemm:properties-BMO-SM}\eqref{item:bar-Phi-SM}.
	\end{proof}
	
	\begin{rema}\label{rema:factor-n}
		Lets us consider the parameter $n$ in the estimates in \cref{theo:BMO-convergent-rate}. 
		\begin{enumerate}[\rm (1)]
			\item  For Items \eqref{item:rate-general-alpha-(0,2]}, \eqref{item:2:coro:BMO-convergent-rate} and \eqref{item:Lp-estimate-BMO}, applying \cref{prop:cardinality-combined-nets} with $q=2$, $r=\infty$ and $\bQ = \p$ we find that  the parameter $n$ in front of the $\bmo_2^\Phi(\p)$-, $\BMO_p^{\ol \Phi}(\p)$- or $S_p(\p)$-norm ($p \in [2, \infty)$) can be regarded, up to a multiplicative constant, as the $L_2(\p)$-norm of the cardinality of the combined time-net used in the corresponding approximations.
			
			\item For Item \eqref{item:rate-general-symmetry}, we apply \cref{prop:cardinality-combined-nets} with $q=2$, $r=\infty$, $\bQ = \p$ and $\alpha =2$ (with keeping \eqref{eq:square-integrability-nu} in mind) to obtain the same interpretation for $n$.
			
			\item  For Item \eqref{item:change-measure-BMO}, thanks to \cref{prop:cardinality-combined-nets} (choose $q=1$, $r=2$), if $s\in [2, \infty)$, then this interpretation is still valid after a change of measure: The parameter $n$ in front of the $\BMO_2^{\ol \Phi}(\bQ)$-norm can be considered as the expected cardinality of the time-net \textit{under} $\bQ$. 	
		\end{enumerate} 
	\end{rema}

	\section{Applications to exponential L\'evy models} \label{sec:application-Levy-model}
	
	We provide several examples for \cref{assumption-stochastic-integral,assumption-pure-jump} in the exponential L\'evy setting so that the main results can be applied. As an important step to obtain them, we establish in \cref{theo:MVH-strategy} an explicit form for the mean-variance hedging strategy of a general European type option, and this formula might also have an independent interest.

	\subsection{L\'evy process}\label{subsection-exponential-levy-process}
	Let $X=(X_t)_{t \in [0, T]}$ be a one-dimensional L\'evy process on $(\Omega, \F, \p)$, i.e., $X_0 = 0$, $X$ has independent and stationary increments and $X$ has c\`adl\`ag paths. Let $\bF^X = (\F^X_t)_{t \in [0, T]}$ be the augmented natural filtration of $X$, and assume that $\F = \F_T^X$. According to the L\'evy--Khintchine formula (see, e.g., \cite[Theorem 8.1]{Sa13}), there is a \textit{characteristic triplet} $(\gamma, \sigma, \nu)$, where $\gamma \in \R$, coefficient of Brownian component $\sigma \gee 0$, L\'evy measure $\nu\colon \cB(\R) \to [0, \infty]$ (i.e., $\nu(\{0\}):=0$ and $\int_{ \R}(x^2 \wedge 1) \nu(\od x) <\infty$), such that  the \textit{characteristic exponent $\psi$} of $X$ defined by $\E \e^{\im u X_t} = \e^{- t \psi(u)}$ is of the form
	\begin{align*}
		\psi(u)=- \im \gamma u + \frac{\sigma^2 u^2}{2} - \int_{\R}\(\e^{\im u x}-1-\im ux \1_{\{|x| \lee 1\}}\)\nu(\od x), \quad u\in \R.
	\end{align*}

	
	\subsection{Mean-variance hedging (MVH)} \label{subsec:MVH}
	Assume that the underlying price process is modelled by the exponential $S = \e^X$. Since models with jumps correspond to incomplete markets in general, there is no ``optimal'' hedging strategy which replicates a payoff at maturity and eliminates risks completely. This leads to consider certain strategies that minimize some types of risk. Here, we use  quadratic hedging which is a common approach, see \cite{Sc01}. To simplify the quadratic hedging problem, we consider the martingale market. Applications of results in \cref{sec:main-results} for L\'evy markets under the semimartingale setting are studied  in \cite{TN20b}.
	
	\begin{assum}\label{assum-maringale-setting}
		$S=\e^X$ is an $L_2(\p)$-martingale and is not a.s. constant.
	\end{assum}

	Under \cref{assum-maringale-setting}, any $\xi \in L_2(\p)$ admits the \textit{Galtchouk--Kunita--Watanabe (GKW) decomposition} 
	\begin{align}\label{eq:GKW-decomposition}
		\xi = \E  \xi + \int_0^T \theta_{t}^\xi \od S_t + L_T^\xi,
	\end{align}
	where $\theta^{\xi}$ is predictable with $\E \int_0^T |\theta^\xi_t|^2 S_{t-}^2 \od t <\infty$, $L^{\xi} = (L^{\xi}_t)_{t\in [0, T]}$ is an $L_2(\p)$-martingale with zero mean and satisfies $\<S, L^{\xi}\>=0$. The integrand $\theta^{\xi}$ is called the \textit{MVH strategy} for $\xi$, which is unique in $L_2(\p\otimes \Leb, \Omega \times [0, T])$. The reader is referred to \cite{Sc01} for further discussion.
	
	Our aim is to apply the approximation results obtained in \cref{sec:main-results} for the stochastic integral term in \eqref{eq:GKW-decomposition}, which can be interpreted in mathematical finance as the hedgeable part of $\xi$. 	To do that, one of our main tasks is to find a representation of $\theta^{\xi}$ which is convenient for verifying assumptions in \cref{sec:main-results}. This issue is handled in \Cref{subsec:MVH-strategy} in which we focus on the European type options $\xi = g(S_T)$.

	\subsection{Explicit MVH strategy} \label{subsec:MVH-strategy} In the literature, there are several methods to determine an explicit form for the MVH strategy of a European type option $g(S_T)$. Let us mention some typical approaches for which the martingale representation of $g(S_T)$ plays the key role. A classical method is by using directly It\^o's formula (e.g., \cite{CTV05,JMP00}) which requires a certain smoothness of $(t, y) \mapsto \E g(y S_{T-t})$. Another idea is based on Fourier analysis to separate the payoff function $g$ and the underlying process $S$ (e.g., \cite{BT11, HKK06, Ta10}). To do that, some regularity for $g$ and $S$ is assumed. As a third method, one can use Malliavin calculus  to determine the MVH strategy (e.g., \cite{BNLOP03}), however the payoff $g(S_T)$ is assumed to be differentiable in the Malliavin sense so that the Clark--Ocone formula is applicable.
	
	To the best of our knowledge, \cref{theo:MVH-strategy} below is new and it provides an explicit formula for the MVH strategy of $g(S_T)$ without requiring any
	regularity from the payoff function $g$ nor any specific structure of the underlying
	process $S$. The proof is given  in \cref{app-sec:proof-MVH-strategy} by exploiting Malliavin calculus. Recall that $\sigma$ and $\nu$ are the coefficient of the Brownian component and the L\'evy measure of $X$ respectively.

	\begin{theo}\label{theo:MVH-strategy} Assume \cref{assum-maringale-setting}. For a Borel function $g\colon \R_+ \to \R$ with $g(S_T) \in L_2(\p)$, there exists a  $\vartheta^g  \in \CL([0, T))$ such that the following assertions hold:
		\begin{enumerate}[\rm (1)]
			\itemsep0.3em
			\item \label{item:prop:strategy} $\vartheta^g_-$ is a MVH strategy of $g(S_T)$;
			\item \label{item:prop:maringale-property} $\vartheta^gS$ is an $L_2(\p)$-martingale and $\vartheta^g_t = \vartheta^g_{t-}$ a.s. for each $t\in [0, T)$;
			\item \label{item:prop:gradient-formula-general} For any $t\in (0, T)$, one has, a.s.,
			\begin{align}\label{eq:integrand-general}
				\vartheta^g_t = \frac{1}{c_{\eqref{eq:integrand-general}}^2}\bigg[\sigma^2 \pd_y G(t, S_t) + \int_{ \R} \frac{G(t, \e^{x} S_t)- G(t, S_t)}{S_t} (\e^x -1)\nu(\od x)\bigg],
			\end{align}
			where $c_{\eqref{eq:integrand-general}}^2  = \sigma^2 + \int_{\R}(\e^x -1)^2 \nu(\od x)$ and $G(t, \cdot) \colon \R_+ \to \R$ is as follows:
			\begin{enumerate}[\rm (a)]
				\item \label{item:sigma>0} If $\sigma >0$, then we choose $G(t, y): = \E g(y S_{T-t})$;
				\item \label{item:sigma=0} If $\sigma =0$, then we choose  $G(t, \cdot)$ such that it is Borel measurable and $G(t, S_t) = \ce{\F_t}{g(S_T)}$ a.s., and we set  $\pd_y G(t, \cdot):=0$ by convention.
			\end{enumerate} 
		\end{enumerate}
	\end{theo}
	\cref{assum-maringale-setting} ensures that $c_{\eqref{eq:integrand-general}} \in (0, \infty)$.   For the case \eqref{item:sigma>0}, the function $G(t, \cdot)$ has derivatives of all orders on $\R_+$ due to the presence of the Gaussian component of $X$, see \cite[Theorem 9.13]{GN20}. Formula \eqref{eq:integrand-general} was also established in \cite[Section 4]{CTV05} and in \cite[Proposition 7]{Ta10} under some extra conditions for $g$ and $S$. A similar formula of \eqref{eq:integrand-general} in a general setting can be found in \cite[Theorem 2.4]{JMP00}.
	
	\subsection{Growth of the MVH strategy and weight regularity} \label{subsec:MVH-growth-weight} 
	
	We investigate the growth in time of $\vartheta^g$ obtained in \eqref{eq:integrand-general} pathwise and relatively to a weight process for H\"older continuous or bounded functions $g$. This growth property is examined in connection to the small jump behavior of the underlying L\'evy process.

	\begin{defi}\label{defi:holde-stable}
		\begin{enumerate}[\rm (1)]
			\item \label{defi:Holder-spaces}(\textit{H\"older spaces}) Let $\emptyset \neq U \subseteq \R$ be an open interval and let $\eta \in [0, 1]$. For a Borel function $f\colon U \to \R$, we define
			\begin{align*}
				|f|_{C^{0, \eta}(U)} : = \inf\{ c \in [0, \infty) : |f(x)- f(y)| \lee c|x-y|^\eta \mbox{ for all } x, y \in U, x\neq y\},
			\end{align*}
		and let $f \in C^{0, \eta}(U)$ if $|f|_{C^{0, \eta}(U)} <\infty$. It is clear that, on $U$, the space $C^{0, 1}(U)$ consists of all Lipschitz functions, $C^{0, \eta}(U)$ contains all $\eta$-H\"older continuous functions for $\eta \in (0, 1)$, and $C^{0, 0}(U)$ consists of all bounded (not necessarily continuous) Borel functions.
			
			\item \label{defi:alpha-stable} (\textit{$\alpha$-stable-like L\'evy measures}) For $\alpha \in (0, 2)$ and a L\'evy measure $\nu$, we let $\nu \in \sS(\alpha)$ if  $\nu = \nu_1 + \nu_2$, where $\nu_1, \nu_2$ are L\'evy measures that satisfy
			\begin{align*}
				\nu_1(\od x) = \frac{k(x)}{|x|^{\alpha +1}}\1_{\{x \neq 0\}} \od x,\quad \limsup_{|u| \to \infty} \frac{1}{|u|^\alpha} \int_{\R} (1-\cos (ux)) \nu_2(\od x) <\infty,
			\end{align*}
			where $0 < \liminf_{x \to 0} k(x) \lee \limsup_{x \to 0} k(x) <\infty$, and the function $x\mapsto k(x)/|x|^{\alpha}$ is non-decreasing on $(-\infty, 0)$ and is non-increasing on $(0, \infty)$.
		\end{enumerate}
	\end{defi}

	\begin{rema}[see also \cite{TN20b}, Lemma A.1]\label{remark:Holder-stable-like} 
		Let $\nu$ be a L\'evy measure and $\alpha \in (0, 2)$.
		\begin{enumerate}[(1)]	
			\itemsep0.3em

			\item \label{item:rema:conditions-C-alpha-1} If $\nu \in \sS(\alpha)$, then $\sup_{r \in (0, 1)} r^\alpha \int_{r < |x| \lee 1} \nu(\od x) <\infty$. Moreover, $\alpha$ is equal to the  \textit{Blumenthal--Getoor index} of $\nu$ due to \cite[Theorem 3.2]{BG61}, i.e., $\alpha = \inf\{q \in [0, 2] : \int_{|x|\lee 1} |x|^q \nu(\od x) <\infty\}$. 
			
			\item \label{item:rema:suff-conditions-S-alpha-1}
			One has $\nu \in \sS(\alpha)$ if $\nu$ has a density $p(x) := \nu(\od x)/\od x$ satisfying
			\begin{align*}
				0 < \liminf_{|x| \to 0} |x|^{1+ \alpha} p(x) \lee \limsup_{|x| \to 0} |x|^{1+ \alpha} p(x) <\infty.
			\end{align*}
		\end{enumerate}
	\end{rema}

	\begin{exam} We provide typical examples in mathematical finance using $C^{0, \eta}(U)$-payoff functions and $\alpha$-stable-like processes. 
		\begin{enumerate}[(1)]
			\itemsep0.3em
			
			\item Let $K>0$. The binary payoff $g_0(x) := \1_{(K, \infty)}(x)$ belongs to $C^{0, 0}(\R_+)$ obviously, the call payoff $g_1(x) := (x-K)^+$ is contained in $C^{0, 1}(\R_+)$, and for $\eta \in (0, 1)$, the \textit{powered call} payoff (see, e.g., \cite{HKK06}) $g_\eta(x) := ((x-K)^+)^\eta$ belongs to $C^{0, \eta}(\R_+)$.
			
			\item  The CGMY process  with parameters $C, G, M >0$ and $Y \in (0, 2)$ (see \cite[Section 5.3.9]{Sc03}) has the L\'evy measure $$\nu_{\mathrm{CGMY}}(\od x) = C \frac{\e^{Gx}\1_{\{x<0\}} + \e^{-Mx}\1_{\{x>0\}} }{|x|^{1+ Y}} \1_{\{x \neq 0\}} \od x$$ 
			which belongs to $\sS(Y)$ due to \cref{remark:Holder-stable-like}\eqref{item:rema:suff-conditions-S-alpha-1}. \\
			The Normal Inverse Gaussian (NIG) process (see \cite[Section 5.3.8]{Sc03}) has the L\'evy density $p_{\mathrm{NIG}}(x) : = \nu_{\mathrm{NIG}}(\od x)/\od x$ that satisfies $$\ts 0< \liminf_{|x| \to 0} x^2 p_{\mathrm{NIG}}(x) \lee \limsup_{|x|\to 0} x^2 p_{\mathrm{NIG}}(x) <\infty.$$
			Hence, \cref{remark:Holder-stable-like}\eqref{item:rema:suff-conditions-S-alpha-1} verifies that $\nu_{\mathrm{NIG}} \in \sS(1)$.
		\end{enumerate}
		
	\end{exam}
	
	Before stating the main result of this part, let us introduce the relevant weight processes. For $\eta \in [0, 1]$, define processes $\Theta(\eta), \Phi(\eta) \in \CL^+([0, T])$ by setting
	\begin{align}\label{eq:definiteion-weight-process}
		\ts \Theta(\eta)_t:=  \sup_{u \in [0, t]} (S_u^{\eta-1}) \quad \mbox{and} \quad \Phi(\eta)_t : = \Theta(\eta)_t S_t.
	\end{align}
	\cref{theo:application-Levy} below verifies \cref{assumption-stochastic-integral,assumption-pure-jump} in the exponential L\'evy setting, and its proof is given later in \Cref{subsec:proof-sec-4}.
	
	\begin{theo}\label{theo:application-Levy} Assume \cref{assum-maringale-setting}.
		Let $\eta \in [0, 1]$ and $g\in C^{0, \eta}(\R_+)$.
		\begin{enumerate}[\rm (1)]
			\item $($\textit{Weight regularity}$)$ \label{item: weight-regularity-Levy} One has $\Phi(\eta) \in \cSM_2(\p)$.
			
			\item $($\textit{MVH strategy growth}$)$ \label{item:strategy-growth} There is a constant $c_{\eqref{eq:thm:approx-bmo}} >0$ such that, for $\vartheta^g$ given in \eqref{eq:integrand-general},
			\begin{align}\label{eq:thm:approx-bmo}
				|\vartheta^g_t| \lee c_{\eqref{eq:thm:approx-bmo}} U(t) S_t^{\eta-1} \quad \mbox{a.s., } \forall t \in [0, T),
			\end{align}
			where the function $U(t)$ is provided in \cref{tab:theta-cases}.

			\item  \label{item:example-assumprion-stoc-inte} Denote $M: = \vartheta^g S$. Then \cref{assumption-stochastic-integral} holds true for
			\begin{align*}
				\vartheta =\vartheta^g,\quad \Upsilon(\cdot, \od t)  = \od \<M\>_t + M_t^2 \od t, \quad \Theta = \Theta(\eta), \quad \Phi = \Phi(\eta)
			\end{align*}
			and for $\theta$ provided in \cref{tab:theta-cases} accordingly. In particular, if $\sigma = 0$ (i.e., $X$ does not have a Brownian component), then  \cref{assumption-pure-jump} is satisfied.
		\end{enumerate}
			\vspace{-.3em}
		\begin{longtable}[c]{|l|l|l|l|l|}
			\caption{Conclusions for $U(t)$ and $\theta$}
			\label{tab:theta-cases}\\
			\hline
			& $\sigma$ and $\eta$  & \begin{tabular}[c]{@{}l@{}} Small jump condition \\
				for $X$
			\end{tabular}   & Function $U(t)$ & Values of  $\theta$\\ \hline
			\endfirsthead
			\endhead
			$\mathrm{A1}$ & \begin{tabular}[c]{@{}l@{}} $\sigma>0$ \\
				$\eta \in [0, 1]$
			\end{tabular}   &              &   $U(t) = (T-t)^{\frac{\eta -1}{2}}$  &     \begin{tabular}[c]{@{}l@{}}    $\theta = 1$ if $\eta =1$, \\
				$\forall \theta \in (0, \eta)$ if $\eta \in (0, 1)$
			\end{tabular}   \\ \hline
			$\mathrm{A2}$ &     \begin{tabular}[c]{@{}l@{}} $\sigma=0$ \\
				$\eta \in [0, 1]$
			\end{tabular}     &    $\int_{|x|\lee 1} |x|^{1+ \eta} \nu(\od x)<\infty$             & $U(t) = 1$      &    $\theta =1$      \\ \hline
			$\mathrm{A3}$ &     \begin{tabular}[c]{@{}l@{}} $\sigma=0$ \\
				$\eta \in [0,1)$
			\end{tabular}    &  \begin{tabular}[c]{@{}l@{}} $\nu \in \sS(\alpha)$\\ for some $\alpha \in [1+ \eta, 2)$ \end{tabular}             &   \begin{tabular}[c]{@{}l@{}}   $U(t) = (T-t)^{\frac{1+ \eta}{\alpha} -1}$ \Tstrut\\
			\qquad  if $\alpha \in (1+ \eta, 2)$, \\
				$U(t) = \max\{1, \log\frac{1}{T-t}\}$\\
				\qquad  if $\alpha = 1+ \eta$
			\end{tabular}  & 
		\begin{tabular}[c]{@{}l@{}}   $\forall\theta \in\(0, \frac{2(1+\eta)}{\alpha} -1\)$
		\end{tabular}   \\ \hline
		\end{longtable}
	\end{theo}

	
Eventually, let us turn to the approximation problem in the exponential L\'evy setting.  Although results in \cref{sec:main-results} are stated in terms of the characteristic of $Z$ (the integrator of the SDE \eqref{eq:SDE-levy}), the result below are formulated involving the characteristic of the log price process $X$. This is slightly more convenient to verify in practice. Based on the relation between $X$ and $Z$ provided in \Cref{subsec:exponential-Levy}, we can easily translate conditions imposed on $X$ to $Z$ and vice versa.
	
	For $\alpha \in (0, 2]$, we let 
	\begin{align*}
		\ts \nu \in \sUS(\alpha) \Longleftrightarrow \sup_{r \in (0, 1)} r^{\alpha} \int_{r < |x| \lee 1} \nu(\od x) <\infty.
	\end{align*}
Combing \cref{theo:application-Levy} with \cref{theo:BMO-convergent-rate} yields the following corollary where the Brownian part of $X$ is also taken into account.
	
	\begin{coro}\label{coro:convergence-rate-levy}
		Assume \cref{assum-maringale-setting} and let $\eta \in [0, 1]$, $g \in C^{0, \eta}(\R_+)$. 
		\begin{enumerate}[\rm (1)]
			\item \label{item:1:coro:convergence-rate-levy}
			For $\vartheta^g$ given in \eqref{eq:integrand-general} and for $$\ts \Phi(\eta) = (\sup_{u \in [0, t]}(S^{\eta -1}_u) S_t)_{t \in [0, T]},\quad \ol \Phi(\eta) = \Phi(\eta) + \sup_{s \in [0, \cdot]}|\Delta \Phi(\eta)_s|,$$
			 one has
			\begin{align}\label{eq:coro:convergence-rate}
				\sup_{n \gee 1} R(n) \big\|E^{\corr}\big(\vartheta^g, \tau^\theta_n\, \big|\, \ep_n, \tfrac{1-\theta}{2}\big)\big\|_{\BMO_2^{\ol \Phi(\eta)}(\p)} <\infty,
			\end{align}
			where $\theta$, $R(n)$ and $\ep_n$ are provided in \cref{tab:rates}:
			\vspace{-.3em}
			\begin{longtable}{|l|l|l|l|}
				\caption{Convergence rate $R(n)$ and jump size threshold $\ep_n$}
				\label{tab:rates}
				\vspace{-.5em}\\
				\hline
				& Interplay between $g$ and $X$ & Values of  $\theta$ & $R(n)$ and $\ep_n$  \\ \hline
				\endfirsthead
				\endhead
				$\mathrm{B1}$ & \begin{tabular}[c]{@{}l@{}} 
					$\sigma =0$ and\\
					$\nu \in \sUS(\alpha)$ for some  \\
					$(\eta,\alpha)  \in ([0, 1) \times (0, 1+ \eta))$\\
					\qquad \qquad $\cup (\{1\} \times (0, 2])$
				\end{tabular}    & $\theta =1$    &      \begin{tabular}[c]{@{}l@{}}   $R(n) = 1/\ep_n = \sqrt[\alpha]{n}$  if $\alpha \in (1, 2]$,  \\
					$R(n) = n/(1+ \log n)$, $\ep_n = 1/n$\\
					\qquad  if $\alpha =1$,\\
					$R(n) = 1/\ep_n = n$ if $\alpha \in (0, 1)$   \end{tabular}   \\ \hline
				$\mathrm{B2}$ & \begin{tabular}[c]{@{}l@{}}
					$\sigma =0$ and\\
					$\nu \in \sS(\alpha)$ for some\\  $(\eta, \alpha) \in [0, 1) \times [1+ \eta, 2)$ \end{tabular}    &     $\forall\theta \in \Big(0, \frac{2(1+\eta)}{\alpha} -1\Big)$       &    \begin{tabular}[c]{@{}l@{}}  	$R(n) = 1/\ep_n = n^{\frac{1}{\alpha}(1- \frac{1}{2}(1- \theta)(\alpha -1))}$ \Tstrut \\
					\qquad if $(\eta, \alpha) \neq (0, 1)$,  \\
					$R(n) = n/(1+ \log n)$, $\ep_n = 1/n$ \\
					\qquad if $(\eta, \alpha) = (0, 1)$	
				\end{tabular} \\ \hline
				$\mathrm{B3}$ & \begin{tabular}[c]{@{}l@{}}  	$\sigma >0$ and $\eta =1$
				\end{tabular} & $\theta = 1$ & $R(n) = 1/\ep_n =  \sqrt{n}$
				\\ \hline 
				$\mathrm{B4}$ &	\begin{tabular}[c]{@{}l@{}}  	$\sigma >0$, $\eta \in (0, 1)$ and\\
					$\nu \in \sUS(\alpha)$ for some 
					$\alpha \in (0, 2]$ 
				\end{tabular} & $\forall \theta \in (0, \eta)$ &  \begin{tabular}[c]{@{}l@{}}  	$R(n) = 1/\ep_n = \sqrt{n}$ if $\alpha \in (0, \frac{3 - \theta}{2- \theta}]$,\\[2pt]
					$R(n) = 1/\ep_n = n^{\frac{1}{\alpha}(1- \frac{1}{2}(1- \theta)(\alpha -1))}$\\
					\qquad  if $\alpha \in (\frac{3- \theta}{2- \theta}, 2]$
				\end{tabular} 	\\ \hline 
			\end{longtable}
			\vspace{-.5em}
			
		\item\label{item:2:coro:convergence-rate-levy} 		If  $\int_{|x| >1} \e^{px} \nu(\od x) <\infty$ for some $p >2$, then \eqref{eq:coro:convergence-rate} holds for the  $\BMO^{\ol \Phi(\eta)}_p(\p)$-norm and, consequently, for the $S_p(\p)$-norm, in place of the $\BMO^{\ol \Phi(\eta)}_2(\p)$-norm.
		\end{enumerate} 
	\end{coro}	
	
According to \cref{prop:cardinality-combined-nets} with $q=2, r = \infty, \bQ = \p$, the parameter $n$ in \eqref{eq:coro:convergence-rate} is comparable to the $L_2(\p)$-norm of the cardinality of the combined time-nets $\comb{\tau_n^\theta}{\rho(\ep_n, \frac{1-\theta}{2})}$ used in the approximation.

	\section{Proofs of results in \cref{sec:main-results} and \Cref{subsec:MVH-growth-weight}} 	\label{sec:proofs-main-results}
	
	\subsection{Proofs of results in \Cref{subsec:Riemann-approx}} We need the following auxiliary result.

	\begin{lemm}\label{lem:sigma(Y)-estimate}
		There are constants $c_{\eqref{eq:1:conditional-Y-1-step}}, c_{\eqref{eq:2:conditional-Y-1-step}}>0$ such that for any $0\lee a <b\lee T$, a.s.,
		\begin{align}
			\ce{\F_a}{\int_a^b \sigma(S_t)^2 \od t} & \lee c_{\eqref{eq:1:conditional-Y-1-step}}^2 (b-a) \sigma(S_a)^2, \label{eq:1:conditional-Y-1-step}\\
			\ce{\F_a}{\int_a^b |\sigma(S_t) - \sigma(S_a)|^2 \od t} & \lee c_{\eqref{eq:2:conditional-Y-1-step}}^2 (b-a)^2 \sigma(S_a)^2 \label{eq:2:conditional-Y-1-step}.
		\end{align}
	\end{lemm}
	
	\begin{proof}
		Fix $a\in [0, T)$. For any $b \in (a, T]$, a.s.,
		\begin{align}\label{eq:3:conditional-Y-1-step}
			&\ce{\F_a}{\int_a^b |\sigma(S_t) - \sigma(S_a)|^2 \od t}  \lee |\sigma|_{\Lip}^2 \, \ce{\F_a}{\int_a^b |S_t - S_a|^2 \od t} \notag\\
			& = |\sigma|_{\Lip}^2 \, \bbce{\F_a}{\int_a^b \left|\int_a^t \sigma(S_{u-}) \(\od Z^\rc_u + \int_{ \R_0} z (N_Z - \pi_Z)(\od u, \od z) + b^Z_u \od u\)\right|^2 \od t} \notag\\
			& \lee 3 |\sigma|_{\Lip}^2  \, \ce{\F_a}{\int_a^b \( \int_a^t \sigma(S_{u-})^2 (|a^Z_u|^2 + |j^Z_u|^2) \od u +  \int_a^t |b^Z_u|^2 \od u\int_a^t \sigma(S_{u-})^2 \od u\) \od t} \notag\\
			& \lee c_{\eqref{eq:3:conditional-Y-1-step}}^2 \ce{\F_a}{\int_a^b \int_a^t \sigma(S_u)^2 \od u \od t},
		\end{align}
		where in order to obtain the second inequality we use the conditional It\^o isometry for the martingale term and apply H\"older's inequality for the finite variation term. The last inequality comes from the fact that $t \mapsto \sigma(S_t)$ has at most countable discontinuities, and
		\begin{align}\label{eq:constant-lipschit-related}
			c_{\eqref{eq:3:conditional-Y-1-step}}^2 := 3 |\sigma|_{\Lip}^2(|a^Z_{\eqref{eq:uniformly-bounded-C}}|^2 + |j^Z_{\eqref{eq:uniformly-bounded-KV}}|^2 + |b^Z_{\eqref{eq:uniformly-bounded-KV}}|^2).
		\end{align}
		Then the triangle inequality implies that, a.s.,
		\begin{align*}
			\ce{\F_a}{\int_a^b \sigma(S_t)^2 \od t} & \lee 2(b-a)\sigma(S_a)^2 + 2\ce{\F_a}{\int_a^b |\sigma(S_t) - \sigma(S_a)|^2 \od t}\\
			& \lee 2(b-a)\sigma(S_a)^2 + 2c_{\eqref{eq:3:conditional-Y-1-step}}^2 \ce{\F_a}{\int_a^b \int_a^t \sigma(S_u)^2 \od u \od t}.
		\end{align*}
		Now, for any $A \in \F_a$, it holds that
		\begin{align*}
			\int_a^b \E\1_A \sigma(S_t)^2 \od t  & \lee 2(b-a) \E \1_A \sigma(S_a)^2 + 2c_{\eqref{eq:3:conditional-Y-1-step}}^2 \int_a^b \int_a^t \E \1_A \sigma(S_u)^2 \od u \od t .
		\end{align*}
		Since $\E\int_0^T \sigma(S_u)^2 \od u < \infty$ due to \eqref{eq:square-integrability-S}, using Gronwall's inequality yields 
		\begin{align*}
			\int_a^b \E \1_A \sigma(S_t)^2 \od t  \lee  2(b-a)\E \1_A \sigma(S_a)^2 \e^{2c_{\eqref{eq:3:conditional-Y-1-step}}^2(b-a)},
		\end{align*}
		which verifies \eqref{eq:1:conditional-Y-1-step} with $c_{\eqref{eq:1:conditional-Y-1-step}}^2 := 2\e^{2c_{\eqref{eq:3:conditional-Y-1-step}}^2T}$. In order to obtain \eqref{eq:2:conditional-Y-1-step}, we apply  \eqref{eq:1:conditional-Y-1-step} to the right-hand side of \eqref{eq:3:conditional-Y-1-step}, and then we can let $c_{\eqref{eq:2:conditional-Y-1-step}}^2 = \tfrac{1}{2}c_{\eqref{eq:1:conditional-Y-1-step}}^2 c_{\eqref{eq:3:conditional-Y-1-step}}^2 = c_{\eqref{eq:3:conditional-Y-1-step}}^2 \e^{2 c_{\eqref{eq:3:conditional-Y-1-step}}^2 T}.$
	\end{proof}
	
	\subsubsection*{Proof of \cref{theo:approximation-QV-BMO}} For $\vartheta \in \adm$ and $\tau = (t_i)_{i=0}^n \in \mathcal T_{\det}$, we define the process $\QV$, which is adapted, has continuous and non-decreasing paths on $[0, T]$, by
	\begin{align}\label{eq:QV-process}
		\QV_t : = \sum_{i=1}^n \int_{t_{i-1}\wedge t}^{t_i \wedge t}  |\vartheta_{u} -  \vartheta_{t_{i-1}} |^2 \sigma(S_u)^2  \od u.
	\end{align}
	For $a\in [0, T)$, applying conditional It\^o's isometry and H\"older's inequality yields, a.s., 
	\begin{align}\label{eq:dominant-QV}
		& \ce{\F_a}{|E^\riem_T(\vartheta, \tau) - E^\riem_a(\vartheta, \tau)|^2}  \notag\\
		& \lee 3\bbce{\F_a}{\int_a^T \bigg|\vartheta_{u-} - \sum_{i=1}^n \vartheta_{t_{i-1}-} \1_{(t_{i-1}, t_i]}(u)\bigg|^2 \sigma(S_{u-})^2 \(|a^Z_u|^2 + |j^Z_u|^2  + \int_a^T |b^Z_r|^2 \od r\) \od u} \notag \\
		& \lee 3(|a^Z_{\eqref{eq:uniformly-bounded-C}}|^2 + |j^Z_{\eqref{eq:uniformly-bounded-KV}}|^2 +  |b^Z_{\eqref{eq:uniformly-bounded-KV}}|^2) \bbce{\F_a}{\int_a^T \bigg|\vartheta_{u-} - \sum_{i=1}^n \vartheta_{t_{i-1}-} \1_{(t_{i-1}, t_i]}(u)\bigg|^2 \sigma(S_{u-})^2 \od u} \notag \\
		& = 3(|a^Z_{\eqref{eq:uniformly-bounded-C}}|^2 + |j^Z_{\eqref{eq:uniformly-bounded-KV}}|^2 +  |b^Z_{\eqref{eq:uniformly-bounded-KV}}|^2) \ce{\F_a}{\QV_T - \QV_a},
	\end{align}
	where the equality comes from the fact that the number of discontinuities of a c\`adl\`ag function is at most countable and $\vartheta \in \adm$ has no fixed-time discontinuity. We recall from \cref{rema:deterministic-bmo} that one can use deterministic times instead of stopping times in the definition of $\|\cdot\|_{\bmo_2^\Phi(\p)}$. Therefore, \cref{theo:approximation-QV-BMO} is a direct consequence of \eqref{eq:dominant-QV} and the following:
	
	\begin{prop}\label{prop:QV-BMO}
		Let \cref{assumption-stochastic-integral} hold for some $\theta \in (0, 1]$. Then there exists
		a constant $c_{\eqref{eq:QV-BMO}}>0$ such that for any $\tau \in \cT_{\det}$ and any $a\in [0, T)$, a.s.,
		\begin{align}\label{eq:QV-BMO}
			\ce{\F_a}{\QV_T - \QV_a} \lee c_{\eqref{eq:QV-BMO}}^2 \|\tau\|_\theta \Phi_a^2.
		\end{align}
		Consequently, $\|\QV\|_{\BMO_1^{\Phi^2}(\p)} \lee c_{\eqref{eq:QV-BMO}}^2 \|\tau\|_\theta$.
	\end{prop}
	
	\begin{proof}		By the monotonicity of $\Theta$ and \eqref{eq:growth-strategy}, we have that for $c_{\eqref{eq:estimate-phi}} : = \sqrt{2} c_{\eqref{eq:growth-strategy}}$ and for any $0\lee s <t <T$, a.s.,
		\begin{align}\label{eq:estimate-phi}
			|\vartheta_t -\vartheta_s|^2 \sigma(S_t)^2 & \lee c_{\eqref{eq:estimate-phi}}^2 \big((T-t)^{\theta -1} + (T-s)^{\theta -1}\big) \Phi_t^2.
		\end{align}
		We aim to apply \cite[Theorem 5.3]{GN20} to get \eqref{eq:QV-BMO}. To do this, we define the random measure
		\begin{align*}
			\Pi(\omega, \od t) : = \sigma(S_t(\omega))^2 \od t, \quad \omega \in \Omega.
		\end{align*}
		It is clear that $\Pi(\omega, (0, t]) <\infty$ for any $(\omega, t) \in \Omega \times (0, T)$. For $0\lee s \lee a < b <T$, a.s., 
		\begin{align*}
			&\bbce{\F_a}{\int_{(a, b]} |\vartheta_t - \vartheta_s|^2 \Pi(\cdot, \od t)} = \bbce{\F_a}{\int_{(a, b]} |\vartheta_t - \vartheta_s|^2 \sigma(S_t)^2 \od t} \\
			& \lee 2 \bbce{\F_a}{|\vartheta_a - \vartheta_s|^2 \int_{(a, b]} \sigma(S_t)^2 \od t + \int_{(a, b]} |\vartheta_t - \vartheta_a|^2 \sigma(S_t)^2 \od t  }\\
			& \lee 2 \bbce{\F_a}{|\vartheta_a - \vartheta_s|^2  \Pi(\cdot, (a, b]) + c_{\eqref{eq:assumption-stochastic-integral}}^2 \int_{(a, b]} (b-t) \Upsilon(\cdot, \od t) }.
		\end{align*}
		Let $\tau = (t_i)_{i=0}^n \in \cT_{\det}$ and $a \in [t_{k-1}, t_k)$ for $k  \in [1, n]$. Applying \cite[Theorem 5.3]{GN20} yields a constant $c>0$ independent of $\tau$ and $a$ such that, a.s., 
		\begin{align*}
			& \ce{\F_a}{\QV_T - \QV_a} \\
			&  \lee c \|\tau\|_\theta \, \bbce{\F_a}{\int_{(a, T)} (T-t)^{1- \theta} \Upsilon(\cdot, \od t) + \frac{(T-t_{k-1})^{1- \theta}}{t_k - t_{k-1}} |\vartheta_a - \vartheta_{t_{k-1}}|^2 \int_{(a, t_k]} \sigma(S_t)^2  \od t}\\
			& \lee c \|\tau\|_\theta \[c_{\eqref{eq:thm:curvature-condition}}^2 \Phi_a^2 + c_{\eqref{eq:1:conditional-Y-1-step}}^2 \frac{(T-t_{k-1})^{1- \theta}}{t_k - t_{k-1}} (t_k-a) |\vartheta_a - \vartheta_{t_{k-1}}|^2  \sigma(S_a)^2 \]\\
			& \lee c \|\tau\|_\theta \[c_{\eqref{eq:thm:curvature-condition}}^2  + c_{\eqref{eq:1:conditional-Y-1-step}}^2 c_{\eqref{eq:estimate-phi}}^2 \frac{(T-t_{k-1})^{1- \theta}}{t_k - t_{k-1}} (t_k-a)  \big((T-a)^{\theta -1} + (T-t_{k-1})^{\theta -1}\big) \] \Phi_a^2\\
			& \lee c \|\tau\|_\theta ( c_{\eqref{eq:thm:curvature-condition}}^2  + 2 c_{\eqref{eq:1:conditional-Y-1-step}}^2 c_{\eqref{eq:estimate-phi}}^2 ) \Phi_a^2,
		\end{align*}
		which implies \eqref{eq:QV-BMO} with $c_{\eqref{eq:QV-BMO}}^2 =  c(c_{\eqref{eq:thm:curvature-condition}}^2  + 2 c_{\eqref{eq:1:conditional-Y-1-step}}^2 c_{\eqref{eq:estimate-phi}}^2)$.	For the ``Consequently'' part, since $\QV$ is continuous, it holds that $\|\QV\|_{\BMO_1^{\Phi^2}(\p)} = \|\QV\|_{\bmo_1^{\Phi^2}(\p)} \lee c_{\eqref{eq:QV-BMO}}^2 \|\tau\|_\theta$.	
	\end{proof}

	\subsection{Proofs of results in  \cref{sebsec:approx-correction,subsec:approx-accuracy,subsec:model-improve-rate}} \label{subsec:proofs-sec3}
	
	Let $\ep >0$, $\kappa\gee 0$ and recall $\rho(\ep, \kappa) = (\rho_i(\ep, \kappa))_{i\gee 0}$ in \cref{defi:random-times}. Due to $\eqref{eq:jump-S-Z}$ and the assumption $\sigma(S_{-}) >0$, it holds that
	\begin{align*}
		|\Delta S| > \sigma(S_{-}) \ep (T-\cdot)^\kappa \Leftrightarrow |\Delta Z| > \ep (T- \cdot)^\kappa.
	\end{align*}
	Hence,  we derive from \eqref{eq:random-times-S} the relations
	\begin{align}\label{eq:random-times-Z}
		\rho_i(\ep, \kappa) & = \inf\left\{T \gee t >\rho_{i-1}(\ep, \kappa) : |\Delta Z_t| > \ep (T-t)^\kappa\right\}\wedge T, \; i \gee 1.
	\end{align}
	Since $Z$ is c\`adl\`ag and the underlying filtration satisfies the usual conditions (right continuity and completeness), it implies that $\rho_i(\ep, \kappa)$ are stopping times satisfying $\rho_{i-1}(\ep, \kappa) < \rho_i(\ep, \kappa)$ for $1 \lee  i \lee \cN_{\eqref{eq:cardinality-S}}(\ep, \kappa)$.
	
	For a non-negative Borel function $f$ defined on $\R$, we denote
	\begin{align*}
		\|f(z)\star \nu\|_{L_{\infty}(\p \otimes \Leb)} : = \left\|(\omega, t) \mapsto \int_{\R} f(z) \nu_t(\omega, \od z) \right\|_{L_\infty(\Omega \times [0, T], \p \otimes \Leb)} \in [0, \infty].
	\end{align*}
	Then condition \eqref{eq:square-integrability-nu} is re-written as
	\begin{align}\label{eq:uniformly-square-integrable}
		\|z^2 \star \nu\|_{L_\infty(\p\otimes \Leb)} < \infty.
	\end{align}

	\begin{lemm}\label{lemm:cardinality-random-nets}
		Let  $\ep >0$, $\kappa \gee 0$ be real numbers. Then, for any $\alpha \in [0, \frac{1}{\kappa})$, one has
		\begin{align}\label{eq:estimate-norm-cardinality}
			\|\cN_{\eqref{eq:cardinality-S}}(\ep, \kappa)\|_{L_2(\p)}  \lee 1 + \sqrt{c_{\eqref{eq:estimate-norm-cardinality}}} + c_{\eqref{eq:estimate-norm-cardinality}},
		\end{align}	
		where $c_{\eqref{eq:estimate-norm-cardinality}} := T \|\1_{\{|z|>1\}} \star \nu\|_{L_\infty(\p\otimes \Leb)} +   \ep^{-\alpha} \sup_{r \in (0, 1)} \|  r^\alpha \1_{\{r < |z|\lee 1\}} \star \nu\|_{L_\infty(\p\otimes \Leb)} \frac{T^{1- \alpha \kappa}}{1- \alpha \kappa}.$
	\end{lemm}
	
	\begin{proof} We may assume that $c_{\eqref{eq:estimate-norm-cardinality}}  <\infty$, otherwise \eqref{eq:estimate-norm-cardinality} is trivial. 
		
		\textit{Step 1}. We show that, a.s.,
		\begin{align*}
			\int_0^T \!\! \int_{\R} \1_{\{|z| > \ep(T-t)^\kappa\}} \pi_Z(\od t, \od z) \lee c_{\eqref{eq:estimate-norm-cardinality}}.
		\end{align*} 
		Indeed, one decomposes
		\begin{align*}
			&  \int_0^T \!\! \int_{\R} \1_{\{|z| > \ep(T-t)^\kappa\}} \pi_Z(\od t, \od z)\\
			& = \int_0^T \!\! \int_{\R} \1_{\{|z| > 1\vee (\ep(T-t)^\kappa)\}} \pi_Z(\od t, \od z) + \int_0^T \!\! \int_{\R} \1_{\{1\gee |z| > \ep(T-t)^\kappa\}} \pi_Z(\od t, \od z),
		\end{align*} 
		where the first term on the right-hand side is  upper bounded by $T \| \1_{\{|z|>1\}} \star \nu\|_{L_\infty(\p\otimes \Leb)}$ a.s. Let us denote
		\begin{align}\label{eq:lemm-constant}
			c_{\eqref{eq:lemm-constant}}: = \ts \sup_{r \in (0, 1)} \|  r^\alpha \1_{\{r < |z|\lee 1\}} \star \nu\|_{L_\infty(\p\otimes \Leb)} <\infty.
		\end{align}
	 By an argument using rational approximation with respect to $r$, we infer that 
		\begin{align*}
			\int_{r < |z| \lee 1} \nu_t(\omega, \od z) \lee c_{\eqref{eq:lemm-constant}} r^{-\alpha},\quad \forall r \in (0, 1),
		\end{align*} 
		for $\p\otimes \Leb$-a.e. $(\omega, t) \in \Omega \times [0, T]$. For the second term in the above decomposition, using Fubini's theorem we get that, a.s., 
		\begin{align*}
			& \int_0^T \!\! \int_{\R} \1_{\{1 \gee |z| > \ep(T-t)^\kappa\}} \pi_Z(\od t, \od z)  =  \int_{\{(t, z) \in [0, T] \times \R \,:\, \ep(T-t)^\kappa < |z| \lee 1\}} \nu_t(\od z)  \od t \\
			& \lee c_{\eqref{eq:lemm-constant}} \ep^{-\alpha} \int_{0}^T (T-t)^{-\alpha \kappa} \od t = c_{\eqref{eq:lemm-constant}} \ep^{-\alpha} \frac{T^{1- \alpha \kappa}}{1 - \alpha \kappa}.
		\end{align*}

		\textit{Step 2}. Combining \textit{Step 1} with \cite[Ch.II, Proposition 1.28]{JS03} allows us to write, a.s.,
		$$\int_0^T \!\! \int_{\R} \1_{\{|z| > \ep(T-t)^\kappa\}} N_Z(\od t, \od z)= \int_0^T \!\! \int_{\R} \1_{\{|z| > \ep(T-t)^\kappa\}} \[(N_Z-\pi_Z)(\od t, \od z)  + \pi_Z(\od t, \od z)\].$$ Since $\cN_{\eqref{eq:cardinality-S}}(\ep, \kappa) \lee 1 + \int_0^T \!\! \int_{\R} \1_{\{|z| > \ep(T-t)^\kappa\}} N_Z(\od t, \od z)$ by \eqref{eq:random-times-Z}, we have
		\begin{align*}
			& \|\cN_{\eqref{eq:cardinality-S}}(\ep, \kappa)\|_{L_2(\p)} \lee 1 + \left\|\int_0^T \!\! \int_{\R} \1_{\{|z| > \ep(T-t)^\kappa\}} N_Z(\od t, \od z) \right\|_{L_2(\p)} \\
			& \lee 1 + \left\|\int_0^T \!\! \int_{\R} \1_{\{|z| > \ep(T-t)^\kappa\}} (N_Z-\pi_Z)(\od t, \od z) \right\|_{L_2(\p)} + \left\|\int_0^T \!\! \int_{\R} \1_{\{|z| > \ep(T-t)^\kappa\}} \pi_Z(\od t, \od z)\right\|_{L_2(\p)}\\
			& = 1 + \left\|\int_0^T \!\! \int_{\R} \1_{\{|z| > \ep(T-t)^\kappa\}} \pi_Z(\od t, \od z)\right\|_{L_1(\p)}^{\frac{1}{2}} + \left\|\int_0^T \!\! \int_{\R} \1_{\{|z| > \ep(T-t)^\kappa\}} \pi_Z(\od t, \od z)\right\|_{L_2(\p)}\\
			& \lee 1 +  \sqrt{c_{\eqref{eq:estimate-norm-cardinality}}} + c_{\eqref{eq:estimate-norm-cardinality}},
		\end{align*}
		where one uses \cite[Ch.II, Theorem 1.33(a)]{JS03} to derive the equality.
	\end{proof}

	In order to prove main results, we need the following technical lemma.
	\begin{lemm}\label{lem:small-ball-property}
		Let $\alpha \in [0, \infty)$. Assume that $\mu$ is a Borel measure on $[-1, 1]$ with $\mu(\{0\}) =0$.  If 
		$\sup_{r \in (0, 1)} r^\alpha \int_{r < |x| \lee 1} \mu(\od x) \lee c_{\mu,\alpha} < \infty$,
		then for $\gamma >\alpha$ one has
		\begin{align*}
			\int_{|x| \lee r} |x|^\gamma \mu(\od x) \lee \frac{c_{\mu, \alpha} 2^\gamma}{1 - 2^{\alpha - \gamma}}  r^{\gamma -\alpha} \quad \mbox{for any } r \in (0, 1],
		\end{align*}
		and for $0 < \gamma \lee \alpha$ one has
		\begin{align*}
			\int_{r < |x| \lee 1} |x|^\gamma \mu(\od x) \lee \begin{cases}
				c_{\mu, \alpha} 2^\alpha (1- \log r) & \mbox{ if } \gamma = \alpha\\[3pt]
				\dfrac{c_{\mu, \alpha} 2^{2\alpha-\gamma}}{2^{\alpha - \gamma}-1} r^{\gamma - \alpha} & \mbox { if } \gamma \in (0, \alpha)
			\end{cases}
			\quad \mbox{for all } r\in (0, 1].
		\end{align*}
	\end{lemm}
	
	\begin{proof}
		We pick an $r \in (0, 1]$ and let $m_r \in \bN$ such that $2^{-m_r -1} < r \lee 2^{-m_r}$. Then, for $\gamma  \in (\alpha, \infty)$, it holds that
		\begin{align*}
			\int_{|x| \lee r} |x|^\gamma \mu(\od x) & \lee \sum_{k = m_r}^\infty \int_{2^{-k - 1} < |x| \lee 2^{-k}} |x|^\gamma \mu(\od x)  \lee \sum_{k = m_r}^\infty 2^{- \gamma k} \int_{2^{-k - 1} < |x| \lee 2^{-k}} \mu(\od x) \\
			& \lee \sum_{k = m_r}^\infty 2^{-\gamma k} \int_{2^{-k - 1} < |x| \lee 1} \mu(\od x)\lee c_{\mu, \alpha} \sum_{k=m_r}^\infty 2^{-\gamma k} 2^{\alpha (k+1)}\\
			& = c_{\mu, \alpha}  \frac{2^{\alpha} 2^{m_r(\alpha - \gamma)}}{1 - 2^{\alpha -\gamma}} \lee  \frac{c_{\mu, \alpha} 2^\gamma}{1 - 2^{\alpha - \gamma}}  r^{\gamma -\alpha}.
		\end{align*}
		For $\gamma \in (0, \alpha]$, using the similar argument as above we also obtain
		\begin{align*}
			\int_{r < |x| \lee 1} |x|^\gamma \mu (\od x) & \lee \sum_{k = 0}^{m_r} \int_{2^{-k  -1} < |x| \lee 2^{-k}} |x|^\gamma \mu(\od x) \lee \sum_{k = 0}^{m_r} 2^{-\gamma k} \int_{2^{-k  -1} < |x| \lee 2^{-k}} \mu(\od x)\\
			& \lee \sum_{k = 0}^{m_r} 2^{-\gamma k} \int_{2^{-k  -1} < |x| \lee 1} \mu(\od x) \lee c_{\mu, \alpha} \sum_{k = 0}^{m_r} 2^{-\gamma k} 2^{\alpha(k+1)}.
		\end{align*}
		If $\gamma =\alpha$, then 
		\begin{align*}
			\int_{r < |x| \lee 1} |x|^\gamma \mu (\od x)  & \lee 2^\alpha c_{\mu, \alpha} (m_r +1)  \lee  c_{\mu, \alpha} 2^\alpha (1 - \log r).
		\end{align*}
		If $\gamma \in (0, \alpha)$, then 
		\begin{align*}
			\int_{r < |x| \lee 1} |x|^\gamma \mu (\od x) & \lee c_{\mu, \alpha} \frac{2^{2\alpha- \gamma}}{2^{\alpha -\gamma}-1} 2^{(\alpha - \gamma)m_r} \lee  \frac{c_{\mu, \alpha} 2^{2\alpha-\gamma}}{2^{\alpha - \gamma}-1} r^{\gamma - \alpha}. \qedhere
		\end{align*}
	\end{proof}

	\subsubsection{Proof of \cref{theo:conergence-rate-BMO}}  Denote $\kappa := \frac{1-\theta}{2} \in [0, \frac{1}{2})$.
	

	\textit{Step 1.} We handle the correction term in \eqref{eq:approximation-correction} and the corresponding error. For $\ep>0$,
	\begin{align*}
		\E \int_0^T\!\!\int_{\R} |z| \1_{\{|z| > \ep (T-t)^\kappa\}} \nu_t(\od z) \od t & \lee  \ep^{-1} \E \int_0^T (T-t)^{-\kappa} \int_{\R} z^2 \nu_t(\od z) \od t\\
		& \lee \ep^{-1} \frac{T^{1- \kappa}}{1 - \kappa} \|z^2 \star \nu\|_{L_\infty(\p \otimes \Leb)} <\infty,
	\end{align*}
	where the finiteness holds due to \eqref{eq:uniformly-square-integrable}. This allows us to decompose		
	\begin{align*}
		\int_0^{\boldsymbol{\cdot}} \!\!\int_{ \R_0} z (N_Z - \pi_Z)(\od u,  \od z) =  Z^{\ep, 1} + Z^{\ep, 2} - \gamma^{\ep},
	\end{align*}
	where 
	\begin{align*}
		Z^{\ep, 1} & : = \int_0^{\boldsymbol{\cdot}} \!\!\int_{\R_0} z \1_{\{|z| \lee \ep (T - u)^\kappa\}}(N_Z - \pi_Z)(\od u, \od z),\\
		Z^{\ep, 2} & : = \int_0^{\boldsymbol{\cdot}} \!\!\int_{\R} z \1_{\{|z| > \ep(T-u)^\kappa\}} N_Z(\od u, \od z),\\
		\gamma^{\ep} & : = \int_0^{\boldsymbol{\cdot}}\!\! \int_{\R} z \1_{\{|z| > \ep(T-u)^\kappa\}} \nu_u(\od z) \od u.
	\end{align*}
	Recall $\vartheta^\tau$ in \cref{defi:approximation-correction}. Since \eqref{eq:uniformly-square-integrable} holds in our context, applying \cref{lemm:cardinality-random-nets} with $\alpha =2$ yields $\cN_{\eqref{eq:cardinality-S}}(\ep, \kappa) <\infty$ a.s. Hence, outside a set of probability zero, we have that, for all $t\in [0, T]$,
	\begin{align*}
		& \sum_{\rho_i(\ep, \kappa) \in [0, t] \cap [0, T)} \big(\vartheta_{\rho_i(\ep, \kappa)-}  - \vartheta^\tau_{\rho_i(\ep, \kappa)}\big)  \Delta S_{\rho_i(\ep, \kappa)}\\
		& = \sum_{\rho_i(\ep, \kappa) \in [0, t] \cap [0, T)} \big(\vartheta_{\rho_i(\ep, \kappa)-}  - \vartheta^\tau_{\rho_i(\ep, \kappa)}\big) \sigma\(S_{\rho_i(\ep, \kappa)-}\) \Delta Z_{\rho_i(\ep, \kappa)} \\
		& = \int_{[0, t] \cap [0, T)} (\vartheta_{u-} - \vartheta^\tau_{u}) \sigma(S_{u-}) \od Z^{\ep, 2}_u.
	\end{align*}
	By the representation of $Z$ in \eqref{eq:decomposition-Z}, one can decompose
	\begin{align*}
		\od S_t & = \sigma(S_{t-}) \od Z_t = \sigma(S_{t-}) \(\od Z^\rc_t + b^Z_t \od t + \int_{ \R_0} z (N_Z - \pi_Z)(\od t,  \od z)\)\\
		& = \sigma(S_{t-}) \big(\od Z^\rc_t + b^Z_t \od t + \od Z^{\ep, 1}_t + \od Z^{\ep, 2}_t - \od \gamma_t^{\ep}\big).
	\end{align*}
	We get from the arguments above, together with the fact $\Delta Z^{\ep, 2}_T = \Delta Z_T = 0$ a.s., that
	\begin{align}\label{eq:error-correction-decomposition}
		&E^{\corr}(\vartheta, \tau| \ep, \kappa)   = \int_0^{\boldsymbol{\cdot}} (\vartheta_{u-} - \vartheta^\tau_{u}) \od S_u - \sum_{\rho_i(\ep, \kappa) \in [0, \boldsymbol{\cdot}] \cap [0, T)} \big(\vartheta_{\rho_i(\ep, \kappa)-} - \vartheta^\tau_{\rho_i(\ep, \kappa)}\big) \Delta S_{\rho_i(\ep, \kappa)} \notag\\
		& = \int_0^{\boldsymbol{\cdot}} (\vartheta_{u-} - \vartheta^\tau_{u}) \sigma(S_{u-}) (\od Z^\rc_u + b^Z_u \od u + \od Z^{\ep, 1}_u + \od Z^{\ep, 2}_u - \od \gamma_u^{\ep})   -  \int_0^{\boldsymbol{\cdot}} (\vartheta_{u-} - \vartheta^\tau_{u}) \sigma(S_{u-}) \od Z^{\ep, 2}_u \notag\\
		& = E^{\rm C}(\vartheta, \tau| \ep, \kappa) + E^{\rm S}(\vartheta, \tau| \ep, \kappa) - E^{\rm D}(\vartheta, \tau| \ep, \kappa),
	\end{align}
	where the ``continuous part'', ``small jump part'' and ``drift part'' are given by
	\begin{align*}
		E^{\rm C}(\vartheta, \tau| \ep, \kappa) & : = \int_0^{\boldsymbol{\cdot}} (\vartheta_{u-} - \vartheta^\tau_{u}) \sigma(S_{u-}) (\od Z^\rc_u + b^Z_u \od u), \\
		E^{\rm S}(\vartheta, \tau| \ep, \kappa) & : = \int_0^{\boldsymbol{\cdot}} (\vartheta_{u-} - \vartheta^\tau_{u}) \sigma(S_{u-}) \od Z^{\ep, 1}_u,\\
		E^{\rm D}(\vartheta, \tau| \ep,  \kappa) & := \int_0^{\boldsymbol{\cdot}} (\vartheta_{u-} - \vartheta^\tau_{u}) \sigma(S_{u-}) \int_{\R} z\1_{\{|z| > \ep(T-u)^\kappa\}} \nu_u(\od z) \od u.
	\end{align*}
	The triangle inequality applied to \eqref{eq:error-correction-decomposition} gives
	\begin{align}\label{eq:decompose-error-correction}
		\|E^{\corr}(\vartheta, \tau| \ep, \kappa)\|_{\BMO_2^{\ol\Phi}(\p)} & \lee \sum_{i \in \{\rm{S}, \rm{C}, \rm{D}\}} \|E^{i}(\vartheta, \tau| \ep, \kappa)\|_{\BMO_2^{\ol\Phi}(\p)}.
	\end{align}

	\textit{Step 2}. We preliminary investigate the right-hand side of \eqref{eq:decompose-error-correction}.
	
	\textit{Step 2.1}. Consider $E^{\rm{C}}(\vartheta, \tau| \ep, \kappa)$. We apply the conditional It\^o isometry for the martingale component and apply the  Cauchy--Schwarz inequality for the finite variation component of $E^{\rm C}(\vartheta, \tau| \ep, \kappa)$ to derive that, for $a\in [0, T)$, a.s., 
	\begin{align*}
		& \sqrt{\ce{\F_a}{|E^{\rm C}_T(\vartheta, \tau| \ep, \kappa) - E^{\rm C}_a(\vartheta, \tau| \ep, \kappa)|^2}} \notag \\
		& \lee   \sqrt{\ce{\F_a}{\int_a^T|\vartheta_{u-} - \vartheta^\tau_{u}|^2 \sigma(S_{u-})^2  \od \<Z^\rc\>_u }}  +  \sqrt{\ce{\F_a}{\int_a^T |\vartheta_{u-} - \vartheta^\tau_{u}|^2 \sigma(S_{u-})^2 \od u \int_a^T |b^Z_u|^2 \od u}} \notag \\
		& \lee  (a^Z_{\eqref{eq:uniformly-bounded-C}} + b^Z_{\eqref{eq:uniformly-bounded-KV}}) \sqrt{\ce{\F_a}{\QV_T - \QV_a}} \lee  (a^Z_{\eqref{eq:uniformly-bounded-C}} + b^Z_{\eqref{eq:uniformly-bounded-KV}}) c_{\eqref{eq:QV-BMO}} \sqrt{\|\tau\|_\theta} \, \Phi_a \notag\\
		& \lee (a^Z_{\eqref{eq:uniformly-bounded-C}} + b^Z_{\eqref{eq:uniformly-bounded-KV}}) c_{\eqref{eq:QV-BMO}} \sqrt{\|\tau\|_\theta} \, \ol \Phi_a,
	\end{align*}
where $\QV$ is given in \eqref{eq:QV-process}, and where we use the fact that a c\`adl\`ag function has at most countably many discontinuities for the second inequality.
	Since $E^{\rm C}(\vartheta, \tau| \ep, \kappa)$ has continuous paths, it implies that 
	\begin{align}\label{eq:BMO-estimate-continuous-part}
		\|E^{\rm C}(\vartheta, \tau| \ep, \kappa)\|_{\BMO_2^{\ol \Phi}(\p)} = \|E^{\rm C}(\vartheta, \tau| \ep, \kappa)\|_{\bmo_2^{\ol \Phi}(\p)} \lee c_{\eqref{eq:QV-BMO}} (a^Z_{\eqref{eq:uniformly-bounded-C}} + b^Z_{\eqref{eq:uniformly-bounded-KV}}) \sqrt{\|\tau\|_\theta}.
	\end{align}
	
	\textit{Step 2.2}. Consider $E^{\rm{S}}(\vartheta, \tau| \ep, \kappa)$. Since $\Phi \in \cSM_2(\p)$ by assumption, \cref{lemm:properties-BMO-SM}\eqref{item:bar-Phi-SM} implies that $\ol\Phi \in \cSM_2(\p)$. Then \cref{lemm:relation-BMO-bmo} asserts that
	\begin{align}\label{eq:equiv-BMO-bmo+jump}
		\|E^{\rm{S}}(\vartheta, \tau| \ep, \kappa)\|_{\BMO_2^{\ol\Phi}(\p)} \sim_{c_{\eqref{eq:equiv-BMO-bmo+jump}}} \|E^{\rm{S}}(\vartheta, \tau| \ep, \kappa)\|_{\bmo_2^{\ol\Phi}(\p)} + |\Delta E^{\rm S}(\vartheta, \tau| \ep, \kappa)|_{\ol\Phi}.
	\end{align}
	Since $\vartheta$, $\sigma(S)$ and $\Phi$ are c\`adl\`ag on $[0, T)$, one can find an $\Omega_0$ with $\p(\Omega_0)=1$ on which \eqref{eq:estimate-phi} holds for all $0 \lee s < t < T$. Hence, for $\omega \in \Omega_0$ we have
	\begin{align*}
		|\vartheta_t - \vartheta_s| \sigma(S_t) \lee  \sqrt{2}c_{\eqref{eq:estimate-phi}} (T-t)^{-\kappa} \Phi_t, \quad \forall 0\lee s<t<T.
	\end{align*}
	Due to \eqref{eq:condition-compensator-measure}, one has $\pi_Z(\omega, \{t\} \times \R_0) = 0$ for any  $(\omega, t) \in \Omega \times [0, T]$. Then it holds that
	\begin{align*}
		\big|\Delta Z^{\ep, 1}_t\big| = \left|\int_{\R_0} z \1_{\{|z| \lee \ep (T-t)^\kappa\}} N_Z(\{t\}, \od z) - \int_{\R_0} z \1_{\{|z| \lee \ep (T-t)^\kappa\}} \pi_Z(\{t\}, \od z)\right| \lee \ep(T-t)^\kappa
	\end{align*} 
	for all $t \in [0, T]$ a.s. Moreover, since $\Delta E^{\rm S}(\vartheta, \tau| \ep, \kappa) = (\vartheta_{-} - \vartheta^\tau) \sigma(S_{-}) \Delta Z^{\ep, 1}$, there is an $\Omega_1$ with $\p(\Omega_1) =1$ (with keeping $\Delta Z^{\ep, 1}_T = 0$ a.s. in mind) such that  for all $(\omega, t) \in \Omega_1 \times [0, T]$,
	\begin{align*}
		|\Delta E_t^{\rm S}(\vartheta, \tau| \ep, \kappa)| & = \big|(\vartheta_{t-} - \vartheta^\tau_{t}) \sigma(S_{t-}) \Delta Z^{\ep, 1}_t\big| \lee  \sqrt{2}c_{\eqref{eq:estimate-phi}} (T-t)^{-\kappa} \Phi_{t-}  \ep (T-t)^\kappa\\
		& =  \sqrt{2}c_{\eqref{eq:estimate-phi}} \ep \Phi_{t-} \lee 
		\sqrt{2}c_{\eqref{eq:estimate-phi}} \ep \ol\Phi_t.
	\end{align*}
	According to the definition of $|\cdot|_{\ol\Phi}$ given in \cref{lemm:relation-BMO-bmo}, one then gets 
	\begin{align}\label{eq:error-small-jump-norm}
		|\Delta E^{\rm S}(\vartheta, \tau | \ep, \kappa)|_{\ol\Phi} \lee \sqrt{2} c_{\eqref{eq:estimate-phi}} \ep.
	\end{align}
	Let us continue with $\|E^{\rm S}(\vartheta, \tau| \ep, \kappa)\|_{\bmo_2^{\ol \Phi}(\p)} $. It follows from the conditional It\^o isometry that, for $a\in [0, T)$, a.s., 
	\begin{align}\label{eq:error-small-jump-bmo}
		& \ce{\F_a}{|E^{\rm S}_T(\vartheta, \tau| \ep, \kappa) - E^{\rm S}_a(\vartheta, \tau| \ep, \kappa)|^2} \notag \\
		& = \ce{\F_a}{\int_a^T|\vartheta_{u-} - \vartheta^\tau_{u}|^2 \sigma(S_{u-})^2 \int_{\R}  \1_{\{|z| \lee \ep(T-u)^\kappa\}} z^2 \nu_u(\od z) \od u} \notag\\
		& \lee \|\1_{\{|z| \lee \ep T^\kappa\}} z^2 \star \nu\|_{L_\infty(\p \otimes \Leb)} \ce{\F_a}{\int_a^T|\vartheta_{u-} - \vartheta^\tau_{u}|^2 \sigma(S_{u-})^2 \od u} \notag\\
		& = \|\1_{\{|z| \lee \ep T^\kappa\}} z^2 \star \nu\|_{L_\infty(\p \otimes \Leb)} \ce{\F_a}{\QV_T - \QV_a} \notag\\
		& \lee  \|\1_{\{|z| \lee \ep T^\kappa\}} z^2 \star \nu\|_{L_\infty(\p \otimes \Leb)} c_{\eqref{eq:QV-BMO}}^2 \|\tau\|_\theta \ol \Phi_a^2.
	\end{align}
		Combing  \eqref{eq:error-small-jump-norm} and \eqref{eq:error-small-jump-bmo} with \eqref{eq:equiv-BMO-bmo+jump}  we obtain for $c_{\eqref{eq:error-small-jump-BMO}}: = c_{\eqref{eq:equiv-BMO-bmo+jump}}((\sqrt{2} c_{\eqref{eq:estimate-phi}}) \vee  c_{\eqref{eq:QV-BMO}})$ that
	\begin{align}\label{eq:error-small-jump-BMO}
		\|E^{\rm S}(\vartheta, \tau| \ep, \kappa)\|_{\BMO_2^{\ol\Phi}(\p)} \lee c_{\eqref{eq:error-small-jump-BMO}} \(\ep + \sqrt{\|\1_{\{|z| \lee \ep T^\kappa\}} z^2 \star \nu\|_{L_\infty(\p \otimes \Leb)}} \sqrt{\|\tau\|_\theta}\).
	\end{align}

	\textit{Step 2.3}. Consider $E^{\rm D}(\vartheta, \tau| \ep, \kappa)$. Since $E^{\rm D}(\vartheta, \tau| \ep, \kappa)$ is continuous, it holds that
	\begin{align*}
		\|E^{\rm D}(\vartheta, \tau| \ep, \kappa)\|_{\BMO_2^{\ol\Phi}(\p)} = \|E^{\rm D}(\vartheta, \tau| \ep, \kappa)\|_{\bmo_2^{\ol \Phi}(\p)}.
	\end{align*}
	Then, for any $a\in [0, T)$, we use the Cauchy--Schwarz inequality to get, a.s., 
	\begin{align}\label{eq:estimate-drift-term}
		& \ce{\F_a}{|E^{\rm D}_T(\vartheta, \tau| \ep, \kappa) - E^{\rm D}_a(\vartheta, \tau| \ep, \kappa)|^2} \notag\\
		& \lee \bbce{\F_a}{\bigg(\int_a^T \bigg| \int_{\R} z \1_{\{|z| > \ep(T-u)^\kappa\}} \nu_u(\od z)\bigg|^2 \od u\bigg)\(\int_a^T|\vartheta_{u-} - \vartheta^\tau_{u}|^2 \sigma(S_{u-})^2 \od u\)} \notag\\
		& =: \ce{\F_a}{\mathrm{I}_{\eqref{eq:estimate-drift-term}} \mathrm{II}_{\eqref{eq:estimate-drift-term}}}.
	\end{align}
	
	\eqref{item:1:small-jump-intensity} We now exploit the condition \eqref{eq:item:1:small-jump-intensity}. A rational approximation argument with respect to $r$ yields that, for $\p\otimes \Leb$-a.e. $(\omega, u) \in \Omega \times [0, T]$,
	\begin{align}\label{eq:item:1:constant-estimate}
		\int_{r < |z| \lee 1} \nu_u(\omega, \od z) \lee c_{\eqref{eq:item:1:small-jump-intensity}} r^{-\alpha}, \quad \forall r \in (0, 1).
	\end{align}
	Let us first examine the right-hand side of \eqref{eq:error-small-jump-BMO}. For any $(\omega, u) \in \Omega \times [0, T]$ such that \eqref{eq:item:1:constant-estimate} is satisfied and that $\int_{\R} z^2 \nu_u(\omega, \od z) \lee \|z^2 \star \nu\|_{L_\infty(\p \otimes \Leb)} <\infty$, we apply \cref{lem:small-ball-property} with $\mu(\cdot) = \nu_u(\omega, \cdot)$ and $\gamma =2 \gee \alpha$ to obtain a $c_{\eqref{eq:estimate:small-jum}} >0$ independent of $\omega, u, \ep$ such that
	\begin{align}\label{eq:estimate:small-jum}
		\int_{|z| \lee \ep T^\kappa} z^2 \nu_u(\omega, \od z) \lee c_{\eqref{eq:estimate:small-jum}} ((\ep T^{\kappa})  \wedge 1)^{2- \alpha},
	\end{align}
	where one remarks that \eqref{eq:estimate:small-jum} particularly holds when $\alpha = 2$ or when $\ep T^\kappa >1$. Thus,
	\begin{align*}
		\|\1_{\{|z| \lee \ep T^\kappa\}} z^2 \star \nu\|_{L_\infty(\p \otimes \Leb)} \lee c_{\eqref{eq:estimate:small-jum}} ((\ep T^{\kappa})  \wedge 1)^{2- \alpha} \lee c_{\eqref{eq:estimate:small-jum}} (1+ T^\kappa)^{2- \alpha} (\ep  \wedge 1)^{2- \alpha}.
	\end{align*}
	Then it follows from \eqref{eq:error-small-jump-BMO}  that
	\begin{align}\label{eq:error-small-jump-BMO-1}
		\|E^{\rm S}(\vartheta, \tau| \ep, \kappa)\|_{\BMO_2^{\ol\Phi}(\p)} \lee c_{\eqref{eq:error-small-jump-BMO-1}} ( \ep + (\ep\wedge 1)^{1 - \frac{\alpha}{2}} \sqrt{\|\tau\|_\theta})
	\end{align}
	for some constant $c_{\eqref{eq:error-small-jump-BMO-1}} >0$ independent of $\ep$ and $\tau$.

	We continue with $\mathrm{I}_{\eqref{eq:estimate-drift-term}}$ and  $\mathrm{II}_{\eqref{eq:estimate-drift-term}}$. Let $(\omega, u) \in \Omega \times [0, T]$ such that $0 < \ep(T-u)^\kappa \lee 1$ and \eqref{eq:item:1:constant-estimate} holds. We first apply \cref{lem:small-ball-property} 
	and then use Fubini's theorem to get that, a.s., 
	\begin{align*}
		& \sqrt{\int_0^T\bigg| \int_{\ep (T-u)^\kappa < |z| \lee 1} |z| \nu_u (\od z) \bigg|^2 \od u} \\
		& \lee 2c_{\eqref{eq:item:1:small-jump-intensity}} \begin{cases}
			\frac{\sqrt{T} }{1- 2^{\alpha -1}} & \mbox{ if } \alpha \in (0, 1)\\
			\sqrt{T} (\log^+(\frac{1}{\ep}) +1) + \kappa \sqrt{\int_0^T \log^2(T-u) \od u} & \mbox{ if } \alpha =1\\[5pt]
			\Big[\frac{2^{2\alpha-2}}{2^{\alpha -1}-1}  \sqrt{\int_0^T  (T-u)^{2\kappa(1-\alpha)} \od u} \Big] \ep^{1-\alpha} & \mbox{ if } \alpha \in (1, 2]
		\end{cases} \\
		& \lee 2 c_{\eqref{eq:item:1:small-jump-intensity}} c_{\alpha, \kappa, T} \begin{cases} 1 & \mbox{ if } \alpha \in (0, 1)\\
			\log^+(\frac{1}{\ep}) +1  & \mbox{ if } \alpha =1\\
			\ep^{1-\alpha} & \mbox{ if } \alpha \in (1, 2]
		\end{cases}
	\end{align*}
	for some constant $c_{\alpha, \kappa, T}>0$ depending at most on $\alpha, \kappa, T$, and where one notices that $2 \kappa (1- \alpha) + 1>0$.		For the first factor $\mathrm{I}_{\eqref{eq:estimate-drift-term}}$, the triangle inequality gives, a.s., 
	\begin{align} \label{eq:first-factor-drift-term}
		\sqrt{\mathrm{I}_{\eqref{eq:estimate-drift-term}}} & \lee \sqrt{\int_a^T \bigg|\int_{|z| > 1 \vee (\ep(T-u)^\kappa)} z \nu_u(\od z)\bigg|^2\od u} + \sqrt{\int_a^T \bigg|\int_{\ep(T-u)^\kappa < |z| \lee 1} z \nu_u(\od z)\bigg|^2\od u} \notag  \\
		& \lee \sqrt{T} \|\1_{\{|z|>1\}} |z| \star \nu\|_{L_\infty(\p\otimes \Leb)} + \sqrt{\int_0^T \bigg|\int_{\ep(T-u)^\kappa < |z| \lee 1} z \nu_u(\od z)\bigg|^2\od u} \notag\\
		& \lee c_{\eqref{eq:first-factor-drift-term}} (1 + h(\ep))
	\end{align}
	for some constant  $c_{\eqref{eq:first-factor-drift-term}} = c_{\eqref{eq:first-factor-drift-term}}(\alpha, \kappa, T, \nu)>0$ and for \begin{align*}
		h(\ep) := \begin{cases}
			1 & \mbox{ if } \alpha \in (0, 1)\\
			\log^+(\frac{1}{\ep}) & \mbox{ if } \alpha =1\\
			\ep^{1- \alpha} & \mbox{ if } \alpha \in (1, 2].
		\end{cases}
	\end{align*}
 For the second factor $\mathrm{II}_{\eqref{eq:estimate-drift-term}}$, we apply  \cref{prop:QV-BMO} to obtain, a.s., 
	\begin{align*}
		\ce{\F_a}{\mathrm{II}_{\eqref{eq:estimate-drift-term}}} =  \ce{\F_a}{\QV_T - \QV_a} \lee c_{\eqref{eq:QV-BMO}}^2 \|\tau\|_\theta \Phi_a^2 \lee c_{\eqref{eq:QV-BMO}}^2 \|\tau\|_\theta \ol\Phi_a^2.
	\end{align*}
	Hence,
	\begin{align}\label{eq:drift-error-process-1}
		\|E^{\rm D}(\vartheta, \tau| \ep, \kappa)\|_{\BMO_2^{\ol\Phi}(\p)}  \lee  c_{\eqref{eq:QV-BMO}} c_{\eqref{eq:first-factor-drift-term}} (1+ h(\ep)) \sqrt{ \|\tau\|_\theta}.
	\end{align}
	Eventually, we plug \eqref{eq:BMO-estimate-continuous-part}, \eqref{eq:error-small-jump-BMO-1} and  \eqref{eq:drift-error-process-1} into \eqref{eq:decompose-error-correction} to derive \eqref{1eq:approx-BMO}.
	
	
	\eqref{item:2:jump-behavior} If \eqref{eq:item:2:jump-behavior} holds, then  $\mathrm{I}_{\eqref{eq:estimate-drift-term}} \lee 2 T (\| \1_{\{|z| >1\}} |z| \star \nu\|_{L_\infty(\p \otimes \Leb)}^2 + c_{\eqref{eq:item:2:jump-behavior}}^2)=:c_{\eqref{eq:drift-error-process-2}}^2$. Hence,
	\begin{align}\label{eq:drift-error-process-2}
		\|E^{\rm D}(\vartheta, \tau| \ep, \kappa)\|_{\BMO_2^{\ol\Phi}(\p)} \lee  c_{\eqref{eq:drift-error-process-2}} c_{\eqref{eq:QV-BMO}}  \sqrt{ \|\tau\|_\theta}.
	\end{align}
	Combining \eqref{eq:BMO-estimate-continuous-part}, \eqref{eq:error-small-jump-BMO-1} and \eqref{eq:drift-error-process-2} with \eqref{eq:decompose-error-correction} yields \eqref{1eq:approx-BMO-symmetric}. \qed

	\subsubsection{Proof of \cref{theo:improve-rate}}
	Denote $\kappa : = \frac{1- \theta}{2} \in [0, \frac{1}{2})$ as before.	Pick $\tau = (t_i)_{i=0}^n \in \cT_{\det}$ and $\ep >0$. For $u \in [0, T)$, we define
	\begin{align*}
		D_{\ep, \kappa}(u) : = \int_{|z| > \ep(T-u)^\kappa} z \nu_u(\od z).
	\end{align*}
The conclusion for $E^{\corr}(\vartheta, \tau| \ep, \kappa)$ is shown by using again \eqref{eq:decompose-error-correction}, where the estimate for the ``small jump part'' $E^{\rm S}(\vartheta, \tau| \ep, \kappa)$ is taken from \eqref{eq:error-small-jump-BMO-1}. Here, we focus on improving the estimate for the ``drift part'' $E^{\rm D}(\vartheta, \tau| \ep, \kappa)$.

	\textit{Step 1.} We show that there is a constant $c_{\eqref{eq:estimate-time-net-D}}>0$ independent of $\tau$ and $\ep$ such that
	\begin{align}\label{eq:estimate-time-net-D}
		D_{\eqref{eq:estimate-time-net-D}} &:= \sup_{i =1, \ldots, n} \;\sup_{r \in (t_{i-1}, t_i] \cap[0, T)}  \bigg[ \frac{1}{(T-r)^{\kappa}} \int_{r}^{t_i} |D_{\ep, \kappa}(u)| \od u\bigg] \notag \\
		& \lee c_{\eqref{eq:estimate-time-net-D}} \begin{cases}
			\|\tau\|_\theta & \mbox{if } \eqref{eq:small-jump-pure-jump-2} \mbox{ holds}\\
			\big[1+ \log^+(\frac{1}{\ep}) + \log^+(\frac{1}{\|\tau\|_\theta})\big] \|\tau\|_\theta & \mbox{if } \eqref{eq:small-jump-pure-jump-1} \mbox{ holds with } \alpha =1\\[3pt]
			\|\tau\|_\theta + \ep^{1- \alpha} \|\tau\|_\theta^{1- \kappa(\alpha-1)} & \mbox{if } \eqref{eq:small-jump-pure-jump-1} \mbox{ holds with } \alpha \in (1, 2].
		\end{cases}
	\end{align}
	Indeed, since $\|z^2\star \nu\|_{L_\infty([0, T], \Leb)}=\|u \mapsto \int_{\R} z^2 \nu_u(\od z)\|_{L_\infty([0, T], \Leb)} <\infty$ by assumption, applying \cref{lem:small-ball-property} yields a constant $c_{\eqref{eq:estimate-D}} > 0$ not depending on $\ep$ such that, for $\Leb$-a.e. $u \in [0, T)$,
	\begin{align}\label{eq:estimate-D}
		|D_{\ep, \kappa}(u)| \lee c_{\eqref{eq:estimate-D}} \begin{cases}
			1 & \mbox{ if } \eqref{eq:small-jump-pure-jump-2} \mbox{ holds}\\
			1 + \log^+(\frac{1}{\ep}) + \log^+(\frac{1}{T-u}) & \mbox{ if } \eqref{eq:small-jump-pure-jump-1} \mbox{ holds with } \alpha =1\\
			1 + \ep^{1- \alpha}(T-u)^{\kappa(1- \alpha)} & \mbox{ if } \eqref{eq:small-jump-pure-jump-1} \mbox{ holds with } \alpha \in (1, 2].
		\end{cases}
	\end{align}
\textit{Case 1.} If \eqref{eq:small-jump-pure-jump-2}  holds, then \eqref{eq:estimate-D} immediately implies
\begin{align*}
	D_{\eqref{eq:estimate-time-net-D}} \lee c_{\eqref{eq:estimate-D}} \sup_{i =1, \ldots, n} \;\sup_{r \in (t_{i-1}, t_i] \cap[0, T)}  \bigg[ \frac{t_i - r}{(T-r)^{2\kappa}} (T-r)^\kappa\bigg] \lee c_{\eqref{eq:estimate-D}} T^\kappa \|\tau\|_\theta.
\end{align*}
\textit{Case 2.} \eqref{eq:small-jump-pure-jump-1}  holds with  $\alpha =1$.  For $r \in [t_{i-1}, t_i)$, $i = 1, \ldots, n$, one has
\begin{align*}
	&  \int_r^{t_i} \log^+\(\frac{1}{T-u}\) \od u  \lee (t_i - r) \log^+\(\frac{1}{t_i - r}\) + (t_i - r) \log\e,
\end{align*}
where we first integrate by parts the left-hand side and then use the inequality $b\log^+(\frac{1}{b}) - a \log^+(\frac{1}{a}) \lee (b-a) \log^+(\frac{1}{b-a})$ for all $0<a<b$. Since  $x \mapsto x \log^+(\frac{1}{x})$ is non-decreasing on $(0, \frac{1}{\e}]$ and $0< \frac{1}{\e T^\theta} \frac{t_i - r}{(T-r)^{1-\theta}} \lee \frac{\|\tau\|_\theta}{\e T^\theta} \lee \frac{1}{\e}$ for any $r \in [t_{i-1}, t_i)$, we obtain
\begin{align*}
	&\frac{1}{(T-r)^{\kappa}} \int_r^{t_i} \log^+\(\frac{1}{T-u}\) \od u\\ &\lee \frac{1}{\e T^\theta}\frac{t_i-r}{(T-r)^{2\kappa}}  \[\log^+\(\frac{\e T^\theta(T-r)^{2\kappa}}{t_i -r}\)  + \log^+\(\frac{1}{\e T^\theta (T-r)^{2\kappa}}\) + \log \e\]\e T^\theta (T-r)^\kappa\\
	& \lee \frac{\|\tau\|_\theta}{\e T^\theta} \log^+\(\frac{\e T^\theta}{\|\tau\|_\theta}\) + \|\tau\|_\theta \sup_{r \in [0, T)} \[(T-r)^\kappa \log^+\(\frac{1}{\e T^\theta (T-r)^{2\kappa}}\)\] + \|\tau\|_\theta T^\kappa \log \e\\
	& \lee c_{T, \theta} \|\tau\|_\theta \big( 1 + \log^+\big(\tfrac{1}{\|\tau\|_\theta}\big)\big)
\end{align*}
for some constant $c_{T, \theta}>0$ depending at most on $T, \theta$. Hence,
\begin{align*}
	D_{\eqref{eq:estimate-time-net-D}} \lee c_{\eqref{eq:estimate-D}} \[T^\kappa \big(1 + \log^+\big(\tfrac{1}{\ep}\big)\big)  + c_{T, \theta} \big(1+ \log^+\big(\tfrac{1}{\|\tau\|_\theta}\big)\big)\] \|\tau\|_\theta.
\end{align*}
\textit{Case 3.} \eqref{eq:small-jump-pure-jump-1}  holds with  $\alpha \in (1, 2]$. 
Again, using \eqref{eq:estimate-D} and keeping $\frac{1}{2} < 1 - \kappa(\alpha-1) \lee 1$ in mind we get that, for any $r \in (t_{i-1}, t_i]\cap [0, T)$ and any $i = 1, \ldots, n$,
	\begin{align*}
		&\frac{1}{(T-r)^{\kappa}} \int_r^{t_i} |D_{\ep, \kappa}(u)| \od u   \lee  \frac{c_{\eqref{eq:estimate-D}}}{(T-r)^{\kappa}} \[(t_i - r) + \ep^{1- \alpha} \int_r^{t_i} (T- u)^{\kappa(1- \alpha)} \od u \]\\
		& = c_{\eqref{eq:estimate-D}}  \bigg[\frac{t_i - r}{(T-r)^{\kappa}} +  \frac{\ep^{1- \alpha}}{1 - \kappa(\alpha-1)} \frac{(T-r)^{1 - \kappa(\alpha-1)} - (T -t_{i})^{1 - \kappa(\alpha-1)}}{(T-r)^{\kappa}}  \bigg]\\
		& \lee c_{\eqref{eq:estimate-D}}  \bigg[\frac{t_i - r}{(T-r)^{\kappa}}  +  \frac{\ep^{1- \alpha}}{1 - \kappa(\alpha-1)}\frac{(t_i -r)^{1 - \kappa( \alpha-1)}}{(T-r)^{\kappa}} \bigg]\\
		& = c_{\eqref{eq:estimate-D}}  \[\frac{t_i - r}{(T-r)^{2\kappa}} (T-r)^{\kappa}   +  \frac{\ep^{1- \alpha}}{1 - \kappa(\alpha-1)}\bigg[\frac{t_i -r}{(T-r)^{2\kappa}}\]^{1 - \kappa( \alpha-1)} (T-r)^{2\kappa (1- \kappa(\alpha -1)) - \kappa}\bigg]\\
		& \lee c_{\eqref{eq:estimate-D}} \bigg[ \|\tau\|_\theta T^{\kappa} + \frac{\ep^{1- \alpha}}{1 - \kappa(\alpha-1)} \|\tau\|_\theta^{1 - \kappa(\alpha-1)} T^{2\kappa(1 - \kappa(\alpha -1)) - \kappa}\bigg],
	\end{align*}
	where we notice that $2\kappa(1- \kappa (\alpha -1))-\kappa \gee 0$ so that $(T-r)^{2\kappa(1- \kappa (\alpha -1))-\kappa} \lee T^{2\kappa(1- \kappa (\alpha -1))-\kappa}$ for all $r \in [0, T]$. Thus, the desired assertion follows.

	\textit{Step 2.} We examine the ``drift part'' $E^{\rm D}(\vartheta, \tau| \ep, \kappa)$.
	We let $a \in [t_{k-1}, t_k)$, $k \in[1,n]$, and set $s_i : = a\vee t_i$, $i = k-1, \ldots, n$. Denote
	\begin{align}\label{eq:constant-D}
		D_{\eqref{eq:constant-D}}: = \sup_{i =k, \ldots, n} \;\sup_{r \in (s_{i-1}, s_i] \cap[0, T)}  \bigg[ \frac{1}{(T-r)^{\kappa}} \int_{r}^{s_i} |D_{\ep, \kappa}(u)| \od u\bigg].
	\end{align}
Recall from \cref{assumption-pure-jump} that $\vartheta S = M + V$ where $V_t := \int_0^t v_u \od u$. Then the triangle inequality yields, a.s.,
	\begin{align}\label{eq:estimate-drift-term-special-case}
		& \frac{1}{4}\ce{\F_a}{|E^{\rm D}_T(\vartheta, \tau| \ep, \kappa) - E^{\rm D}_a(\vartheta, \tau| \ep, \kappa)|^2} \notag\\
		& = \frac{1}{4}\bbce{\F_a}{\bigg|(\vartheta_a - \vartheta_{t_{k-1}})\int_a^{t_k}  S_u D_{\ep, \kappa}(u) \od u + \sum_{i=k}^n \int_{s_{i-1}}^{s_i} (\vartheta_u - \vartheta_{s_{i-1}}) S_u D_{\ep, \kappa}(u) \od u\bigg|^2} \notag\\
		& \lee \bbce{\F_a}{\bigg|(\vartheta_a - \vartheta_{t_{k-1}})\int_a^{t_k}  S_u D_{\ep, \kappa}(u) \od u\bigg|^2} + \bbce{\F_a}{\bigg|\sum_{i=k}^n \int_{s_{i-1}}^{s_i} (M_u - M_{s_{i-1}}) D_{\ep, \kappa}(u) \od u\bigg|^2}\notag \\
		& \quad +  \bbce{\F_a}{\bigg|\sum_{i=k}^n \int_{s_{i-1}}^{s_i} (V_u - V_{s_{i-1}}) D_{\ep, \kappa}(u) \od u\bigg|^2} + \bbce{\F_a}{\bigg|\sum_{i=k}^n \vartheta_{s_{i-1}} \int_{s_{i-1}}^{s_i} (S_u - S_{s_{i-1}}) D_{\ep, \kappa}(u) \od u\bigg|^2} \notag \\
		& =: \mathrm{I}_{\eqref{eq:estimate-drift-term-special-case}} + \mathrm{II}_{\eqref{eq:estimate-drift-term-special-case}} + \mathrm{III}_{\eqref{eq:estimate-drift-term-special-case}} + \mathrm{IV}_{\eqref{eq:estimate-drift-term-special-case}}.
	\end{align}
	For $\mathrm{I}_{\eqref{eq:estimate-drift-term-special-case}}$, we make use of the growth property of $\vartheta$ and the monotonicity of $\Theta$ to get, a.s., 
\begin{align*}
	\mathrm{I}_{\eqref{eq:estimate-drift-term-special-case}} & \lee 2c_{\eqref{eq:growth-strategy}}^2 (T-a)^{\theta-1} \Theta_a^2 \bigg[\int_a^{t_k} |D_{\ep, \kappa}(u)|\od u\bigg]^2 \bce{\F_a}{\ts\sup_{u \in [a, t_k]} S_u^2} \\
	& \ts \lee 2 D_{\eqref{eq:constant-D}}^2 c_{\eqref{eq:growth-strategy}}^2 \bce{\F_a}{\sup_{u \in [a, t_k]} \Theta_u^2 S_u^2} = 2 D_{\eqref{eq:constant-D}}^2 c_{\eqref{eq:growth-strategy}}^2 \bce{\F_a}{\sup_{u \in [a, t_k]}\Phi_u^2} \\
	& \lee 2 D_{\eqref{eq:constant-D}}^2 c_{\eqref{eq:growth-strategy}}^2 \|\Phi\|_{\cSM_2(\p)}^2 \Phi_a^2.
\end{align*}
	For $\mathrm{II}_{\eqref{eq:estimate-drift-term-special-case}}$, using the orthogonality of martingale increments we find that the mixed terms in the square expansion vanish under the conditional expectation. Then, applying the stochastic Fubini theorem and the conditional It\^o isometry we obtain, a.s., 
	\begin{align*}
		\mathrm{II}_{\eqref{eq:estimate-drift-term-special-case}} & = \sum_{i=k}^n\bbce{\F_a}{ \bigg| \int_{(s_{i-1},s_i]\cap(0, T)} (M_u - M_{s_{i-1}}) D_{\ep, \kappa}(u) \od u\bigg|^2}  \\
		& = \sum_{i=k}^n \bbce{\F_a}{ \bigg| \int_{(s_{i-1}, s_i] \cap (0, T)} \bigg(\int_{[r, s_{i}]} D_{\ep, \kappa}(u) \od u\bigg) \od M_r  \bigg|^2} \\
		& = \sum_{i=k}^n \bbce{\F_a}{\int_{(s_{i-1}, s_i] \cap (0, T)} \bigg|\int_{[r, s_{i}]} D_{\ep, \kappa}(u) \od u\bigg|^2 \od \<M\>_r}\\
		& \lee \sup_{i =k, \ldots, n} \;\sup_{r \in (s_{i-1}, s_i] \cap[0, T)} \bigg[\frac{1}{(T-r)^{1- \theta}} \bigg|\int_{[r, s_{i}]} D_{\ep, \kappa}(u) \od u\bigg|^2 \bigg] \bbce{\F_a}{\int_{(a, T)} (T-u)^{1- \theta} \od \<M\>_u}\\
		& \lee  D_{\eqref{eq:constant-D}}^2 \bbce{\F_a}{\int_{(a, T)} (T-u)^{1- \theta} \Upsilon(\cdot, \od u)}\\
		& \lee D_{\eqref{eq:constant-D}}^2 c_{\eqref{eq:thm:curvature-condition}}^2 \Phi_a^2.
	\end{align*}
For $\mathrm{III}_{\eqref{eq:estimate-drift-term-special-case}}$, we use Fubini's theorem and H\"older's inequality to obtain, a.s., 
\begin{align*}
	\mathrm{III}_{\eqref{eq:estimate-drift-term-special-case}}& \lee \bbce{\F_a}{\bigg|\sum_{i=k}^n \int_{(s_{i-1}, s_i] \cap (0, T)} \bigg(\int_{s_{i-1}}^u |v_r| \od r\bigg) |D_{\ep, \kappa}(u)| \od u \bigg|^2}\\
	& \lee  D_{\eqref{eq:constant-D}}^2 \bbce{\F_a}{\bigg|\int_{(a, T)} (T-r)^{\frac{1-\theta}{2}} |v_r| \od r\bigg|^2} \lee D_{\eqref{eq:constant-D}}^2 (T-a) \bbce{\F_a}{\int_{(a, T)} (T-r)^{1-\theta} v_r^2 \od r}\\
	& \lee D_{\eqref{eq:constant-D}}^2 T \bbce{\F_a}{\int_{(a, T)} (T-r)^{1-\theta} \Upsilon(\cdot, \od r)} \lee D_{\eqref{eq:constant-D}}^2 T c_{\eqref{eq:thm:curvature-condition}}^2 \Phi_a^2.
\end{align*} 
For $\mathrm{IV}_{\eqref{eq:estimate-drift-term-special-case}}$, we also exploit the martingale property of $S$, the monotonicity of $\Theta$, and follow the same argument as for $\mathrm{II}_{\eqref{eq:estimate-drift-term-special-case}}$ to get, a.s.,
\begin{align*}
	\mathrm{IV}_{\eqref{eq:estimate-drift-term-special-case}} & = \sum_{i=k}^n 
	\bbce{
		\F_a}{\bigg|\vartheta_{s_{i-1}}\int_{s_{i-1}}^{s_i}(S_u - S_{s_{i-1}})D_{\ep, \kappa}(u)\od u\bigg|^2}\\
	& = \sum_{i=k}^n 
	\bbce{
		\F_a}{\vartheta_{s_{i-1}}^2\int_{(s_{i-1}, s_i]}\bigg|\int_{[r, s_i]} D_{\ep, \kappa}(u) \od u\bigg|^2 \od \<S\>_r}\\
	& 
	\lee D_{\eqref{eq:constant-D}}^2 \sum_{i=k}^n \bbce{\F_a}{\vartheta_{s_{i-1}}^2 \int_{(s_{i-1}, s_i]} (T-r)^{1- \theta} \od \<S\>_r}\\
	& \lee D_{\eqref{eq:constant-D}}^2 \|z^2 \star \nu\|_{L_\infty([0, T], \Leb)} \sum_{i=k}^n \bbce{\F_a}{(T-s_{i-1})^{1- \theta}\vartheta_{s_{i-1}}^2 \bbce{\F_{s_{i-1}}}{\int_{s_{i-1}}^{s_i} S_r^2 \od r}}\\
	& \lee D_{\eqref{eq:constant-D}}^2 \|z^2 \star \nu\|_{L_\infty([0, T], \Leb)}  c_{\eqref{eq:growth-strategy}}^2 \sum_{i=k}^n \bbce{\F_a}{\int_{s_{i-1}}^{s_i} \Theta_r^2 S_r^2 \od r}\\
	& \lee  D_{\eqref{eq:constant-D}}^2 \|z^2 \star \nu\|_{L_\infty([0, T], \Leb)}  c_{\eqref{eq:growth-strategy}}^2 (T-a) \|\Phi\|_{\cSM_2(\p)}^2 \Phi_{a}^2,
\end{align*}
where in order to obtain the second inequality we employ the assumption that $
\od S_t = S_{t-} \od Z_t$ and $
\od \<Z\>_t = \int_{\R} z^2 \nu_t(\od z) \od t$ with $\|z^2 \star \nu\|_{L_\infty([0, T], \Leb)}  <\infty$.

	 Eventually, plugging the estimates for $\mathrm{I}_{\eqref{eq:estimate-drift-term-special-case}}$--$\mathrm{IV}_{\eqref{eq:estimate-drift-term-special-case}}$ into \eqref{eq:estimate-drift-term-special-case}, and using the fact that $D_{\eqref{eq:constant-D}}  \lee D_{\eqref{eq:estimate-time-net-D}}$, we derive a constant $c_{\eqref{eq:estimate-BMO-drift}}>0$ independent of $
	 \tau$ and $
	 \ep$ such that
	 \begin{align}\label{eq:estimate-BMO-drift}
	 	& \|E^{\rm D}_T(\vartheta, \tau| \ep, \kappa)\|_{\BMO_2^{\ol \Phi}(\p)}  \lee \|E^{\rm D}_T(\vartheta, \tau| \ep, \kappa)\|_{\BMO_2^{\Phi}(\p)} \lee c_{\eqref{eq:estimate-BMO-drift}} D_{\eqref{eq:constant-D}} \notag \\
&  \lee c_{\eqref{eq:estimate-BMO-drift}} c_{\eqref{eq:estimate-time-net-D}}  \begin{cases}
			\|\tau\|_\theta & \mbox{ if } \eqref{eq:small-jump-pure-jump-2} \mbox{ holds}\\
			\big[1 + \log^+(\frac{1}{\ep}) + \log^+(\tfrac{1}{\|\tau\|_\theta})\big] \|\tau\|_\theta & \mbox{ if } \eqref{eq:small-jump-pure-jump-1} \mbox{ holds with } \alpha =1\\[3pt]
			\|\tau\|_\theta + \ep^{1- \alpha} \|\tau\|_\theta^{1- \kappa(\alpha -1)} & \mbox{ if } \eqref{eq:small-jump-pure-jump-1} \mbox{ holds with } \alpha \in (1, 2].
		\end{cases}
	\end{align}
	
	\textit{Step 3.} Combining \eqref{eq:estimate-BMO-drift} and \eqref{eq:error-small-jump-BMO-1} with \eqref{eq:decompose-error-correction} yields the conclusion, where we remark that the condition \eqref{eq:small-jump-pure-jump-1} with $\alpha \in (0, 1)$ implies \eqref{eq:small-jump-pure-jump-2} due to \cref{lem:small-ball-property} (with $\gamma =1$). \qed

	\subsubsection{Proof of  \cref{prop:cardinality-combined-nets}} Denote $\kappa: = \frac{1 - \theta}{2} \in [0, \frac{1}{2})$.
	We first consider the particular case when $\bQ= \p$,  $r=\infty$ and $q=2$.  By \cref{defi:approximation-correction}\eqref{item:defi:combined-net},
	\begin{align*}
		n+1 = \#\tau_n &\lee \# \comb{\tau_n}{\rho\(\ep_n, \kappa\)}\lee n+1 + \cN_{\eqref{eq:cardinality-S}}\(\ep_n, \kappa\).
	\end{align*}
	Thus,
	\begin{align*}
		n+1 \lee \left\|\# \comb{\tau_n}{\rho\(\ep_n, \kappa\)} \right\|_{L_2(\p)} \lee n+1 + \left\|\cN_{\eqref{eq:cardinality-S}}\(\ep_n, \kappa\)\right\|_{L_2(\p)}.
	\end{align*}
	Since $\inf_{n \gee 1} \sqrt[\alpha]{n} \ep_n >0$ by assumption, we derive from \eqref{eq:estimate-norm-cardinality} that
	\begin{align*}
		c_{\eqref{eq:estimate-norm-cardinality}} & = T \|\1_{\{|z|>1\}} \star \nu\|_{L_\infty(\p\otimes \Leb)} + \ep_n^{-\alpha}  \sup_{r \in (0, 1)} \| r^\alpha \1_{\{r< |z|\lee 1\}}\star \nu\|_{L_\infty(\p\otimes \Leb)} \frac{T^{1- \alpha \kappa}}{1- \alpha \kappa} \lee c n
	\end{align*}
	for some constant $c>0$ independent of $n$. Using \eqref{eq:estimate-norm-cardinality} gives the desired conclusion.

	We next assume a probability measure $\bQ \ll \p$ with $\od \bQ/\od \p \in L_r(\p)$. Since $\frac{1}{2/q} + \frac{1}{r} =1$, applying H\"older's inequality yields
	\begin{align*}
		\left\|\#\comb{\tau_n}{ \rho\(\ep_n, \kappa\)}\right\|_{L_q(\bQ)} \lee \left\|\#\comb{\tau_n}{ \rho\(\ep_n, \kappa\)}\right\|_{L_2(\p)} \|\od \bQ/\od \p\|_{L_r(\p)}^{1/q},
	\end{align*}
	and hence, \eqref{eq:norm-combined-nets} follows. \qed

	\subsection{Proofs of results in \Cref{subsec:MVH-growth-weight}}\label{subsec:proof-sec-4}
	
	\subsubsection{Proof of \cref{theo:application-Levy}} It follows from  \cref{assum-maringale-setting} that $\int_{|x|>1} \e^{2x} \nu(\od x) <\infty$. Since $g$ has at most linear growth at infinity, one has $g(S_T) \in L_2(\p)$.
	
	\eqref{item: weight-regularity-Levy} follows from \cref{prop:regularity-weight-process}. 
	
	
	\eqref{item:strategy-growth} We let $\ell := \nu$ in \eqref{eq:varGamma-function} and obtain from \eqref{eq:integrand-general}  that 
	$$\vartheta^g_t = c_{\eqref{eq:integrand-general}}^{-2} \varGamma_\nu(T-t, S_t) \quad \mbox{a.s., } \forall t \in [0, T).$$
	Let us examine each case in \cref{tab:theta-cases}. Applying \cref{theo:envelop-MVH-strategy}\eqref{item:1:sigma>0-envelope} and \eqref{item:2:sigma=0-envelope} yields $\mathrm{A1}$ and  $\mathrm{A2}$ respectively. For $\mathrm{A3}$, since $\nu \in \sS(\alpha)$, \cref{remark:Holder-stable-like}\eqref{item:rema:conditions-C-alpha-1} asserts that $\sup_{r \in (0, 1)} r^\alpha \int_{r< |x| \lee 1} \nu(\od x) <\infty$. Applying \cref{theo:envelop-MVH-strategy}\eqref{item:2.2:sigma=0-envelope} with $\beta  = \alpha$  yields \begin{align*}
		|\vartheta^g_t| \lee c_{\eqref{eq:integrand-general}}^{-2} c_{\eqref{eq:thm:Holder-case-estimate:psi-process}} U(t) S_t^{\eta -1} \mbox{ a.s., } \forall t\in [0, T),
	\end{align*}
	which verifies $\mathrm{A3}$. 
	
	
	\eqref{item:example-assumprion-stoc-inte} The SDE for $S = \e^X$ is 
	\begin{align}\label{eq:SDE-levy}
		\od S_t = S_{t-} \od Z_t, \quad S_0 =1,	
	\end{align}
	where $Z$ is \textit{another} L\'evy process under $\p$.  Under \cref{assum-maringale-setting}, it is known that $Z$ is also an $L_2(\p)$-martingale with zero mean (see, e.g., \cite[Proposition 8.20]{CT03}). Hence, conditions \textbf{[S]} and \textbf{[Z]} in   \Cref{subsec:setting-stochastic-integal} are fulfilled. Moreover, $\vartheta^g \in \adm$ due to \cref{theo:MVH-strategy}(\ref{item:prop:strategy},~\ref{item:prop:maringale-property}).
	
	Let us now verify \cref{assumption-stochastic-integral}. Since $M = \vartheta^gS$ is an $L_2(\p)$-martingale by \cref{theo:MVH-strategy}\eqref{item:prop:maringale-property}, \cref{pre-assumption-stochastic-integral} holds because of  \cref{exam:Upsilon-martingale-setting} (with $V \equiv 0$). Thanks to \eqref{eq:thm:approx-bmo}, the growth condition \eqref{eq:growth-strategy} is satisfied for $\theta$ given in \cref{tab:theta-cases} case-wise.  We now only need to check the curvature condition \eqref{eq:thm:curvature-condition}. If $U \equiv 1$ (in $\mathrm{A1}$ with $\eta =1$ and in $\mathrm{A2}$), then the martingale $M$ is closed in $L_2$ by $M_T: = L_2(\p)\mbox{-}\lim_{t \uparrow T} M_t$ due to \eqref{eq:thm:approx-bmo} and $\Phi(\eta) \in \cSM_2(\p)$. Then, for $\theta =1$ and for any $a\in [0, T)$, one has, a.s., 
	\begin{align*}
		\bbce{\F_a}{\int_{(a, T)} \Upsilon(\cdot, \od t)} & = \bbce{\F_a}{\int_{(a, T)} \od \<M\>_t + \int_{(a, T)} M_t^2 \od t }\\
		& \lee \bbce{\F_a}{|M_T-M_a|^2 + c_{\eqref{eq:thm:approx-bmo}}^2  (T-a) \sup_{ t \in (a, T)} \Phi(\eta)_t^2}\\
		& \lee c_{\eqref{eq:thm:approx-bmo}}^2 (T + 1) \|\Phi(\eta)\|_{\cSM_2(\p)}^2 \Phi(\eta)_a^2.
	\end{align*}
	For remaining cases, we set $\hat\theta: = \eta$ in $\mathrm{A1}$ for $\eta \in (0, 1)$, and set $\hat\theta : = \frac{2(1+ \eta)}{\alpha} -1 \in (0, 1]$ in $\mathrm{A3}$. Then, for any $\theta \in (0, \hat \theta)$, using the function $U$ in \cref{tab:theta-cases} we get
	$$(T-t)^{1- \theta} M_t^2 \lee c_{\eqref{eq:thm:approx-bmo}}^2 (T-t)^{1- \theta} U(t)^2 \Phi(\eta)_t^2 \to 0 \quad \mbox{a.s. as } t \uparrow T.$$
	Thus, for any $a\in [0, T)$, a.s., 
	\begin{align}\label{eq:second-term-Upsilon}
		\bbce{\F_a}{\int_{(a, T)} (T-t)^{1- \theta} M_t^2 \od t} \lee c_{\eqref{eq:second-term-Upsilon}} c_{\eqref{eq:thm:approx-bmo}}^2   \|\Phi(\eta)\|_{\cSM_2(\p)}^2 \Phi(\eta)_a^2
	\end{align}
	for some constant $c_{\eqref{eq:second-term-Upsilon}} >0$ depending at most on $\hat\theta, \theta, T$. Integrating by parts and applying conditional It\^o's isometry yield, a.s., 
	\begin{align}\label{eq:first-term-Upsilon}
		& \bbce{\F_a}{\int_{(a, T)} (T-t)^{1- \theta} \od \<M\>_t} = \bbce{\F_a}{\lim_{a< b \uparrow T} \int_{(a, b]} (T-t)^{1- \theta} \od \<M\>_t} \notag\\
		& =  \bbce{\F_a}{\lim_{a< b \uparrow T}\bigg[(T-b)^{	1- \theta}|M_b - M_a|^2 + (1- \theta) \int_{(a, b]}(T - t)^{-\theta} |M_t - M_a|^2 \od t\bigg]} \notag\\
		& \lee (1- \theta) \bbce{\F_a}{\int_{(a, T)} (T-t)^{- \theta}M_t^2 \od t}  \lee c_{\eqref{eq:first-term-Upsilon}} c_{\eqref{eq:thm:approx-bmo}}^2 \|\Phi(\eta)\|_{\cSM_2(\p)}^2 \Phi(\eta)_a^2
	\end{align}
	for some $c_{\eqref{eq:first-term-Upsilon}} = c_{\eqref{eq:first-term-Upsilon}}(\hat\theta, \theta, T) >0$. 
	Combining \eqref{eq:second-term-Upsilon} with \eqref{eq:first-term-Upsilon} yields the conclusion.
	
	For the particular case $\sigma =0$, it is easy to check that  \cref{assumption-pure-jump} holds true. \qed

	\subsubsection{Proof of \cref{coro:convergence-rate-levy}} Let $\nu_Z$ denote the L\'evy measure of $Z$ (which appears in the SDE \eqref{eq:SDE-levy}) under $\p$. By the relation between the L\'evy measures of $Z$ and of $X$ given in  \Cref{subsec:exponential-Levy}, some simple calculations yield that, for any $\alpha \in [0, 2]$,
	\begin{align*}
		\sup_{r \in (0, 1)} r^\alpha \int_{r < |z|\lee 1} \nu_Z(\od z) <\infty \; \Leftrightarrow \; \sup_{r \in (0, 1)} r^\alpha \int_{r < |x|\lee 1} \nu(\od x) <\infty.
	\end{align*}

	\eqref{item:1:coro:convergence-rate-levy} Since $\int_{|x| >1} \e^{2x} \nu(\od x) <\infty$ by \cref{assum-maringale-setting}, we  get from  \cref{prop:regularity-weight-process} that	$\Phi(\eta)$ belongs to $\cSM_2(\p)$. 	Let us examine each case in  \cref{tab:rates}.
	
	\underline{For $\mathrm{B1}$}, the given range of $(\eta, \alpha)$ yields $\int_{|x| \lee 1} |x|^{1+ \eta} \nu(\od x) <\infty$. We apply  \cref{theo:application-Levy}\eqref{item:example-assumprion-stoc-inte}, case $\mathrm{A2}$ in  \cref{tab:theta-cases}, to get $\theta =1$. Then, applying  \cref{theo:BMO-convergent-rate}\eqref{item:2:coro:BMO-convergent-rate} gives the corresponding $R(n)$ and $\ep_n$.
	
	\underline{For $\mathrm{B2}$}, we note that $\sS(\alpha) \subset \sUS(\alpha)$. Using  \cref{theo:application-Levy}\eqref{item:example-assumprion-stoc-inte}, case $\mathrm{A3}$ in \cref{tab:theta-cases}, to get the range of  $\theta$, and then applying  \cref{theo:BMO-convergent-rate}\eqref{item:2:coro:BMO-convergent-rate} we obtain the conclusions for $R(n)$ and $\ep_n$.
	
	\underline{For $\mathrm{B3}$}, we first get $\theta =1$ due to \cref{tab:theta-cases}, case $\mathrm{A1}$.  We decompose $E^{\corr}$ into three components $E^{\rm C}$, $E^{\rm S}$ and $E^{\rm D}$ as in \eqref{eq:error-correction-decomposition}. For the ``continuous part'' $E^{\rm C}$, \eqref{eq:BMO-estimate-continuous-part} yields
	\begin{align*}
		\sup_{n \gee 1} \sqrt{n}\|E^{\rm C}(\vartheta^g, \tau^1_n \,| \sqrt{1/n}, 0)\|_{\BMO_2^{\ol \Phi(1)}(\p)} < \infty.
	\end{align*}
	For the ``small jump part'' $E^{\rm S}$ and the ``drift part'' $E^{\rm D}$, we apply  \eqref{eq:error-small-jump-BMO-1} and \eqref{eq:estimate-BMO-drift} with $\alpha =2$ respectively to obtain
	\begin{align*}
		\sup_{n \gee 1} \sqrt{n} \Big(\|E^{\rm S}(\vartheta^g, \tau^1_n \,| \sqrt{1/n}, 0)\|_{\BMO_2^{\ol \Phi(1)}(\p)} + \|E^{\rm D}(\vartheta^g, \tau^1_n \,| \sqrt{1/n}, 0)\|_{\BMO_2^{\ol \Phi(1)}(\p)} \Big) <\infty.
	\end{align*}
	Thus, the desired assertion follows from \eqref{eq:decompose-error-correction}. The argument \underline{for $\mathrm{B4}$} is similar to $\mathrm{B3}$.
	
	\eqref{item:2:coro:convergence-rate-levy} Since $\int_{|x| >1} \e^{px} \nu(\od x) <\infty$, it follows from  \cref{prop:regularity-weight-process,lemm:properties-BMO-SM} that $\Phi(\eta)$ and $\ol\Phi(\eta)$ belong to  $\cSM_p(\p)$ for all $\eta \in [0, 1]$. Hence, the conclusion follows from   \cref{lemm:feature-BMO}(\ref{item:Lp-estimate-BMO-feature}, \ref{item:SM-BMO-relation}). \qed

	\section{It\^o's chaos expansion and proof of \cref{theo:MVH-strategy}} \label{app-sec:proof-MVH-strategy}
	
	\subsection{Exponential L\'evy processes}\label{subsec:exponential-Levy} Let $X$ be a L\'evy process with characteristic triplet $(\gamma, \sigma, \nu)$ as in \Cref{subsection-exponential-levy-process}.	It is known that the ordinary exponential $S = \e^X$ can be represented as the \textit{Dol\'eans--Dade exponential} (or \textit{stochastic exponential}) $\cE(Z)$ of another L\'evy process $Z$ (see, e.g., \cite[Theorem 5.1.6]{Ap09}). This means that $S = \cE(Z)$ and
	\begin{align*}
		\od S_t= S_{t-}\od Z_t, \quad S_0 = 1.
	\end{align*}

The path relation of $X$ and $Z$ is given by 
			\begin{align*}
				Z_t = X_t + \frac{\sigma^2 t}{2} + \sum_{0 \lee s\lee t}\(\e^{\Delta X_s}-1-\Delta X_s\),\quad \forall t\in [0, T] \;\mbox{ a.s.,}
			\end{align*}
			which implies $\Delta Z = \e^{\Delta X} -1$. Hence, if $\bF^Z = (\F^Z_t)_{t\in [0, T]}$ is the augmented natural filtration induced by $Z$, then $\F^Z_t = \F^X_t$ for all $t\in [0, T]$. 
			
			 The relation between the characteristic triplets $(\gamma, \sigma, \nu)$ of $X$ and $(\gamma_Z,  \sigma_Z, \nu_Z)$ of $Z$ is provided, e.g., in \cite[Theorem 5.1.6]{Ap09}. In particular, one has $\sigma_Z = \sigma$ and $\nu_Z(\cdot)   = \int_{ \R}\1_{\{\e^x -1 \in \cdot\}}\nu(\od x)$.

	\subsection{It\^o's chaos expansion} \label{subsection-Levy-process} 
	
	We present briefly the Malliavin calculus for L\'evy processes by means of It\^o's chaos expansion. For further details, the reader can refer to \cite{Ap09, SUV07b} and the references therein. 
	
	We define the $\sigma$-finite measures $\mu$ on $\cB(\R)$ and $\m$ on $\cB([0, T] \times \R)$ by setting
	\begin{align*}
		\mu(\od x) := \sigma^2 \delta_0(\od x) + x^2 \nu(\od x)\quad \mbox{and} \quad \m := \Leb \otimes \mu,
	\end{align*}
	where $\delta_0$ is the Dirac measure at zero. For $B\in \cB([0, T] \times \R)$ with $\m(B)<\infty$, the random measure $M$ is defined by
	\begin{align*}
		M(B): = \sigma \int_{\{t\in [0, T] : (t, 0) \in B\}} \od W_t + L_2(\p)\mbox{-}\lim_{ n \to \infty}\int_{B \cap ([0, T] \times \{\frac{1}{n} <|x| < n\})} x \wt N(\od t, \od x),
	\end{align*} 
	where $W$ is the standard Brownian motion and $\wt N(\od t, \od x) : = N(\od t, \od x) - \od t \nu(\od x)$ is the compensated Poisson random measure appearing in the L\'evy--It\^o decomposition of $X$ (see \cite[Theorem 2.4.16]{Ap09}). 	Set $L_2(\mu^{0}) = L_2(\m^{0}) : = \R$, and for $n \gee 1$ we denote 
	\begin{align*}
		L_2(\mu^{\otimes n}) : = L_2(\R^n, \cB(\R^n), \mu^{\otimes n}),\quad 
		L_2(\m^{\otimes n}) : = L_2(([0, T] \times \R)^n, \cB(([0, T] \times \R)^n), \m^{\otimes n}).
	\end{align*} 	
For $n\gee 1$, the multiple integral $I_n \colon L_2(\m^{\otimes n}) \to L_2(\p)$ is defined by a standard approximation where the multiple integral of simple functions is given as follows: for 
	$\xi_n^m: = \sum_{k=1}^m a_k \1_{B^k_1 \times \cdots \times B^k_n}$,
	 $a_k \in \R$, $B^k_i \in \cB([0, T]\times \R)$ with $\m(B^k_i) <\infty$ and $B^k_i \cap B^k_j = \emptyset$ for $m \gee 1$, $k=1,\ldots, m$, $i, j = 1, \ldots, n$, $i\neq j$,  we define 
	$I_n(\xi_n^m): = \sum_{k=1}^m a_k M(B^k_1)\cdots M(B^k_n).$
	
	According to \cite[Theorem 2]{It56}, we have the following It\^o chaos expansion
	\begin{align*}
		L_2(\Omega, \F_T^X, \p) =  \oplus_{n=0}^\infty  \{ I_n(\xi_n) : \xi_n \in L_2(\m^{\otimes n})\},
	\end{align*}
	where $I_0(\xi_0):= \xi_0 \in \R$. Let $\tilde \xi_n$ denote the symmetrization\footnote{$\tilde \xi_n((t_1, x_1), \ldots, (t_n, x_n)) := \frac{1}{n!}\sum_{\pi}\xi_n ((t_{\pi(1)}, x_{\pi(1)}), \ldots, (t_{\pi(n)}, x_{\pi(n)}))$,
	where the sum is taken over all permutations $\pi$ of $\{1, \ldots, n\}$.} of $\xi_n \in L_2(\m^{\otimes n})$.  It then turns out from the definition of $I_n$ that $I_n(\xi_n) = I_n(\tilde \xi_n)$ a.s. By the It\^o chaos decomposition,  $\xi \in L_2(\p)$ if and only if there are $\xi_n \in L_2(\m^{\otimes n})$ so that $\xi =  \sum_{n=0}^\infty I_n(\xi_n)$ a.s., and this  expansion is unique if every $\xi_n$ is symmetric. 
	Furthermore,  $\|\xi\|_{L_2(\p)}^2  = \sum_{n=0}^\infty n! \|\tilde \xi_n\|_{L_2(\m^{\otimes n})}^2$.

	\begin{defi}\label{Malliavin-derivative}
		Let $\bD_{1,2}$ consist of all $\xi = \sum_{n=0}^\infty I_n(\xi_n) \in L_2(\p)$ with
		\begin{align*}
			\|\xi\|_{\bD_{1,2}}^2 := \sum_{n=0}^\infty (n+1)! \|\tilde \xi_n\|_{L_2(\m^{\otimes n})}^2 <\infty.
		\end{align*}
		The \textit{Malliavin derivative operator}  $D \colon \bD_{1,2} \to L_2(\p \otimes \m)$, where $L_2(\p\otimes \m) : = L_2(\Omega \times [0, T] \times \R, \F \otimes \cB([0, T] \times \R), \p \otimes  \m)$, is defined for $\xi = \sum_{n=0}^\infty I_n(\xi_n) \in \bD_{1,2}$ by
		\begin{align*}
			D_{t, x} \xi := \sum_{n=1}^\infty n I_{n-1}(\tilde \xi_n((t, x), \cdot)), \quad (\omega, t, x) \in \Omega \times [0, T] \times \R.
		\end{align*}
	\end{defi}

The following proposition was established in \cite[Corollary 3.1 in the second article of this thesis]{La13} and it provides an equivalent condition such that a functional of $X_t$ belongs to $\bD_{1,2}$. For $X$ being the Brownian motion, see \cite[Proposition V.2.3.1]{MAKL95}, and for $X$ without a Brownian component, see \cite[Lemma 3.2]{GS16}.

	\begin{prop}[\cite{La13}] \label{lemm:Malliavin-derivative-functional}
		Let $t \in (0, T]$ and a Borel function $f\colon \R \to \R$ such that $f(X_t) \in L_2(\p)$. Then, $f(X_t) \in \bD_{1, 2}$ if and only if the following two assertions hold:
		\begin{enumerate}[\quad \rm (a)]

			\item when $\sigma >0$, $f$ has a weak derivative\footnote{A locally integrable function $h$ is called a \textit{weak derivative} (unique up to a $\Leb$-null set) of a locally integrable function $f$ on $\R$ if 
				$\int_\R f(x) \phi'(x)\od x = - \int_\R h(x) \phi(x) \od x$ for any smooth function $\phi$ with compact support in $\R$. If such an $h$ exists, then we denote $f'_w := h$.} $f'_w$ on $\R$ with $f'_w(X_t) \in L_2(\p)$,
			\item the map $(s, x) \mapsto \frac{f(X_t + x) - f(X_t)}{x}\1_{[0, t] \times \R_0}(s, x)$ belongs to $L_2(\p \otimes \m)$.
		\end{enumerate}
		Furthermore, if $f(X_t) \in \bD_{1,2}$, then for $\p\otimes \m$-a.e. $(\omega, s, x) \in \Omega \times [0, T]\times \R$ one has
		\begin{align*}
			D_{s, x}f(X_t) = f'_w(X_t)\1_{[0, t] \times \{0\}}(s, x) + \frac{f(X_t + x) - f(X_t)}{x}\1_{[0, t] \times \R_0}(s, x),
		\end{align*}
		where we set, by convention, $f'_w:=0$ whenever $\sigma =0$.
	\end{prop}

	\subsection{Preparation for the proof of \cref{theo:MVH-strategy}}
	
	Since we shall work simultaneously with the two L\'evy processes $X$ and $Z$ (under $\p$) for which it holds $\e^X = \cE(Z)$ as introduced in \Cref{subsec:exponential-Levy}, we agree on the following convention to avoid confusions and determine clearly the referred process.
	
	\begin{conven}\label{conven:notations}
		For $Y \in \{X, Z\}$,  the notations $\gamma_Y$, $\sigma_Y$, $\nu_Y$, $N_Y$, $\mu_Y$, $\m_Y$, $M_Y$, $I^Y$,  $D^Y$, $\bD^{Y}_{1,2}$ introduced in \Cref{subsection-exponential-levy-process} and \Cref{subsection-Levy-process} are assigned to $Y$.
	\end{conven}
	
	The following lemma shows that $X$ and $Z$ generate the same Malliavin--Sobolev space.
	
	\begin{lemm} \label{lemm:relation-D12-X-Z}
		One has $\bD_{1, 2}^X = \bD_{1, 2}^Z$.
	\end{lemm}
	
	\begin{proof}
	Let $\varrho \colon \R \to (-1, \infty)$ be a bijection defined by
		$\varrho(x) := \e^x -1$. Set $\frac{\e^0 -1}{0} := 1$ and $\frac{\ln (0 +1)}{0}:=1$. The relation between $\nu_X$ and $\nu_Z$ implies that, for any Borel function $w\gee 0$,
		\begin{align}\label{eq:relation-mu-XZ}
			\int_{(-1, \infty)} w(z)\mu_{Z}(\od z)   = \int_{\R} w(\varrho(x)) \left| \frac{\varrho(x)}{x}\right|^2 \mu_{X} (\od x).
		\end{align}
		Fix $n\gee 1$. We define the operator $\Psi_n \colon  L_2(\m^{\otimes n}_{X}) \to L_2(\m_{Z}^{\otimes n})$, $\Psi_n(\xi^X_n) = \xi^{Z}_n$, 
		by setting, for $((t_1, z_1), \ldots, (t_n, z_n)) \in ([0, T] \times (-1, \infty))^n$, that
		\begin{align}\label{eq:operator-definition}
			\xi^{Z}_n((t_1, z_1), \ldots, (t_n, z_n)) := \xi_n^X((t_1, \varrho^{-1}(z_1)), \ldots, (t_n, \varrho^{-1}(z_n)))\prod_{i =1}^n \frac{\varrho^{-1}(z_i)}{z_i},
		\end{align}
		and set $\xi^Z_n := 0$ otherwise. 	
		Using \eqref{eq:relation-mu-XZ} and Fubini's theorem, we obtain
		\begin{align*}
			\|\xi^{Z}_n\|_{L_2(\m^{\otimes n}_{Z})}^2  = \|\xi^X_n\|_{L_2(\m_X^{\otimes n})}^2,
		\end{align*} 
		which ensures $\xi^{Z}_n \in  L_2(\m^{\otimes n}_{Z})$, and thus $I_n^{Z}(\xi^{Z}_n)$ exists as an element in $L_2(\p)$. Consequently,  $\Psi_n$ is linear and bounded.	We next show for any $\xi^X_n \in L_2(\m_X^{\otimes n})$ that, a.s.,
		\begin{align}\label{eq:relation-multi-integral-X-Z}
			I_n^X(\xi^X_n) = I_n^{Z}(\Psi_n(\xi^X_n)) = I_n^{Z}(\xi^{Z}_n).
		\end{align}
		We prove \eqref{eq:relation-multi-integral-X-Z} only for $n=1$ since it follows for $n\gee 2$ in the same way. Let $(a, b] \subset [0, T]$ and $B \in \cB((-1, \infty))$ with $0 \notin \ol B$. Then $0 \notin \ol{\varrho^{-1}(B)}$. Since $\Delta Z = \varrho(\Delta X)$, it holds that, a.s., 
		\begin{align*}
			& \int_{[0, T] \times \R_0} \1_{(a, b] \times B}(s, x) x  N_X(\od s, \od x)   = \int_{[0, T] \times \R_0} \Psi_1(\1_{(a, b] \times B})(s,z) z N_{Z}(\od s, \od z).
		\end{align*}
		As a consequence, the expected values of both sides are equal, and hence, a.s., 
		\begin{align*}
			I_1^X(\1_{(a, b] \times B})  = I_1^{Z}(\Psi_1(\1_{(a, b] \times B})).
		\end{align*}
		Remark that the Gaussian components of $X$ and $Z$ coincide pathwise. Hence, due to the denseness in $L_2(\m_X)$ of the linear hull of $\{\1_{(a, b] \times \{0\}}, \1_{(a, b] \times B} : (a, b] \subset [0, T], B \in \cB(\R), 0 \notin \ol B\}$, together with the continuity of $I^X_1$, $I^{Z}_1$ and $\Psi_1$, we deduce that, a.s.,
		$$I_1^X(\xi_1^X) = I_1^{Z}(\Psi_1(\xi^X_1)) = I_1^{Z}(\xi^{Z}_1).$$		
		Let $\xi \in \bD_{1,2}^X$ and suppose that 
		$\xi = \sum_{n=0}^\infty I^X_n(\tilde \xi^X_n)$, where $\tilde \xi^X_n \in  L_2(\m_X^{\otimes n})$ are symmetric. By the definition of $\Psi_n$, the function $\Psi_n(\tilde \xi^X_n)$ is also symmetric. Since $\F^{Z}_T = \F^X_T$, the uniqueness of chaos expansion and \eqref{eq:relation-multi-integral-X-Z} lead to, a.s., 
		\begin{align}\label{eq:chaos-xi-XZ}
			\xi = \E \xi + \sum_{n=1}^\infty I^X_n(\tilde \xi^X_n) = \E \xi + \sum_{n=1}^\infty I^{Z}_n(\tilde \xi^{Z}_n).
		\end{align} 
		Since
		$\|\tilde \xi_n^X\|_{L_2(\m_X^{\otimes n})}^2 = \|\tilde \xi_n^{Z}\|_{L_2(\m_{Z}^{\otimes n})}^2$, it 
		implies that $\xi \in \bD^{Z}_{1, 2}$. Hence, $\bD^X_{1, 2} \subseteq \bD^{Z}_{1, 2}$.	By exchanging the role of $\varrho$ and $\varrho^{-1}$, together with the fact that $\nu_X = \nu_{Z}\circ (\varrho^{-1})^{-1}$, the converse inclusion $\bD^{Z}_{1, 2} \subseteq \bD^X_{1, 2}$ follows. Therefore, $\bD^X_{1, 2} = \bD^{Z}_{1, 2}$ as desired.
	\end{proof}

	We use \cref{assum-maringale-setting} from now until the end of this section. Recall that $Z$ is an $L_2(\p)$-martingale with zero mean, hence one can write $\od Z_t = \int_{\R} M_Z(\od t, \od z)$. 
	
	We approach  the GKW decomposition of $g(S_T) \in L_2(\p)$ by means of chaos expansion with respect to the \textit{L\'evy process} $Z$ in the way introduced in \cite{GGL13} as follows. First, it is known that (see, e.g., \cite[Definiton 1 and Lemma 1]{GGL13}), a.s.,
	\begin{align*}
		S_T = 1 + \sum_{n=1}^\infty I^Z_n\bigg(\frac{\1_{[0, T] \times \R}^{\otimes n}}{n!}\bigg),
	\end{align*}
	where the kernels in the chaos expansion of $S_T$ do not depend on the time variables. According to  \cite[Theorem 4]{BG17}, this property is preserved for $g(S_T) \in L_2(\p)$.  Namely, a.s.,
	\begin{align}\label{eq:chaos-expansion-g}
		g(S_T)  = \sum_{n=0}^\infty I^Z_n\(\tilde g_n \1_{[0, T]}^{\otimes n}\),
	\end{align}
	where $\tilde g_n \in \tilde L_2(\mu_Z^{\otimes n})$. For each $n\geqslant 1$, define the function $\tilde h_{n-1} \in \tilde L_2(\mu_Z^{\otimes(n-1)})$ by 
	\begin{align}\label{eq:integrand-of-varphi}
		\tilde h_{n-1}(z_1, \ldots, z_{n-1}): = \int_{ \R} \tilde g_n(z_1, \ldots, z_{n-1}, z) \frac{\mu_Z(\od z)}{\mu_Z(\R)}.
	\end{align}

	\begin{defi}\label{def:martingale-varphi}
		\begin{enumerate}[\rm (1)]
			\item Let $\varphi^g = (\varphi^g_t)_{t\in [0, T)}$ be a c\`adl\`ag version of the $L_2(\p)$-martingale
			\begin{align*}
				\bigg(\tilde h_0 + \sum_{n=1}^\infty (n+1) I^Z_n\(\tilde h_n \1_{[0, t]}^{\otimes n}\)\bigg)_{t\in [0, T)},
			\end{align*}
			where the infinite sum is taken in $L_2(\p)$.
			\item Define the process $\vartheta^g \in \CL([0, T))$ by setting $\vartheta^g : = \varphi^g/S$.
		\end{enumerate}

	\end{defi}
	
	\begin{lemm} \label{lemm:gradient-GKW} Let $g(S_T) \in L_2(\p)$. Then $\vartheta^g_-$ is a MVH strategy corresponding to $g(S_T)$.
	\end{lemm}
	
	\begin{proof} We use the functions $\tilde g_n, \tilde h_n$ defined in \eqref{eq:chaos-expansion-g}--\eqref{eq:integrand-of-varphi}. Since each element in $\tilde L_2(\m_Z^{\otimes n})$ is symmetric, we only need to define it on $((t_1, z_1), \ldots, (t_n, z_n))$ with $0<t_1<\cdots < t_n <T$. Thus, for $n\gee 2$, we define $\tilde k_n \in \tilde L_2(\m_Z^{\otimes n})$ by
		\begin{align*}
			\tilde k_n((t_1, z_1), \ldots, (t_n, z_n)) : = \tilde h_{n-1}(z_1, \ldots, z_{n-1})\quad\mbox{for } 0 <t_1<\cdots<t_n<T,
		\end{align*}
		and set $\tilde k_1(t,z) : = \tilde h_0$. According to the argument  in \cite[Eqs. (7)--(10)]{GGL13}, it holds that the stochastic integral $\int_0^T \varphi^g_{t-}\od Z_t$ is well-defined and
		\begin{align*}
			\int_0^T \varphi^g_{t-}\od Z_t = \sum_{n=1}^\infty I_{n}^Z (\tilde k_{n}).
		\end{align*}
		Let $L^g=(L^g_t)_{t\in [0, T]}$ be the c\`adl\`ag version of the martingale closed by
		\begin{align*}
			L_T^g: = g(S_T) - \E g(S_T) - \int_0^T \varphi^g_{t-}\od Z_t.
		\end{align*}
		Then $g(S_T)$ can be re-written as
		\begin{align*}
			g(S_T)  & = \E g(S_T) + \int_0^T \varphi^g_{t-}\od Z_t + L_T^g  =  \E g(S_T) + \sum_{n=1}^\infty I_{n}^Z (\tilde k_{n}) + L_T^g.
		\end{align*} 
		We now show that $\<L^g, Z\> =0$. For $t\in (0, T]$, one has, a.s., 
		\begin{align*}
			L^g_t =  \sum_{n=1}^\infty I^Z_n\(\tilde g_n \1_{[0, t]}^{\otimes n}\) - \sum_{n=1}^\infty I^Z_n\(\tilde k_{n} \1_{[0, t]}^{\otimes n}\)  = \sum_{n=1}^\infty I^Z_n\((\tilde g_n - \tilde k_n)\1_{[0, t]}^{\otimes n}\).
		\end{align*}
		Since $Z_t = \int_0^t \od Z_s = \int_0^t \int_{\R} M_Z(\od s, \od z)$, one has  for any $t\in (0, T]$ and $n \gee 1$ that, a.s.,
		\begin{align*}
			\left\<I^Z_n\((\tilde g_n - \tilde k_n)\1_{[0, \cdot]}^{\otimes n}\), Z\right\>_t & = n\int_0^t \!\! \int_{\R} I_{n-1}^Z\((\tilde g_n(\cdot, z) - \tilde k_n(\cdot, (s, z)))\1_{[0, s]}^{\otimes (n-1)}\)\mu_Z(\od z) \od s\\
			& = n\int_0^t  I_{n-1}^Z\(\int_{\R}(\tilde g_n(\cdot, z) - \tilde k_n(\cdot, (s, z)))\mu_Z(\od z) \1_{[0, s]}^{\otimes (n-1)}\) \od s\\
			& = 0,
		\end{align*} 
		where one can see that the second equality holds by testing with multiple integrals. Since the infinite sum in the chaos representation of $L^g_t$ is taken in $L_2(\p)$, we conclude that $\<L^g, Z\> =0$.
		Hence, it follows from $\vartheta^g_{t-}\od S_t = \varphi^g_{t-}\od Z_t$ that $g(S_T) = \E g(S_T) + \int_0^T \vartheta^g_{t-} \od S_t + L_T^g$ is the GKW decomposition of $g(S_T)$. 
	\end{proof}

	\subsection{Proof of \cref{theo:MVH-strategy}} We verify that the process $\vartheta^g$ in \cref{def:martingale-varphi} satisfies the requirements. The assertion \eqref{item:prop:strategy} and the martingale property of $\vartheta^gS$ are clear by the definition of $\vartheta^g$ and \cref{lemm:gradient-GKW}. For the latter part of \eqref{item:prop:maringale-property}, since $\varphi^g$ and $S$ are martingales adapted to the quasi-left continuous filtration $\bF^X$, it implies that $\varphi^g_t = \varphi^g_{t-}$ a.s. and $S_{t} = S_{t-}$ a.s. for each $t\in [0, T)$ (see \cite{Pr05}). Therefore, $\vartheta^g_t = \vartheta^g_{t-}$ a.s. for each $t \in [0, T)$.	
	
	
	\eqref{item:prop:gradient-formula-general} Recall from \cref{lemm:relation-D12-X-Z} that $\bD_{1,2}^X= \bD_{1,2}^Z$. We have in \cref{def:martingale-varphi} and \cref{lemm:gradient-GKW} the strategy given as chaos expansion with respect to $Z$. In order to get the explicit representation  \eqref{eq:integrand-general}, we change it into a representation with respect to $X$ where we can use \cref{lemm:Malliavin-derivative-functional}.
	
	\textit{Step 1}. Let $\xi \in \bD_{1,2}^Z
	$ have the expansion \eqref{eq:chaos-xi-XZ}.  We first write the Malliavin derivative of $\xi$ as the element in $L_2(\p \otimes \m_Z)$ and then integrate it with respect to $\m_Z$ to obtain, a.s.,
	\begin{align}
		&\int_0^T \!\!\! \int_{ \R} \(D^Z_{s,z}\xi\) \m_Z(\od s, \od z) \notag  \\
		& =   \int_0^T \!\!\!\int_{ \R}\bigg( L_2(\p \otimes \m_Z)\mbox{-}\lim_{N\to \infty} \sum_{n=1}^N n I^Z_{n-1}(\tilde \xi^Z_n((s, z), \cdot))\bigg)\m_Z(\od s, \od z) \notag\\
		& = L_2(\p)\mbox{-}\lim_{N\to \infty} \sum_{n=1}^N\int_0^T \!\!\!\int_{ \R} n I^Z_{n-1}(\tilde \xi^Z_n((s, z), \cdot))\m_Z(\od s, \od z) \label{eq:interchange-limit-1}\\
		& = L_2(\p)\mbox{-}\lim_{N\to \infty} \sum_{n=1}^N\int_0^T \!\!\!\int_{ \R} n I^X_{n-1}(\tilde \xi^X_n((s, x), \cdot))\frac{\e^x -1}{x} \m_X(\od s, \od x) \label{eq:interchange-X-Z} \\
		& =   \int_0^T \!\!\!\int_{ \R}\bigg( L_2(\p \otimes \m_X)\mbox{-}\lim_{N\to \infty} \sum_{n=1}^N n I^X_{n-1}(\tilde \xi^X_n((s, x), \cdot))\bigg) \frac{\e^x -1}{x}\m_X(\od s, \od x) \label{eq:interchange-limit-2}\\
		& = \int_0^T \!\!\! \int_{ \R} \(D^X_{s,x}\xi\) \frac{\e^x-1}{x}\m_X(\od s, \od x), \label{eq:integral-Malliavain-X-Z}
	\end{align}
	where one uses the fact that $\m_Z([0, T]\times \R) = \int_0^T\!\!\int_{\R}\big|\frac{\e^x -1}{x}\big|^2\m_X(\od s, \od x) <\infty$  to derive \eqref{eq:interchange-limit-1} and \eqref{eq:interchange-limit-2}.  In order to achieve \eqref{eq:interchange-X-Z}, we apply the definition of $\Psi_{n-1}$  in \eqref{eq:operator-definition} and then use \eqref{eq:relation-multi-integral-X-Z} with the convention that $\Psi_0$ is the identical map on $\R$.

	\textit{Step 2}. For $x\in \R, t\in (0, T)$, we define
	\begin{align*}
		f(x) : = g(\e^x) \quad \mbox{and}\quad F(t, x) : = G(t, \e^x).
	\end{align*}
	It turns out that $F(t, X_t) = \ce{\F_t}{f(X_T)}$ a.s. We then derive from \cite[Lemma D.1]{GN20} that  $F(t, X_t) \in \bD_{1, 2}^X$. Applying \cref{lemm:Malliavin-derivative-functional}, we obtain
	\begin{align}\label{eq:Malliavin-derivative-conditional-expectation}
		D^X_{s, x} F(t, X_t) = \pd_x F(t, X_t)\1_{[0, t] \times \{0\}}(s, x) + \frac{F(t, X_t + x) - F(t, X_t)}{x} \1_{[0, t] \times \R_0}(s, x)
	\end{align}
	for $\p \otimes \m_X$-a.e. $(\omega, s, x) \in \Omega \times [0, T]\times \R$. We multiply both sides of \eqref{eq:Malliavin-derivative-conditional-expectation} with $\frac{\e^x -1}{x}$ and then integrate them with respect to  $\m_X$ to obtain, a.s., 
	\begin{align}\label{eq:integrate-Malliavin-derivative-1}
		&   \int_0^T \!\!\!\int_{ \R} \(D^X_{s, x} F(t, X_t) \frac{\e^x-1}{x}\) \m_X(\od s, \od x) \notag\\
		& = t \int_{ \R} \(\pd_x F(t, X_t)\1_{\{x=0\}} + \frac{F(t, X_t + x) - F(t, X_t)}{x} \frac{\e^x-1}{x} \1_{\{x\neq 0\}}\) \mu_X(\od x) \notag\\
		& = t \(\sigma^2 S_t \pd_y G(t, S_t) + \int_{ \R} (G(t, \e^{x}S_t)- G(t,S_t)) (\e^x -1)\nu_X(\od x)\).
	\end{align}
	On the other hand, for the representation of $g(S_T)$ given in \eqref{eq:chaos-expansion-g}, taking the conditional expectation of $g(S_T)$ with respect to $\F_t$ yields, a.s., 
	\begin{align*}
		G(t, S_t) = \ce{\F_t}{g(S_T)} = \sum_{n=0}^\infty I^Z_n\(\tilde g_n \1_{[0, t]}^{\otimes n}\).
	\end{align*}
	Since $G(t, S_t)\in \bD_{1, 2}^Z$, we write the chaos representation of the Malliavin derivative of $G(t, S_t)$ with respect to the underlying process $Z$ as in \cref{Malliavin-derivative}, and then, integrate that with respect to the measure $\m_Z$ to obtain, a.s.,
	\begin{align}\label{eq:integrate-Malliavin-derivative-2}
		\int_0^T\!\!\!\int_{ \R}  \(D^Z_{s, z} G(t, S_t)\) \m_Z(\od s, \od z) & =  \int_0^T\!\!\!\int_{ \R} \bigg(\sum_{n=1}^\infty n I^Z_{n-1}\(\tilde g_n(\cdot, z) \1_{[0, t]}^{\otimes (n-1)} \1_{[0, t]}(s)\) \bigg) \m_Z(\od s, \od z) \notag\\
		& = \sum_{n=1}^\infty n I^Z_{n-1}\(\int_0^T\!\!\!\int_{\R} \tilde g_n(\cdot, z) \1_{[0, t]}^{\otimes (n-1)} \1_{[0, t]}(s) \m_Z(\od s, \od z)\) \notag\\
		& = t \sum_{n=1}^\infty nI^Z_{n-1}\(\(\int_{ \R}\tilde g_n(\cdot, z) \mu_Z(\od z)\) \1_{[0, t]}^{\otimes (n-1)}\) \notag\\
		& = t c_{\eqref{eq:integrand-general}}^2 \vartheta^g_t S_t,
	\end{align}
	where the last equality comes from \eqref{eq:integrand-of-varphi},  \cref{def:martingale-varphi}, and $\mu_Z(\R) = c_{\eqref{eq:integrand-general}}^2$. Applying \textit{Step 1} for $\xi  = F(t, X_t) = G(t, S_t)$, we derive from \eqref{eq:integral-Malliavain-X-Z} that, a.s., 
	\begin{align}\label{eq:integrate-Malliavin-derivative-3}
		\int_0^T\!\!\!\int_{ \R} \(D^X_{s, x} F(t, X_t) \frac{\e^x-1}{x}\) \m_X(\od s, \od x) = \int_0^T\!\!\!\int_{ \R}  \(D^Z_{s, z} G(t, S_t)\) \m_Z(\od s, \od z).
	\end{align}
	Combining \eqref{eq:integrate-Malliavin-derivative-1}, \eqref{eq:integrate-Malliavin-derivative-2} with \eqref{eq:integrate-Malliavin-derivative-3}, we get \eqref{eq:integrand-general}.

	\appendix

	\section{Regularity of the weight processes $\ol \Phi$ and $\Phi(\eta)$} \label{sec:weight-regularity}
	
	We recall $\ol\Phi$ from \eqref{defi-Phi-bar} and $\cSM_p$ from \cref{definition:weighted_bmo}.

	\begin{prop} \label{lemm:properties-BMO-SM} 
		\begin{enumerate}[\rm (1)]
			\item \label{item:Holder-inequality-SM} Let $p, q, r \in (0, \infty)$ with $\frac{1}{r} = \frac{1}{p} + \frac{1}{q}$. Then, for any $\Phi, \Psi \in \CL^+([0, T])$,
			$$\|\Phi \Psi\|_{\cSM_r(\p)} \lee \|\Phi\|_{\cSM_p(\p)} \|\Psi\|_{\cSM_q(\p)}.$$
			
			\item \label{item:bar-Phi-SM} If $\Phi \in \cSM_p(\p)$ for some $p \in (0, \infty)$, then  $\ol \Phi \in \cSM_p(\p)$ with
			\begin{align*}
				\|\ol \Phi\|_{\cSM_p(\p)} \lee \begin{cases}
					3 \|\Phi\|_{\cSM_p(\p)} +1 & \mbox{ if } p \in [1, \infty)\\
					(3\|\Phi\|_{\cSM_p(\p)}^p +1)^{\frac{1}{p}} & \mbox{ if } p \in (0, 1).
				\end{cases}
			\end{align*}
		\end{enumerate}
	\end{prop}

	\begin{proof} Item \eqref{item:Holder-inequality-SM} is given in \cite[Proposition A.2]{GN20}. We now prove Item \eqref{item:bar-Phi-SM}. Let $a \in [0, T)$ be arbitrary.  For $ p\in [1, \infty)$, applying  the conditional Minkovski inequality yields, a.s., 
		\begin{align*}
			\(\ce{\F_a}{\ts\sup_{t \in [a, T]} \ol \Phi_t^p}\)^{\frac{1}{p}} & \lee  \(\ce{\F_a}{\ts \sup_{t \in[a, T]} \Phi_t^p}\)^{\frac{1}{p}} + \(\ce{\F_a}{\ts \sup_{ s\in [0, T]} |\Delta \Phi_s|^p}\)^{\frac{1}{p}} \\
			& \lee \|\Phi\|_{\cSM_p(\p)} \Phi_a + \ts \sup_{s \in [0, a]}|\Delta \Phi_s| + \(\ce{\F_a}{\sup_{t \in (a, T]} |\Delta \Phi_t|^p}\)^{\frac{1}{p}}\\
			& \lee \|\Phi\|_{\cSM_p(\p)} \Phi_a + \ts \sup_{s\in[0, a]}|\Delta \Phi_s| + 2\(\ce{\F_a}{\sup_{t \in (a, T]} \Phi_t^p}\)^{\frac{1}{p}}\\
			& \lee (3 \|\Phi\|_{\cSM_p(\p)}+1) \ol \Phi_a.
		\end{align*}
		
		For $ p \in (0, 1)$, we use the same argument as in the previous case where one applies the inequality $|x+y|^p \lee |x|^p + |y|^p$ for $x, y \in \R$ to obtain, a.s., 
		\begin{align*}
			\ce{\F_a}{\ts \sup_{t \in [a, T]} \ol\Phi_t^p} \lee (3\|\Phi\|_{\cSM_p(\p)}^p + 1) \ol \Phi_a^p.
		\end{align*}
		Hence, the desired conclusion follows. 
	\end{proof}
	
	Recall the L\'evy process $X$ with characteristic triplet $(\gamma, \sigma, \nu)$ and exponent $\psi$ mentioned in \Cref{subsection-exponential-levy-process}. Recall $\Phi(\eta)$ from \eqref{eq:definiteion-weight-process} and $S= \e^X$.
	
	\begin{prop} \label{prop:regularity-weight-process} If $\int_{|x|>1} \e^{qx} \nu(\od x) <\infty$ for some $q \in (1, \infty)$, then $\Phi(\eta) \in \cSM_q(\p)$  for all $\eta \in [0, 1]$. Moreover, 
		\begin{align*}
			\|\Phi(\eta)\|_{\cSM_q(\p)}^q \lee \e^{T|\psi(-\im)|(2q+1)} 2^{1-\eta} \(\frac{q}{q-1}\)^{2q} \|S_T\|_{L_q(\p)}^q.
		\end{align*}
	\end{prop}

	\begin{proof}  The first step considers the particular case when $S$ is a martingale, and the general  case is handled in the second step.

		\textit{Step 1}. Assume that $S$ is a $\p$-martingale. According to \cite[Theorem 25.3]{Sa13}, $\int_{|x|>1} \e^{qx} \nu(\od x) <\infty$  implies $\e^{X_t} \in L_q(\p)$ for all $t>0$. 	
		Denote $c_q: = (\frac{q}{q-1})^q$ and define 
		$M= (M_t)_{t\in [0, T]}$ by 
		\begin{align*}
			\ts M_t := \sup_{u \in [0, t]}\e^{X_t - X_u}.
		\end{align*}
		We show that $M$ is a positive $L_q(\p)$-submartingale. The adaptedness and positivity are clear. Pick a $t\in (0, T]$. Since $(X_t - X_{t-u})_{u\in [0, t]}$ is \textit{c\`agl\`ad} (left-continuous with right limits) and $(X_{u})_{u\in [0, t]}$ is c\`adl\`ag, and both processes have the same finite-dimensional distribution, applying Doob's maximal inequality yields
		\begin{align}\label{eq:estimate-M_t}
			\E M_t^q & = \ts 
			\E \[ \sup_{u \in [0, t]} \e^{q(X_t - X_u)}\] = \E \[\sup_{u \in [0, t]} \e^{q(X_t-X_{t-u})}\] \\
			& \ts = \E \[\sup_{u\in [0, t]} \e^{qX_{u}}\]	\lee c_q  \E \e^{q X_t} <\infty. \notag
		\end{align}
		For $0 \lee s\lee t\lee T$ one has, a.s., 
		\begin{align*}
			\ce{\F_s}{M_t} & \ts \gee  \ce{\F_s}{\sup_{u \in [0,s]}\e^{X_t - X_u}} = \sup_{u \in [0,s]} \e^{X_s - X_u} \E\e^{X_t - X_s} = M_s,
		\end{align*}
		where we use $\E\e^{X_t - X_s} = \E S_{t-s} = 1$.
		
		We observe that the process $\Phi(\eta)$ can be re-written as
		\begin{align*}
			\ts \Phi(\eta)_t  =  \e^{\eta X_t} \sup_{s \in [0, t]}\e^{(1-\eta)(X_t - X_s)}  = \e^{\eta X_t}  M_t^{1-\eta}.
		\end{align*}
		Let us fix $\eta \in (0, 1)$ and $a\in [0, T]$. For $\e^{\eta X} = (\e^{\eta X_t})_{t \in [0, T]}$, applying Doob's maximal inequality and Jensen's inequality we obtain that, a.s., 
		\begin{align*}
			\ts \ce{\F_a}{\sup_{t\in [a, T]} (\e^{\eta X_t})^{\frac{q}{\eta}}} & \ts =   \e^{qX_a} \E \[\sup_{t \in [a, T]}\e^{q(X_t - X_a)}\]  \lee c_q \e^{q X_a} \E \e^{q(X_T - X_a)} \\
			& \ts =  c_q \e^{q X_a} \E \e^{qX_{T -a}} \lee c_q \e^{qX_a} \E \e^{q X_T},
		\end{align*}
		which implies 
		\begin{align*}
			\|\e^{\eta X}\|_{\cSM_{q/\eta}(\p)} \lee (c_q \E \e^{qX_T})^{\frac{\eta}{q}}.
		\end{align*}
		For $M^{1-\eta} = (M^{1-\eta}_t)_{t\in [0, T]}$, one has that, a.s.,
		\begin{align*}
			\ts \ce{\F_a}{\sup_{t \in [a, T]}(M_t^{1-\eta})^{\frac{q}{1-\eta}}} & =  \ts \ce{\F_a}{\sup_{t \in [a, T]}M_t^q} \lee c_q \ce{\F_a}{M_T^q}\\
			&  \ts \lee c_q \ce{\F_a}{\sup_{s \in [0, a]} \e^{q(X_T - X_s)}} + c_q\ce{\F_a}{\sup_{s \in [a, T]} \e^{q(X_T - X_s)}}\\
			& \ts =   c_q \sup_{s \in [0, a]}\e^{q(X_a - X_s)} \E\e^{q(X_T - X_a)} + c_q \E \[\sup_{s \in [a, T]} \e^{q(X_T - X_s)}\]\\
			& \ts \lee   2 c_q \sup_{s \in [0, a]}\e^{q(X_a - X_s)} \E \[\sup_{s \in [a, T]} \e^{q(X_T - X_s)}\]\\
			&  \ts \lee  \(2 c_q\E \[\sup_{s \in [0, T]} \e^{q(X_T - X_s)}\]\)  M_a^q\\
			& \lee \(2 c_q^2 \E \e^{qX_T}\) M_a^q,
		\end{align*}
		where the conditional Doob maximal inequality is applied for the positive sub-martingale $M$ to obtain the first inequality, and the last one comes from \eqref{eq:estimate-M_t}. Hence,
		\begin{align*}
			\|M^{1-\eta}\|_{\cSM_{q/(1-\eta)}(\p)} \lee (2 c_q^2 \E \e^{qX_T})^{\frac{1-\eta}{q}}. 
		\end{align*}
		Applying \cref{lemm:properties-BMO-SM}\eqref{item:Holder-inequality-SM} with $\frac{1}{q} = \frac{1}{q/\eta} + \frac{1}{q/(1-\eta)}$, we obtain
		\begin{align*}
			\|\Phi(\eta)\|_{\cSM_q(\p)} \lee \|\e^{\eta X}\|_{\cSM_{q/\eta}(\p)} \|M^{1-\eta}\|_{\cSM_{q/(1-\eta)}(\p)} \lee 2^{\frac{1-\eta}{q}} \(\frac{q}{q-1}\)^2 \|S_T\|_{L_q(\p)} <\infty,	
		\end{align*}
		which asserts $\Phi(\eta)\in \cSM_q(\p)$. 	When $\eta =0$ or $\eta =1$, the desired conclusion is straightforward as $\Phi(0) = M$, $\Phi(1) = \e^X$.
		
		\textit{Step 2}. In the general case, we define
		\begin{align*}
			\wt S_t : = \e^{t \psi(-\im)} S_t.
		\end{align*}
		Then it is known that $\wt S$ is a martingale under $\p$. Some standard calculations yield
		\begin{align*}
			\e^{-T|\psi(-\im)|} \wt \Phi(\eta)_t \lee \Phi(\eta)_t \lee \e^{T|\psi(-\im)|}\wt \Phi(\eta)_t,
		\end{align*}
		where $\wt \Phi(\eta)_t : = \wt S_t \sup_{u \in [0, t]}(\wt S_u^{\eta -1})$. Applying \textit{Step 1} for $\p$-martingale $\wt S$ we derive that $\wt \Phi(\eta) \in \cSM_q(\p)$. Hence, for $a\in [0, T]$, one has, a.s., 
		\begin{align*}
			\ce{\F_a}{\ts \sup_{t\in [a, T]} \Phi(\eta)_t^q} & \lee \e^{qT|\psi(-\im)|} \ce{\F_a}{\ts \sup_{t \in [a, T]} \wt \Phi(\eta)_t^q}\\
			& \lee  \e^{qT|\psi(-\im)|}\|\wt \Phi(\eta)\|_{\cSM_q(\p)}^q \wt \Phi(\eta)_a^q\\
			& \lee \e^{2qT|\psi(-\im)|} 2^{1-\eta} \(\frac{q}{q-1}\)^{2q} \|\wt S_T\|_{L_q(\p)}^q \Phi(\eta)_a^q\\
			& \lee \e^{T|\psi(-\im)|(2q+1)} 2^{1-\eta} \(\frac{q}{q-1}\)^{2q} \|S_T\|_{L_q(\p)}^q \Phi(\eta)_a^q,
		\end{align*}
		which proves the desired conclusion.
	\end{proof}

	
	\section{Gradient type estimates for a L\'evy semigroup on H\"older spaces}\label{sub-sec:estimae-semigroup}
	This section provides some gradient type estimates in the L\'evy setting for proving \cref{theo:application-Levy}, and they might also be of independent interest.
	
	Let us introduce some notations. For a non-empty and open set $U \subseteq \R$ and for $n\gee 1$, let $C^n(U)$ denote the family of $n$ times continuously differentiable functions on $U$, and set $C^\infty(U) := \cap_{n\gee 1} C^n(U)$. The space $C^\infty_c(U)$ consists of all $f\in C^\infty(U)$ with compact support in $U$. Let $C_0^\infty(\R)$ denote the family of all $f \in C^\infty(\R)$ with $\lim_{|x| \to \infty} f^{(n)}(x) = 0$  for all $n\gee 0$. 
	
	For $s\in \R$, we define the weighted Lebesgue measure $\Leb_s$ on $\mathcal B(\R)$ by setting
	\begin{align*}
		\Leb_s(\od x) : = \e^{sx}\od x.
	\end{align*}

	\subsection{An integral estimate for H\"older functions}
	
	For a Borel function $g$ and a random variable $Y$ such that $\E |g(y \e^Y)| <\infty$ for all $y>0$, we define
	\begin{align*}
		G(y) : = \E g(y \e^Y),\quad y>0.
	\end{align*}
	For later use, we establish in this part an estimate for $|G(z) - G(y)|$, where $g$ is a H\"older continuous function or a bounded Borel function.
	
	\begin{prop}\label{thm:upper-bound-general-Holder}
		Let $\eta \in [0, 1]$ and $g \in C^{0, \eta}(\R_+)$. If $Y$ has a density $p \in C^1(\R) \cap L_1(\R, \Leb_\eta)$ with the derivative $p' \in L_1(\R) \cap L_1(\R, \Leb_\eta)$, then for all $z, y >0$,
		\begin{align}\label{eq:pointwise-upper-estimate-Holder-1}
			|G(z) - G(y)|  \lee 
			\(|g|_{C^{0, \eta}(\R_+)}  \|p'\|_{L_1(\R)}^{1-\eta}  \inf_{\kappa>0} \left|\int_{\R} |\e^x - \kappa| |p'(x)| \od x\right|^\eta\)\frac{|z^\eta - y^\eta|}{\eta},
		\end{align}
		where we set, by convention, $\frac{|z^0 - y^0|}{0} : = \lim_{\eta \downarrow 0} \frac{|z^\eta - y^\eta|}{\eta} = \left| \ln z - \ln y\right|$ when $\eta =0$.
	\end{prop}
	
	\begin{proof}  The assumption $p \in L_1(\R, \Leb_\eta)$ means that $\E \e^{\eta Y} <\infty$, and hence $\E|g(y\e^Y)|<\infty$ for all $y>0$. Let us pick a constant $\kappa'>0$ arbitrarily. By a change of variables,
		\begin{align*}
			G(z) - G(y) & =\E g(z\e^{Y}) - \E g(y\e^{Y})\\
			& = \int_0^\infty  g(u)(p(\ln u - \ln z) - p(\ln u - \ln y))\frac{\od u}{u}\\
			& = \int_0^\infty (g(u) - g(\kappa'))(p(\ln u - \ln z) - p(\ln u - \ln y))\frac{\od u}{u}.
		\end{align*}
		Since $p\in C^1(\R)$, the fundamental theorem of calculus gives
		\begin{align*}
			G(z) - G(y) & = \int_0^\infty (g(u) - g(\kappa'))\((\ln y - \ln z) \int_0^1 p'(\ln u - \ln y + r(\ln y - \ln z)) \od r\)\frac{\od u}{u}\\
			& = (\ln y - \ln z) \int_0^\infty (g(u) - g(\kappa'))\(\int_0^1 p' (\ln u -\ln y + r(\ln y - \ln z)) \od r\)\frac{\od u}{u}.
		\end{align*}
		Since $|g(u) - g(\kappa')| \lee |g|_{C^{0, \eta}(\R_+)}|u-\kappa'|^\eta$, where $0^0:=1$, we have
		\begin{align}\label{eq:pointwise-upper-estimate-Holder-2}
			&|G(z) - G(y)| & \notag\\
			&\lee |g|_{C^{0, \eta}(\R_+)} \left|\ln z - \ln y\right| \int_0^\infty |u - \kappa'|^\eta \int_0^1 | p'(\ln u - \ln y + r(\ln y - \ln z))|\od r \frac{\od u}{u} \notag\\
			& = |g|_{C^{0, \eta}(\R_+)} \left|\ln z - \ln y\right| \int_0^1 \( \int_0^\infty |u - \kappa'|^\eta | p'(\ln u - \ln y + r(\ln y - \ln z))|   \frac{\od u}{u}\) \od r \notag\\
			& = |g|_{C^{0, \eta}(\R_+)} \left|\ln z - \ln y\right| \int_0^1 \( \int_{\R} |\e^x y^{1-r} z^r - \kappa'|^\eta |p'(x)|   \od x\) \od r.
		\end{align}
		If $\eta = 0$, then \eqref{eq:pointwise-upper-estimate-Holder-1} is obvious in the view of \eqref{eq:pointwise-upper-estimate-Holder-2}. Let us now consider $\eta \in (0, 1]$. Thanks to \eqref{eq:pointwise-upper-estimate-Holder-2}, $G$ is locally Lipschitz on $\R_+$, which implies the absolute continuity of  $G$ on any compact interval of $\R_+$. Consequently,  $G$ is differentiable $\Leb$-a.e. on $\R_+$. Let $y>0$ be such that $G'(y)$ exists and is finite. We divide both sides of \eqref{eq:pointwise-upper-estimate-Holder-2} by $|z-y|$ and then let $z\to y$, where the dominated convergence theorem is applicable on the right-hand side due to $p' \in L_1(\R) \cap L_1(\R, \Leb_\eta)$, to derive that, for all $\kappa' >0$,
		\begin{align*}
			|G'(y)| \lee |g|_{C^{0, \eta}(\R_+)} y^{-1} \int_{\R}  |y\e^x - \kappa'|^\eta |p'(x)|   \od x.
		\end{align*}
		Hence, for any $\kappa >0$, we obtain by choosing $\kappa' = y \kappa$ that
		\begin{align*}
			|G'(y)|\lee |g|_{C^{0, \eta}(\R_+)} y^{\eta-1}  \int_{\R} |\e^x - \kappa|^\eta |p'(x)|\od x.
		\end{align*} 
		Now, for $z, y >0$, using the fundamental theorem of (Lebesgue integral) calculus yields
		\begin{align*}
			|G(z) - G(y)| & = \left|\int_y^z G'(u)\od u\right| \lee  \sign(z-y) \int_y^z|G'(u)| \od u\\
			& \lee |g|_{C^{0, \eta}(\R_+)}  \sign(z-y) \int_y^z u^{\eta -1} \od u \int_{\R} |\e^x - \kappa|^\eta |p'(x)| \od x\\
			& = |g|_{C^{0, \eta}(\R_+)}  \frac{|z^\eta - y^\eta|}{\eta} \int_{\R} |\e^x - \kappa|^\eta |p'(x)| \od x\\
			& \lee \(|g|_{C^{0, \eta}(\R_+)}  \|p'\|_{L_1(\R)}^{1-\eta}  \left|\int_{\R} |\e^x - \kappa| |p'(x)| \od x\right|^\eta\)\frac{|z^\eta - y^\eta|}{\eta},
		\end{align*}
		where one applies H\"older's inequality with $\frac{1}{1/\eta} + \frac{1}{1/(1-\eta)} =1$ to obtain the last estimate. By taking the infimum over $\kappa>0$, \eqref{eq:pointwise-upper-estimate-Holder-1} follows.
	\end{proof}

	\subsection{H\"older estimates for a L\'evy semigroup}
	Let $X = (X_t)_{t \gee 0}$ be a L\'evy process with characteristic triplet $(\gamma, \sigma, \nu)$ and  exponent $\psi$ as in \Cref{subsection-exponential-levy-process}. Let us define
	\begin{align*}
		\scrD_{\exp} : = \{g \colon \R_+ \to \R \mbox{ Borel} : \;\E |g(y\e^{X_t})| <\infty \mbox{ for all } y >0, t\gee 0\}.
	\end{align*}
	It is clear that $\scrD_{\exp}$  depends on the distribution of $X$. For example, if $\int_{|x|>1} \e^{r x} \nu(\od x) <\infty$ for some $r\in \R$, then any Borel function $g$ with $\sup_{y>0}(1+y)^{-r}|g(y)| <\infty$ belongs to $\scrD_{\exp}$ because of \cite[Theorem 25.3]{Sa13}. For $t\gee 0$, define the mapping $P_t\colon \scrD_{\exp} \to \scrD_{\exp}$ by
	\begin{align*}
		P_tg(y) : = \E g(y \e^{X_t}).
	\end{align*}
	Since $P_{t+s} = P_t \circ P_s$ for any $s, t\gee 0$, the family $(P_t)_{t \gee 0}$ is a semigroup on $\scrD_{\exp}$.

	To be able to estimate the integral term of the MVH strategy formula \eqref{eq:integrand-general}, we aim to establish an estimate for
	$$|P_tg(z) - P_tg(y)|,$$
	where $g$ is bounded or H\"older continuous. We first need the following lemma.
	\begin{lemm}\label{lemm-growth-transition-densities}
		Let $X$ be a L\'evy process with characteristic exponent $\psi$. If 
		\begin{align}\label{lemm:eq:lemm-growth-transition-densities}
			0  < \liminf_{|u| \to \infty} |u|^{-\alpha} \re\psi(u) \lee \limsup_{|u| \to \infty} |u|^{-\alpha} \re \psi(u) <\infty
		\end{align}
		for some $\alpha \in (0, 2)$, then $X$ has transition densities $(p_t)_{t>0} \subset C^\infty_0(\R)$ such that 
		\begin{align*}
			\ts \sup_{t \in (0, T]} t^{\frac{1}{\alpha}}\|\pd_x p_t\|_{L_1(\R)} <\infty, \quad T>0.
		\end{align*}
	\end{lemm}
	\begin{proof}
		See the proof of \cite[Example 9.32]{GN20}.
	\end{proof}

	Since we aim to apply \cref{thm:upper-bound-general-Holder}, and in order to handle the quantity involving the infimum in  \eqref{eq:pointwise-upper-estimate-Holder-1}, we provide in \cref{lemm:extra-quantity} below estimates
	under assumptions which are typically satisfied in applications.

	\begin{lemm}\label{lemm:extra-quantity} For some $t>0$ such that $X_t$ has a differentiable density $p_t$ on $\R$, we define
		\begin{align*}
			K_t: = \inf_{\kappa >0} \int_{\R} |\e^x - \kappa| \left|\pd_x p_t(x)\right| \od x \in [0, \infty].
		\end{align*}
		\begin{enumerate}[\rm(1)]
			\item \label{item:1:sigma>0} If $\sigma >0$, then $K_t \lee \frac{1}{\sigma\sqrt{t}} \|\e^{X_t}-1\|_{L_2(\p)}$ for all $t>0$.
			\item \label{item:2:unimodal} If there is an $m_t \in \R$ such that $p_t$ is non-decreasing on $(-\infty, m_t)$ and non-increasing on $(m_t, \infty)$, then $K_t \lee \E \e^{X_t}$.
		\end{enumerate}
	\end{lemm}
	
	\begin{proof}	
		\eqref{item:1:sigma>0} Denote $J: = X - \sigma W$. Let $p^{\sigma W}_t$ be the density of $\sigma W_t$. Then the independence of $\sigma W$ and $J$ implies $p_t = p^{\sigma W}_{t}*\p_{J_t}$ for all $t>0$.  Choosing $\kappa= 1$ yields
		\begin{align*}
			K_t & \lee \int_{\R} |\e^x - 1| \left| \pd_x p_t(x)\right| \od x   \lee \frac{1}{\sigma^2 t} \int_{ \R} \int_{ \R} |\e^x - 1| \left|x-y\right| p^{\sigma W}_{t} (x-y) \p_{J_t}(\od y) \od x \\
			& = \frac{1}{\sigma^2 t} \int_{ \R} \int_{ \R} |\e^{x+y} - 1| \left|x\right| p^{\sigma W}_{t} (x) \od x \p_{J_t}(\od y) = \frac{1}{\sigma^2 t} \E\left|\sigma W_t \left(\e^{\sigma W_t + J_t} -1\right)\right|\\
			&  \lee \frac{1}{\sigma t} \|W_t\|_{L_2(\p)} \|\e^{X_t}-1\|_{L_2(\p)} = \frac{\|\e^{X_t}-1\|_{L_2(\p)}}{\sigma \sqrt{t}}.
		\end{align*}

		\eqref{item:2:unimodal} We may assume that $\E\e^{X_t} <\infty$, otherwise the inequality is obvious. By the monotonicity of $p_t$, one has $p_t(x) \to 0$ as $|x|\to \infty$, and for $x>m_t +1$,   
		\begin{align*}
			\e^x p_t(x) \lee \e^{x} \int_{x-1}^x p_t(u) \od u \lee \e^{x} \int_{x-1}^\infty p_t(u) \od u \to 0 \quad \mbox{as } x \to \infty,
		\end{align*}
		where the limit holds due to $\E \e^{X_t}<\infty$. 
		Now, choosing $\kappa = \e^{m_{t}}$ and using integration by parts, together with $\lim_{|x| \to \infty} \e^x p_t(x) =0$, we have
		\begin{align*}
			K_{t} & \lee \int_{\R} |\e^x - \e^{m_t}| \left|\pd_x p_t(x)\right| \od x\\
			& = \int_{-\infty}^{m_{t}} (\e^{m_{t}} - \e^x) \pd_x p_t(x) \od x + \int_{m_{t}}^\infty (\e^x - \e^{m_{t}}) (- \pd_x p_t(x)) \od x \\
			& = \int_{\R} \e^x p_t(x) \od x = \E \e^{X_t}.\qedhere
		\end{align*}
	\end{proof}
	
	\cref{prop:Holder-estimate} below is an extension of \cite[Theorem 9.14]{GN20} to the exponential L\'evy setting. Because of the weighted setting caused by the exponential L\'evy process, it seems that the interpolation techniques using in the proof of \cite[Theorem 9.14]{GN20} is not applicable here, at least in a straightforward way.	We recall the class $\sS(\alpha)$ of stable-like L\'evy measures  from 	\cref{defi:holde-stable}\eqref{defi:alpha-stable}.

	\begin{prop}\label{prop:Holder-estimate} Let $g\in C^{0, \eta}(\R_+)$ with $\eta \in [0, 1]$. Then, for $T \in (0, \infty)$, there exists a constant  $c_{\eqref{eq:prop:Holder-estimate}}>0$ such that for any $z>0, y>0$ and any $t\in (0, T]$ one has
		\begin{align}\label{eq:prop:Holder-estimate}
			|P_tg(z) - P_{t} g(y)| \lee c_{\eqref{eq:prop:Holder-estimate}}  U_t(y, z),
		\end{align}
		where the cases for $U_t(y, z)$ are provided as follows:
		\begin{enumerate}[\rm (1)] 
			\itemsep0.3em
			
			\item \label{item:1:sigma>0-function-U} If $\sigma >0$ and $\int_{|x|>1} \e^{2x} \nu(\od x) <\infty$, then $U_t(y, z) = \big(t^{\frac{\eta -1}{2}} \frac{|z^\eta - y^\eta|}{\eta}\big) \wedge |z-y|^\eta$.
			\item \label{item:2:sigma=0-function-U} If $\sigma =0$ and $\int_{|x|>1} \e^{x} \nu(\od x) <\infty$ and if $\nu \in \sS(\alpha)$ for some $\alpha \in (0, 2)$, then $U_t(y, z) = \big(t^{\frac{\eta -1}{\alpha}} \frac{|z^\eta - y^\eta|}{\eta}\big)\wedge |z-y|^\eta$.
		\end{enumerate}
		Here, we set $0^0 :=1$ and $\frac{|z^0 - y^0|}{0} : = \lim_{\eta \downarrow 0} \frac{|z^\eta - y^\eta|}{\eta} = |\ln z - \ln y|$ by convention.
	\end{prop}
	
	\begin{proof} For $r \in \R$, since $\e^{-t \psi(-\im r)} = \E \e^{r X_t} <\infty$ for all $t>0$ if and only if $\int_{|x|>1} \e^{rx} \nu(\od x) <\infty$, it follows from the integrability conditions for $\nu$ in items \eqref{item:1:sigma>0-function-U} and \eqref{item:2:sigma=0-function-U} that $C^{0, \eta}(\R_+) \subseteq \scrD_{\exp}$ for any $\eta \in [0, 1]$.  Let $T \in (0, \infty)$. Then, for any $t \in [0, T]$ and $z >0$,  $y >0$, the H\"older continuity of $g$ implies that
		\begin{align}\label{eq:Holder-estimate-general}
			|P_t g(z) - P_t g(y)|  \lee |g|_{C^{0, \eta}(\R_+)} \E \e^{\eta X_t} |z-y|^\eta \lee |g|_{C^{0, \eta}(\R_+)} \e^{T|\psi(-\im \eta)|} |z-y|^\eta.
		\end{align}

		\eqref{item:1:sigma>0-function-U}  Set $J: = X - \sigma W$. Let $p^{\sigma W}_t$ (resp. $p_t$) be the probability density of $\sigma W_t$ (resp. $X_t$). For $t \in (0, T]$, since $p_t = p_t^{\sigma W} * \p_{J_t}$, one has
		\begin{align*}
			\|\pd_x p_t\|_{L_1(\R)} = \| \pd_x p_{t}^{\sigma W} * \p_{J_t}\|_{L_1(\R)} \lee  \| \pd_x p_{t}^{\sigma W}\|_{L_1(\R)} = \sqrt{2/(\pi \sigma^2 t)}.
		\end{align*}
		It is clear that $p_t \in L_1(\R, \Leb_\eta)$, and similar computations as in the proof of \cref{lemm:extra-quantity}\eqref{item:1:sigma>0}  show $\pd_x p_t \in L_1(\R, \Leb_{\eta})$. Hence, the assumptions for $p_t$ required in \cref{thm:upper-bound-general-Holder} are  satisfied.  Since $\e^{-t \psi(-\im)} = \E \e^{X_t} <\infty$ and $\e^{-t \psi(-2\im)} = \E\e^{2X_t} <\infty$ for all $t \in (0, T]$, we get
		\begin{align*}
			\E|\e^{X_t}-1|^2  = \E \e^{2 X_t} - 2 \E \e^{X_t} +1 = \e^{-t \psi(-2 \im)} - 2 \e^{-t \psi(-\im)} +1,
		\end{align*}
		which implies
		\begin{align}\label{eq:constant-sigma>0}
			c_{\eqref{eq:constant-sigma>0}}^2: = \ts\sup_{t \in (0, T]} \(t^{-1}\E |\e^{X_t}-1|^2\) <\infty.
		\end{align}
		Then, for $\eta \in [0, 1]$, $t\in (0, T]$, $z>0$, $y>0$, combining \eqref{eq:pointwise-upper-estimate-Holder-1} with \cref{lemm:extra-quantity}\eqref{item:1:sigma>0} yields
		\begin{align}\label{eq:upper-estimate-sigma>0-1}
			|P_{t}g(z) - P_{t}g(y)| & \lee |g|_{C^{0, \eta}(\R_+)} \|\pd_x p_{t}\|_{L_1(\R)}^{1-\eta} \frac{c_{\eqref{eq:constant-sigma>0}}^\eta}{\sigma^{\eta}}  \frac{|z^\eta - y^\eta|}{\eta} \lee c_{\eqref{eq:upper-estimate-sigma>0-1}}  t^{\frac{\eta-1}{2}} \frac{|z^\eta - y^\eta|}{\eta},
		\end{align}
		where $c_{\eqref{eq:upper-estimate-sigma>0-1}}:=|g|_{C^{0, \eta}(\R_+)} c_{\eqref{eq:constant-sigma>0}}^\eta (1/\sigma) (2/\pi)^{(1-\eta)/2}$. Then \eqref{eq:Holder-estimate-general} and \eqref{eq:upper-estimate-sigma>0-1} imply the assertion.
		
		
		\eqref{item:2:sigma=0-function-U} Let $\nu = \nu_1 + \nu_2$ with $\nu_1, \nu_2$ as in 	\cref{defi:holde-stable}\eqref{defi:alpha-stable}. Assume that $X^1$ and $X^2$ are independent L\'evy processes defined on $(\wt \Omega, \wt \F, \wt \p)$  with characteristics $(0, 0, \nu_1)$ and $(\gamma, 0, \nu_2)$ respectively. Then $X$ and $X^1+X^2$ have the same finite-dimensional distribution.	Since $\int_{|x|>1} \e^x \nu_i(\od x) \lee \int_{|x|>1} \e^x \nu(\od x) <\infty$, $i = 1, 2$, it implies that $\wt \E \e^{X^1_t} <\infty$ and  $\wt \E \e^{X^2_t} <\infty$ for all $t \in (0, T]$.
		
		Because of the conditions imposed on $\nu_1$, it is straightforward to check that \eqref{lemm:eq:lemm-growth-transition-densities} is satisfied for the characteristic exponent of $X^1$. According to \cref{lemm-growth-transition-densities}, $X^1$ has transition densities $(p^1_t)_{t>0} \subset C^\infty_0(\R)$ with $\pd_x^n p^1_t \in \cap_{1\lee s \lee \infty} L_s(\R)$ for all $n\gee 0$, $t \in (0, T]$ and there is a constant $c_{\eqref{eq:growth-derivative-densities}}>0$ such that
		\begin{align}\label{eq:growth-derivative-densities}
			\|\pd_x p^1_t\|_{L_1(\R)} \lee c_{\eqref{eq:growth-derivative-densities}} t^{-\frac{1}{\alpha}}, \quad t \in (0, T].
		\end{align}
		Since $X^1$ is selfdecomposable (see  \cite[Sec.53]{Sa13}), applying \cite[Theorem 53.1]{Sa13} yields that $\wt \p_{X^1_t}$ is unimodal for all $t\in (0, T]$. Let $m_{t}$ be a mode of $\wt \p_{X^1_{t}}$ so that the density $p^1_{t}$ of $X^1_t$ is non-decreasing on $(-\infty, m_{t})$ and non-increasing on $(m_{t}, \infty)$.  \cref{lemm:extra-quantity}\eqref{item:2:unimodal} gives
		\begin{align*}
			\inf_{\kappa>0}\int_{\R} |\e^x - \kappa| |\pd_x p^1_t(x)| \od x \lee \wt \E \e^{X^1_t}, \quad t \in (0, T].
		\end{align*}
		A similar argument as in the proof of \cref{lemm:extra-quantity}\eqref{item:2:unimodal} yields $\pd_x p^1_t \in L_1(\R, \Leb_{\eta})$. Hence, for $t \in (0, T]$ and $z>0$, $y>0$, exploiting the independence of $X^1$ and $X^2$, together with \cref{thm:upper-bound-general-Holder}, we get
		\begin{align}\label{eq:upper-estimate-sigma=0-1}
			& |P_{t}g(z) - P_{t}g(y)|  = |\E g(z \e^{X_{t}}) - \E g(y \e^{X_{t}})| = |\wt \E g(z \e^{X^2_t}\e^{X^1_{t}}) - \wt \E g(y \e^{X^2_t} \e^{X^1_{t}})| \notag \\
			& \lee \wt \E\bigg[\(|g|_{C^{0, \eta}(\R_+)}  \|\pd_x p^1_{t}\|_{L_1(\R)}^{1-\eta}  \inf_{\kappa>0} \left|\int_{\R} |\e^x - \kappa| |\pd_x p^1_{t}(x)| \od x\right|^\eta\)\frac{|(z\e^{X^2_t})^\eta - (y\e^{X^2_t})^\eta|}{\eta}\bigg] \notag \\
			& \lee \(|g|_{C^{0, \eta}(\R_+)}  \|\pd_x p^1_{t}\|_{L_1(\R)}^{1-\eta} |\wt \E \e^{X^1_t}|^\eta\)   \frac{|z^\eta - y^\eta|}{\eta} \wt \E \e^{\eta X^2_t} \notag \\
			& \lee |\E\e^{X_t}|^\eta |g|_{C^{0, \eta}(\R_+)}  c_{\eqref{eq:growth-derivative-densities}}^{1- \eta} t^{\frac{\eta-1}{\alpha}} \frac{|z^\eta - y^\eta|}{\eta} \notag \\
			& \lee c_{\eqref{eq:upper-estimate-sigma=0-1}} t^{\frac{\eta-1}{\alpha}} \frac{|z^\eta - y^\eta|}{\eta},
		\end{align}
		where $c_{\eqref{eq:upper-estimate-sigma=0-1}} := \e^{\eta T |\psi(-\im)|} |g|_{C^{0, \eta}(\R_+)} c_{\eqref{eq:growth-derivative-densities}}^{1- \eta}$. Combining \eqref{eq:upper-estimate-sigma=0-1} with \eqref{eq:Holder-estimate-general} yields the assertion.
	\end{proof}

	\subsection{Estimate for the gradient in the GKW decomposition}
	
	Motivated by the formula \eqref{eq:integrand-general}, for a L\'evy measure $\ell$  and a Borel function $g$, let us write symbolically 
	\begin{align}\label{eq:varGamma-function}
		\varGamma_\ell(t, y) := \sigma^2 \pd_y P_{t}g(y) + \int_{ \R} \frac{P_{t} g(\e^x y) - P_{t} g(y)}{y}(\e^x -1)\ell(\od x), \quad t>0, y >0,
	\end{align}
	where we set $\pd_y P_{t}g(y):=0$ if $\sigma =0$. Although we choose $\ell = \nu$ for \eqref{eq:integrand-general}, it is useful to consider the general $\ell$ because it might have applications in other contexts (e.g., see \cite{TN20b}).

	\cref{theo:envelop-MVH-strategy}\eqref{item:2.2:sigma=0-envelope} below is a variant of \cite[Theorem 9.18]{GN20} in the exponential L\'evy setting. Here, the exponent of the time variable $t$ in the estimates we obtain is the same as in \cite[Theorem 9.18]{GN20}. Again, we recall  $\sS(\alpha)$ from \cref{defi:holde-stable}\eqref{defi:alpha-stable}.

	\begin{prop}\label{theo:envelop-MVH-strategy} Let $\ell$ be a L\'evy measure  and $g\in C^{0, \eta}(\R_+)$ with $\eta \in [0, 1]$. Assume that $\int_{|x|>1} \e^{(\eta +1)x} \ell(\od x) <\infty$. Then, for any $T\in (0, \infty)$ there is a constant  $c_{\eqref{eq:thm:Holder-case-estimate:psi-process}}>0$ such that
		\begin{align}\label{eq:thm:Holder-case-estimate:psi-process}
			|\varGamma_\ell(t, y)| \lee c_{\eqref{eq:thm:Holder-case-estimate:psi-process}}  V(t) y^{\eta -1}, \quad \forall (t, y) \in (0, T]\times \R_+,
		\end{align}
		where the cases for $V(t)$ are provided as follows:
		\begin{enumerate}[\rm (1)]
			\itemsep0.3em
			
			\item \label{item:1:sigma>0-envelope} If $\sigma >0$ and $\int_{|x|>1} \e^{2x} \nu(\od x) <\infty$, then $V(t) = t^{\frac{\eta -1}{2}}$.
			
			\item \label{item:2:sigma=0-envelope} If $\sigma =0$, $\int_{|x|>1}\e^{\eta x} \nu(\od x) <\infty$ and $\int_{|x| \lee 1} |x|^{\eta +1} \ell(\od x) <\infty$, then $V(t)=1$.
			
			\item \label{item:2.2:sigma=0-envelope} If $\sigma =0$, $\eta \in [0, 1)$ and if the following two conditions hold:
			\begin{enumerate}[\rm (a)]
				\item \label{item:3a} $\nu \in \sS(\alpha)$ for some $\alpha \in (0, 2)$ and $\int_{|x|>1} \e^x \nu(\od x) <\infty$,
				\item \label{item:3b} there is a $\beta \in [0, 2]$ such that 
				\begin{align}\label{eq:small-ball-condition-ell}
					\ts 0 < c_{\eqref{eq:small-ball-condition-ell}} : = \sup_{r \in (0, 1)} r^\beta \int_{r < |x| \lee 1} \ell(\od x) < \infty,
				\end{align}
			\end{enumerate}
			then
			\begin{align*}
				V(t) = \begin{cases}
					t^{\frac{\eta +1 -\beta}{\alpha}} & \mbox{ if } \beta \in (1 + \eta, 2]\\
					\max\{1, \log(1/t)\} & \mbox{ if } \beta = 1+ \eta\\
				1  & \mbox{ if } \beta \in [0, 1+ \eta).
				\end{cases}
			\end{align*}
		\end{enumerate}
		Here, the constant $c_{\eqref{eq:thm:Holder-case-estimate:psi-process}}$ may depend on $\beta$ in Item \eqref{item:2.2:sigma=0-envelope}.
	\end{prop}
	
	\begin{rema}\label{rema:sufficient-ell}
		\begin{enumerate}[\rm (1)]
			\item \label{item:1:rema:sufficient-ell} If $\eta \in [0, 1)$ and $\beta \in (1+ \eta, 2]$, then the condition \eqref{eq:small-ball-condition-ell} is equivalent to 
			\begin{align}\label{eq:small-ball-condition-ell-1}
				 0 <  c_{\eqref{eq:small-ball-condition-ell-1}} : = \sup_{ r \in (0, 1]} r^{\beta} \int_{|x| \lee 1} \Big(\Big|\frac{x}{r}\Big|^2 \wedge \Big|\frac{x}{r}\Big|^{\eta +1}\Big) \ell (\od x) <\infty.
			\end{align}
			Indeed, we set $\delta: = \eta +1$ so that $\beta \in (\delta, 2]$. If $\beta \in (\delta, 2)$, then applying \cref{lem:small-ball-property} with $\mu = \ell$ we get, for any $r \in (0, 1]$,
			\begin{align*}
				A_{\beta,  \delta}(r) & : = r^{\beta} \int_{|x| \lee 1} \Big(\left|\frac{x}{r}\right|^2 \wedge \left|\frac{x}{r}\right|^{ \delta}\Big) \ell (\od x)  = r^{\beta -2} \int_{|x| \lee r} x^2 \ell(\od x) + r^{\beta -  \delta} \int_{r< |x| \lee 1} |x|^ \delta \ell(\od x)\\
				& \lee c_{\eqref{eq:small-ball-condition-ell}} \( \frac{4}{1 - 2^{\beta -2}} + \frac{2^{2\beta -  \delta}}{2^{\beta -  \delta}-1}\)
			\end{align*}
			which implies $\sup_{r \in (0, 1]} A_{\beta,  \delta}(r) <\infty$. If $\beta =2$, then 
			\begin{align*}
				A_{2,  \delta}(r)  \lee \int_{|x| \lee 1} x^2 \ell(\od x) + r^{2 -  \delta} \int_{r< |x| \lee 1} |x|^ \delta \ell(\od x) \lee \int_{|x| \lee 1} x^2 \ell(\od x) + c_{\eqref{eq:small-ball-condition-ell}} \frac{2^{4 -  \delta}}{2^{2 -  \delta}-1},
			\end{align*}
			which yields to $\sup_{r \in (0, 1]} A_{2,  \delta}(r) <\infty$. For the converse implication, it is clear that
			\begin{align*}
				r^\beta \int_{r < |x| \lee 1} \ell(\od x) \lee r^{\beta -  \delta} \int_{r < |x| \lee 1} |x|^ \delta \ell(\od x) \lee A_{\beta,  \delta}(r),
			\end{align*}
			which implies that $c_{\eqref{eq:small-ball-condition-ell}} \lee c_{\eqref{eq:small-ball-condition-ell-1}}$.
			
			\item \label{item:2:rema:sufficient-ell}  If $\beta = 1 + \eta \in [1, 2)$, then applying \cref{lem:small-ball-property} as before yields a $c_\beta>0$ such that
			\begin{align*}
				A_{\beta, \beta}(r) = r^{\beta} \int_{|x| \lee 1} \Big(\left|\frac{x}{r}\right|^2 \wedge \left|\frac{x}{r}\right|^{\beta}\Big) \ell (\od x) \lee c_\beta c_{\eqref{eq:small-ball-condition-ell}} \max\{1, \log\tfrac{1}{r}\}, \quad \forall r \in (0, 1].
			\end{align*}
		\end{enumerate}
	\end{rema}

	\begin{proof}[Proof of \cref{theo:envelop-MVH-strategy}]	
		We use the following inequality without mentioning it again:
		\begin{align*}
			\frac{|\e^{\eta x} - 1|}{\eta} \lee \e^{\eta} |x|, \quad \forall |x| \lee 1, \eta \in [0, 1],
		\end{align*}
		where $\frac{|\e^{0x} -1|}{0} : = \lim_{\eta \downarrow 0} \frac{|\e^{\eta x} -1|}{\eta} = |x|$. Let us fix $T \in (0, \infty)$.
		
		\eqref{item:1:sigma>0-function-U} Since $\sigma>0$ and $\int_{|x|>1} \e^{2x} \nu(\od x) <\infty$,  \cref{prop:Holder-estimate}\eqref{item:1:sigma>0-function-U} implies that 
		\begin{align}\label{eq:case-C1}
			|P_{t} g(z)- P_{t}g(y)|\lee c_{\eqref{eq:prop:Holder-estimate}} \(\(t^{\frac{\eta-1}{2}} \frac{|z^\eta -y^\eta|}{\eta}\)\wedge |z-y|^\eta\)
		\end{align}
		for all $z>0$, $y>0$, $t\in (0, T]$. Moreover, since $P_{t}g \in C^\infty(\R_+)$ due to $\sigma >0$, we divide both side of \eqref{eq:case-C1} by $|z-y|$ and then let $z \to y$ to obtain that 
		\begin{align*}
			|\pd_y P_{t}g(y)| \lee c_{\eqref{eq:prop:Holder-estimate}} t^{\frac{\eta-1}{2}} y^{\eta -1},\quad \forall (t, y) \in (0, T] \times \R_+.
		\end{align*}
		Hence, we separate $\int_{ \R} = \int_{|x| \lee 1} + \int_{|x|>1}$ and apply \eqref{eq:case-C1} with $z = y \e^x$ to obtain
		\begin{align}\label{eq:Holder-estimate-sigma>0}
			& |\varGamma_\ell(t, y)| \notag \\
			& \lee c_{\eqref{eq:prop:Holder-estimate}}\bigg[\sigma^2 + \int_{|x|\lee 1}\frac{|\e^{\eta x} -1|}{\eta} |\e^x -1| \ell(\od x)\bigg] t^{\frac{\eta -1}{2}} y^{\eta -1}  + c_{\eqref{eq:prop:Holder-estimate}} y^{\eta -1}\int_{|x|>1} |\e^x -1|^{\eta +1}\ell(\od x).
		\end{align}
		Since $0 < \sigma^2 + \int_{|x|\lee  1}\frac{|\e^{\eta x} -1|}{\eta} |\e^x -1| \ell(\od x) \lee \sigma^2+  \e^{\eta +1}\int_{|x| \lee 1} |x|^2 \ell(\od x) <\infty$ and $\int_{|x|>1} |\e^x -1|^{\eta +1} \ell(\od x) <\infty$, 
		together with $\inf_{t \in (0, T]} t^{\frac{\eta -1}{2}} >0$, the second term on the right-hand side of \eqref{eq:Holder-estimate-sigma>0} can be upper bounded by the first term up to a positive multiplicative constant. Hence, the desired conclusion follows.
		
		\medskip
		
		\eqref{item:2:sigma=0-envelope} One has $\e^{-t \psi(-\eta \im)} =\E \e^{\eta X_t} <\infty$ for $t>0$. The H\"older continuity of $g$ implies that $|P_{t}g(\e^x y) - P_{t}g(y)| \lee |g|_{C^{0, \eta}(\R_+)}\E \e^{\eta X_{t}} |\e^x -1|^\eta y^\eta$, and hence
		\begin{align*}
			|\varGamma_\ell(t, y)| & \lee |g|_{C^{0, \eta}(\R_+)}\E \e^{\eta X_{t}} y^{\eta -1} \int_{\R} |\e^x -1|^{\eta +1} \ell(\od x)\\
			& \lee |g|_{C^{0, \eta}(\R_+)} \e^{T|\psi(-\eta \im)|}\bigg[\e^{\eta +1}\int_{|x| \lee 1} |x|^{\eta +1} \ell(\od x) + \int_{|x| >1} |\e^x -1|^{\eta +1} \ell(\od x)\bigg] y^{\eta -1},
		\end{align*}
		which implies the assertion.
		
		\medskip
		
		\eqref{item:2.2:sigma=0-envelope} If $\beta \in [0, 1+ \eta)$, then $\int_{|x| \lee 1} |x|^{1 + \eta} \ell(\od x) <\infty$ due to \cref{lem:small-ball-property}, and hence, we can apply Item \eqref{item:2:sigma=0-envelope} to get $V(t) =1$. We now consider $\beta \in [1+ \eta, 2]$. Let $t \in (0, T]$ and $y>0$. We separate $\int_{\R} = \int_{|x| \lee 1} + \int_{|x| > 1}$, and then apply \cref{prop:Holder-estimate}\eqref{item:2:sigma=0-function-U} with $z = y\e^x$ to obtain
		\begin{align}\label{eq:estimate-Holder-C3}
			& |\varGamma_\ell(t, y)| \notag \\
			& \lee  c_{\eqref{eq:prop:Holder-estimate}} y^{\eta-1} \bigg[\int_{|x| \lee 1} \(\(t^{\frac{\eta -1}{\alpha}} \frac{|\e^{\eta x}-1|}{\eta}\) \wedge |\e^x -1|^\eta\) |\e^x -1| \ell(\od x) + \int_{|x|> 1} |\e^x -1|^{\eta+1} \ell(\od x)\bigg] \notag\\
			& \lee c_{\eqref{eq:prop:Holder-estimate}}  y^{\eta-1}\bigg[\e^{\eta +1} t^{\frac{\eta +1}{\alpha}} \int_{|x| \lee 1} \(\left|\frac{x}{t^{1/\alpha}}\right|^2 \wedge \left|\frac{x}{t^{1/\alpha}}\right|^{\eta +1}\) \ell(\od x)  + \int_{|x| >1} |\e^x-1|^{\eta +1} \ell(\od x)\bigg] \notag\\
			& \lee c_{\eqref{eq:prop:Holder-estimate}}  y^{\eta -1} \bigg[c_{\eqref{eq:estimate-Holder-C3}} V(t) +  \int_{|x| >1} |\e^x-1|^{\eta +1} \ell(\od x)\bigg],
		\end{align}
		where, thanks to \eqref{eq:small-ball-condition-ell} and \cref{rema:sufficient-ell},  $V(t) = \max\{1, \log(1/t)\}$ if $\beta = 1+ \eta$ and $V(t) = t^{\frac{1+ \eta - \beta}{\alpha}}$ if $\beta  \in (1+ \eta, 2]$, and  
		$c_{\eqref{eq:estimate-Holder-C3}} >0$ is a constant independent of $t$. 
		Since $\inf_{t \in (0, T]} V(t) >0$, the desired conclusion follows from \eqref{eq:estimate-Holder-C3}.
	\end{proof}

	\subsection*{Acknowledgment} A major part of this work was done when the author was affiliated to the Department of Mathematics and Statistics, University of Jyv\"askyl\"a, Finland. The author is very grateful to Christel Geiss and Stefan Geiss for helpful discussions.

\bibliographystyle{amsplain}

\end{document}